\documentclass[11pt]{amsart}%
\usepackage{amsfonts}
\usepackage{epsfig}
\usepackage{graphicx}  
\usepackage{amsmath}
\usepackage{amsfonts}
\usepackage{amssymb}
\usepackage{latexsym}
\usepackage{rotating}
\usepackage{pstricks, pst-node, pst-text, pst-3d}
\usepackage{amsbsy}
\usepackage{bm}
\usepackage{mathrsfs}
\usepackage{tikz}
\newcommand*\circled[1]{\tikz[baseline=(char.base)]{
            \node[shape=circle,draw,inner sep=2pt] (char) {#1};}}

\usepackage{lscape}

\input{amssym.def}
\newtheorem{theorem}{Theorem}
\newtheorem{proposition}{Proposition}
\newtheorem{corollary}{Corollary}
\newtheorem{lemma}{Lemma}
\newtheorem{remark}{Remark}

\newtheorem{definition}{Definition}
\theoremstyle{remark}

\newcommand{\spn}{\text{\rm span}\,}

\newcommand{\Cl}{\text{\rm Cl}}
\newcommand{\End}{\text{\rm End}}
\newcommand{\Id}{\text{\rm Id}}

\newcommand{\R}{\mathbb{R}}
\newcommand{\Q}{\mathbb{Q}}

\input epsf
\usepackage{geometry}
\date{\today}
\begin{document}
\title[Lattice on general $H$-type groups]{Existence of the lattice on general $H$-type groups}
\author[K.~Furutani, I.~Markina ]{Kenro Furutani, Irina Markina}

\thanks{The first author has been partially
supported by 
the Grant-in-aid for
Scientific Research (C) No. 23540251, 
{\it Japan Society for the Promotion of Science}.\\ 
\quad The second author has been partially supported by the grants of the 
Norwegian Research Council \#204726/V30. This work was mostly done when the second author was
visiting Tokyo University of Science (TUS) 
under a support by a grand from the TUS}
 


\address{K.~Furutani:  Department of Mathematics, Faculty of Science 
and Technology, Science University of Tokyo, 2641 Yamazaki, Noda, Chiba (278-8510), Japan}
\email{furutani\_kenro@ma.noda.tus.ac.jp}

\address{I.~Markina: Department of Mathematics, University of Bergen, P.O.~Box 7803,
Bergen N-5020, Norway}
\email{irina.markina@uib.no}

\begin{abstract}
Let $\mathscr N$ be a two step nilpotent Lie algebra endowed with
non-degenerate scalar product 
$\langle\cdot\,,\cdot\rangle$ 
and let $\mathscr N=V\oplus_{\perp}Z$, 
where $Z$ is the center of the Lie algebra 
and $V$ its orthogonal complement with respect to the scalar product. 
We prove that if $(V,\langle\cdot\,,\cdot\rangle_V)$ is the Clifford
module for the Clifford algebra $\Cl(Z,\langle\cdot\,,\cdot\rangle_Z)$ 
such that the homomorphism 
$J\colon \Cl(Z,\langle\cdot\,,\cdot\rangle_Z)\to\End(V)$ 
is skew symmetric with respect to the scalar product 
$\langle\cdot\,,\cdot\rangle_V$, or in other words the Lie algebra 
$\mathscr N$ satisfies conditions of general $H$-type Lie algebras
~\cite{Ciatti, GKM}, then there is a basis with respect to which the 
structural constants of the Lie algebra $\mathscr N$ are all $\pm 1$ or $0$.
\end{abstract}

\maketitle


\section{Introduction and definitions}


We denote by $\langle\cdot\,,\cdot\rangle_{V}$ a real valued 
symmetric non-degenerate bi-linear
form defined on a real vector space $V$ and call it a scalar product. 
If the form is positive definite, we denote it by $(\cdot\,,\cdot)$
and call an inner product. We use 
the notation $[\cdot\,,\cdot]$ for commutator, 
or in other words for a skew symmetric bi-linear vector valued form. 
The $H$-type Lie algebras were
introduced by A.~Kaplan in~\cite{Kap} and were widely studied, see,
for instance~\cite{CCM,CDKR, ChMar,Kap2,Kor}.
In works~\cite{Ciatti,CP,GKM} the analogous of classical $H$-type Lie 
algebras were introduced and studied. 
\begin{definition}\label{def:general}
A 2-step nilpotent Lie algebra
$\mathscr N$ 
endowed with a scalar product $\langle\cdot\,,\cdot\rangle$ is called
a general $H$-type algebra, if 
\begin{itemize}
\item[1.]{$\mathscr N=V\oplus_{\perp}Z$, 
where $Z$ is the center of the Lie algebra $\mathscr N$, 
which is non-degenerate subspace of the scalar product 
vector space $(\mathscr N,\langle\cdot\,,\cdot\rangle)$, 
and $V$ its orthogonal complement $($we call it the horizontal space$)$,}
\item[2.]{the skew symmetric map $J\colon Z\to\End (V)$ defined by
\begin{equation}\label{eq:J}
\langle J_zu,v\rangle=\langle z, [u,v]\rangle
\end{equation}
satisfies the orthogonality condition
\begin{equation}\label{eq:J_orth}
\langle J_zu,J_zv\rangle=\langle z,z\rangle \langle u, v\rangle.
\end{equation}
}
\end{itemize}
\end{definition} 
Conditions~\ref{eq:J} and~\ref{eq:J_orth} imply
\begin{equation}\label{eq:J_Clif}
J^2_z=-\langle z,z\rangle \Id_{V}
\end{equation}
see, for example,~\cite{Ciatti, GKM,Lam}.

Due to~\eqref{eq:J_Clif} the horizontal space $V$ becomes a
$\Cl(Z,\langle\cdot\,,\cdot\rangle_Z)$-module, where
$\langle\cdot\,,\cdot\rangle_Z$ is the restriction of the scalar
product $\langle\cdot\,,\cdot\rangle$ onto the center~$Z$. So, from the
definition we see that any general $H$-type algebra $\mathscr N$
defines a $\Cl(Z,\langle\cdot\,,\cdot\rangle_Z)$-module $V$. Moreover,
the module $V$ is endowed with the scalar product
$\langle\cdot\,,\cdot\rangle_V$, obtained by the restriction of
$\langle\cdot\,,\cdot\rangle$ on $V$, such that the representations
$J_z$ are skew symmetric with respect to the scalar product
$\langle\cdot\,,\cdot\rangle_V$ for any $z\in Z$.
\smallskip

From the other side, if we assume that $V$ is a Clifford module 
for some Clifford algebra $\Cl(Z,\langle\cdot\,,\cdot\rangle_Z)$, 
and $V$ carried a scalar product $\langle\cdot\,,\cdot\rangle_V$ 
such that~\eqref{eq:J_orth} holds, then $J$ is skew symmetric 
with respect to $\langle\cdot\,,\cdot\rangle_V$:
\begin{equation}\label{eq:J_skew}
\langle J_zu,v\rangle_V = - \langle u, J_zv\rangle_V.
\end{equation}
{\it
Therefore, one can define the Lie bracket $[\cdot\,,\cdot]\colon
V\times V\to Z$ 
by using~\eqref{eq:J} with
$\langle\cdot\,,\cdot\rangle=\langle\cdot\,,\cdot\rangle_V+\langle\cdot\,,\cdot\rangle_Z$ 
and show that the Lie algebra $\mathscr N=(V\oplus_{\perp}
Z,[\cdot\,,\cdot],\langle\cdot\,,\cdot\rangle)$ 
is a general $H$-type algebra, see~\cite{Ciatti, CDKR, GKM, Kap2}.}

In general, among the conditions~\eqref{eq:J_orth},~\eqref{eq:J_Clif}, 
and~\eqref{eq:J_skew} any two of them imply the third one.

We say that a $\Cl(Z,\langle\cdot\,,\cdot\rangle_Z)$-module $V$ is an
{\it  admissible module}, if there is a scalar product 
$\langle \cdot\,,\cdot\rangle_V$ defined
on $V$ such that the representations $J_z\colon \Cl(Z,\langle\cdot\,,\cdot\rangle_Z)\to\End(V)$, satisfy 
the skew symmetry condition~\eqref{eq:J_skew} 
for any $z\in Z$. The scalar product $\langle \cdot\,,\cdot\rangle_V$ 
will be called an {\it admissible scalar product}.
The following is known about admissible modules. 
Let $V$ be a given $\Cl(Z,\langle\cdot\,,\cdot\rangle_Z)$-module 
and $\{z_1,\ldots, z_n\}$ an orthonormal basis of $Z$ with respect to 
the scalar product~$\langle\cdot\,,\cdot\rangle_{Z}$. Let us denote by 
$J_{z_j}\in \End(V)$ the representations of the generators 
$z_1,\ldots, z_n$ of the Clifford algebra $\Cl(Z,\langle\cdot\,,\cdot\rangle_Z)$. Then $J_{z_j}$ satisfy
$$
J_{z_j}^2=-\langle z_j, z_j\rangle_{Z}\Id_{V},\quad J_{z_i}J_{z_j}=-J_{z_j}J_{z_i}, \quad i,j=1,\ldots,n,\ \ i\neq j.
$$
It is always possible to find an 
inner product $(\cdot\,,\cdot)_V$ on $V$ such that, 
the representations $J_{z_j}$ satisfy the orthogonality 
condition~\eqref{eq:J_orth}, 
since the group generated by the
operators $\{J_{z_j}\}_{j=1}^{n}$ is a finite group in $\End(V)$ (see~\cite{Hus}). 
In the special case of Clifford algebra  $\Cl(Z,(\cdot\,,\cdot)_Z)$,
generated by an inner product space $(Z,(\cdot\,,\cdot)_Z)$
the chosen inner product $(\cdot\,,\cdot)_V$ 
on $\Cl(Z,(\cdot\,,\cdot)_Z)$-module $V$ satisfies the orthogonality
condition~\eqref{eq:J_orth} 
for any $J_z$, $z\in Z$ and not only for representations $J_{z_j}$ 
of generators $z_j$, $j=1,\ldots,n$, see~\cite{LawMich}. As a consequence, in this case 
we obtain the skew symmetry property~\eqref{eq:J_skew} 
and $\bigr(V,(\cdot\,,\cdot)_V\bigr)$ became an admissible module. 
The Lie algebra $\mathscr N=V\oplus_{\perp}Z$, where the Lie bracket
is defined in~\eqref{eq:J} by making use of skew symmetric maps $J_z$ and
the inner product on $\mathscr N$ is the sum of inner products on $V$
and $Z$, is the $H$-type algebra introduced by A.~Kaplan
in~\cite{Kap}. We call such an algebra a {\it classical} $H$-type Lie algebra.

Let $\Cl(Z,\langle\cdot\,,\cdot\rangle_Z)$ be a Clifford algebra
generated 
by an arbitrary scalar product~$\langle\cdot\,,\cdot\rangle_{Z}$. It
was shown in~\cite{Ciatti} that given a
$\Cl(Z,\langle\cdot\,,\cdot\rangle_Z)$-module $V$ there always exists
a scalar product $\langle\cdot\,,\cdot\rangle_V$ on $V$ (or on
$V\oplus V$), such that the map $J_z\colon Z\to\End(V)$ \big(or a
modified map $\widetilde J_z\colon Z\to \End(V\oplus V)$\big)
satisfies~\eqref{eq:J_orth}, or equivalently~\eqref{eq:J_skew}, for an
arbitrary $z\in Z$. As a consequence, we obtain that for any Clifford
algebra $\Cl(Z,\langle\cdot\,,\cdot\rangle_Z)$ there exists an
admissible module $(V,\langle\cdot\,,\cdot\rangle_V)$ (or $(V\oplus
V,\langle\cdot\,,\cdot\rangle_{V\oplus V})$). Moreover, the admissible
module $(V,\langle\cdot\,,\cdot\rangle_V)$ (or $(V\oplus
V,\langle\cdot\,,\cdot\rangle_{V\oplus V})$) will be necessarily a neutral space, that is the dimensions of
the maximal positive and negative definite subspace of $V$ (or $V\oplus V$) coincide, see
Proposition~\ref{prop:1} and~\cite{Ciatti}. The corresponding 2-step
nilpotent 
Lie algebra satisfies Definition~\ref{def:general} and is called general $H$-type algebra.

Now, let $N$ be a simply connected, nilpotent Lie group and $\Gamma$ its
discrete subgroup such that the quotient space $\Gamma\backslash N$ is
compact. Then the group $\Gamma$ is called {\it lattice} and quotient
$\Gamma\backslash N$ is called a {\it compact nilmanifold}, see, for
instance,~\cite{Eberline1}. Nilmanifolds, as a generalization of
higher dimensional tori, play important role in study
of the sub-Riemannian geometry, the Riemannian geometry with
singularities, hypoelliptic operators, and spectral properties of
differential operators of the Grushin type, see for example~
\cite{BF-2,BFI-1,BFI-2,BFI-3,GW,II}. 
Also see~\cite{BF-1,CL}
for such study of the compact sub-Riemannian manifolds
coming from simple Lie groups.
According to the Mal'cev criterium~\cite{Malcev} 
a nilpotent Lie group $N$ 
admits a lattice if and only if the corresponding Lie algebra 
of $N$ has a basis with rational structure constants. 
Not all, even 2-step, nilpotent Lie algebras admit such a basis. 
In the work~\cite{CrDod} it was shown that classical 
$H$-type Lie algebras $\mathscr N$ have integer structure constants,
or more precisely, 
there is a basis
$\{v_1,\ldots,v_m\}$ of $V$ and a basis $\{z_1,\ldots,z_n\}$ of $Z$, such 
that $[v_\alpha,v_{\beta}]=\sum_{k=1}^{n}A_{\alpha\beta}^kz_k$, 
where the numbers $A_{\alpha\beta}^k$ equal $0,1$, or $-1$. 
So it will be natural to ask whether general H-type algebras have such
a basis too.

In the present work we show the following statement.
\begin{theorem}\label{th:integral_structure}
Let $\big(\mathscr
N=V\oplus_{\bot}Z,[\cdot\,,\cdot]\langle\cdot\,,\cdot\rangle\big)$ 
be a general $H$-type algebra. 
Then there is an orthonormal basis $\{v_1,\ldots,v_m,z_1,\ldots,z_n\}$ of
$V\oplus_{\bot}Z$ 
such that $[v_\alpha,v_{\beta}]=\sum_{k=1}^{n}A_{\alpha\beta}^kz_k$, 
where the coefficients $A_{\alpha\beta}^k$ are equal to $\pm 1$, or $0$.
\end{theorem}

Denote by $\langle\cdot\,,\cdot\rangle_V$ and
$\langle\cdot\,,\cdot\rangle_Z$ 
the restrictions of the scalar product to the subspaces $V$ and $Z$, and assume that the scalar product spaces $(V,\langle\cdot\,,\cdot\rangle_V)$ and $(Z,\langle\cdot\,,\cdot\rangle_Z)$ are nondegenerate.
Let 
$$
\nu_{\alpha}^V= \langle v_\alpha,v_\alpha\rangle_V,\ \
\alpha=1,\ldots,m,\qquad  \nu_k^{Z}
=\langle z_k,z_k\rangle_Z,\ \ k=1,\ldots,n
$$ 
be corresponding indices. Let $J_{z_k}\colon
\Cl(Z,\langle\cdot\,,\cdot\rangle_Z)\to \End(V)$ 
be representations of orthonormal generators $z_1,\ldots,z_n$ of $\Cl(Z,\langle\cdot\,,\cdot\rangle_Z)$. We write 
\begin{equation}\label{eq:str_con}
[v_\alpha,v_{\beta}]=\sum_{k=1}^{n}A_{\alpha\beta}^kz_k\quad\text{and}\quad J_{z_k}v_{\alpha}=\sum_{\beta=1}^{m}B_{\alpha\beta}^{k}v_{\beta}.
\end{equation}
Then, as a consequence of~\eqref{eq:J} and~\eqref{eq:str_con}, we obtain
\begin{equation}\label{eq:AB}
\langle J_{z_k}v_\alpha,v_\beta\rangle_V=\langle
z_k,[v_\alpha,v_\beta]\rangle_Z
\quad\Longrightarrow\quad\nu_{\beta}^VB^k_{\alpha\beta}=\nu_k^ZA^k_{\alpha\beta}.
\end{equation}

Therefore the result of Theorem~\ref{th:integral_structure} can be reformulated as follows.

\begin{theorem}\label{th:main}
Given a scalar product space $(Z,\langle\cdot\,,\cdot\rangle_Z)$ with
an orthonormal basis 
$z_1,\ldots,z_n$ there is an admissible Clifford
$\Cl(Z,\langle\cdot\,,\cdot\rangle_Z)$-module 
$(V,\langle\cdot\,,\cdot\rangle_V)$ of minimal dimension with representations
\[
J\colon \Cl(Z,\langle\cdot\,,\cdot\rangle_Z)\to\End(V) 
\]
and an orthonormal basis $v_1,\ldots,v_m$ on $(V,\langle\cdot\,,\cdot\rangle_V)$, 
such that $J_z$ satisfies~\eqref{eq:J_orth},~\eqref{eq:J_Clif} and moreover,
\begin{equation}\label{eq:integral}
\langle J_{z_k} v_{\alpha},v_{\beta}\rangle_V=\pm 1,\text{ or
 }0,\quad\text{for all}
\quad k=1,\ldots,n,\ \ \alpha,\beta=1,\ldots,m.
\end{equation}
\end{theorem}

In the following we always use the identification
$\Cl_{r,s}\cong\Cl(Z,\langle\cdot\,,\cdot\rangle_Z)$, arising from the isomorphism
$(Z,\langle\cdot\,,\cdot\rangle_Z)\cong\mathbb \R^{r,s}$. Here
$\R^{r,s}$ is the space $\R^{r+s}$ with the quadratic
form  
$\Q_{r,z}(x) = x_1^2+\ldots+x_r^2-x_{r+1}^2-\ldots-x_{r+s}^2$.

We call $\Cl_{r,s}$-modules, satisfying Theorem~\ref{th:main}
{\it admissible integral modules} and the corresponding orthonormal basis $\{v_{\alpha}\}$ {\it integral basis}. The existence of admissible integral
$\Cl_{r,0}$-modules was shown in~\cite{CrDod} and here we reconstruct $\Cl_{r,0}$-modules with an integral basis by making use of a different method, see also an observation in Section~\ref{fremark}. 
The admissible integral $\Cl_{r,0}$-modules lead to the presents of a lattice on classical $H$-type groups. 
Notice also, that in the work~\cite{Eberline1}, the existence of a rational structure constants on
classical $H$-type algebras was shown by realizing its Lie algebra 
as a direct sum of the space $\mathbb R^m$ and the center $Z$, given as a
{\it Lie triple  system} 
embedded in a subspace of $\mathfrak{so}(m)$.

In the present work we construct explicitly an orthonormal basis of 
any minimal dimensional admissible $\Cl_{r,s}$-module with respect to
which the structure constants $A_{\alpha\beta}^k$ defined 
in~\eqref{eq:str_con}, or equivalently in~\eqref{eq:AB}, equal to $\pm1$ or $0$.

There are several methods of construction of such an integral basis. To be able to use the Bott $8$-periodicity the  number of necessary admissible integral modules is 64. We use the isomorphism $\Cl_{r,s+1}\cong\Cl_{s,r+1}$ that preserves the integral basis, we take the tensor products with $\Cl_{4,4}$-module or with $\Cl_{1,1}$-module, the construction of $\Cl_{r,1}$-module from $\Cl_{r,0}$-module and reduce the number of required modules to $28$. To explain the main idea for the construction of the integral basis in the remaining cases we recall the terminology. A vector $v\in V$ is called 
\begin{itemize}
\item{ \it spacelike} if $\langle v,v\rangle_V>0$ or $v=0$;  
\item{\it timelike} if $\langle v,v\rangle_V<0$;
\item{\it null} if $\langle v,v\rangle_V=0.$
\end{itemize}
A linear map $P\colon V\to V$ is called {\it involution} if $P^2=\Id_V$ and {\it anti-involution} if $P^2=-\Id_V$. We say that a bijective linear map $T\colon V\to V$ is an {\it isometry} if 
$$
\langle Tv,Tv\rangle_V=\langle v,v\rangle_V\quad\text{for all vectors}\quad v\in V.
$$
and it is an {\it anti-isometry} if $\langle Tv,Tv\rangle_V=-\langle v,v\rangle_V
$ for all vectors.
The principal method for the construction of integral bases
starts by picking up a maximal number of mutually commuting isometric involutions
together with ``complementary'' isometries or anti-isometries satisfying some
commutation relations with the original involutions.  These choice of involutions and complementary operators give an orthogonal
decomposition of the representation space for the Clifford algebra. Choosing a common spacelike eigenvector of the original involutions we construct an integral basis by means of action on it of representations of the orthonormal generators for the corresponding Clifford algebra. There are several differences in the construction 
of orthogonal decompositions of the representation spaces by those involutions and complementary operators.
The purpose of the present work is, 
not only to show the existence of an integral structure for all $\Cl_{r,s}$-modules,
but also to present several possible methods for such kind of constructions,
especially for the cases of low dimensions.

The structure of this paper is as follows: 
Section~\ref{sec:aux} is an auxiliary section where we 
collected the information about properties of admissible modules and auxiliary technical lemmas. 
In Section~\ref{sec:00n} we prove that the isomorphism between Clifford algebras $\Cl_{r,s+1}$ and $\Cl_{s,r+1}$
preserves the admissible integral modules. This isomorphism reduces significantly the
number of the Clifford modules where we need to construct integral basis before we are able to apply the Bott periodicity.
In Section~\ref{sec:0n}
we show the existence of admissible integral $\Cl_{0,s}$-modules 
of minimal dimensions for $s\leq 8$. We set apart this section to emphasise the difference between the admissible modules for Clifford algebras $\Cl_{r,0}$ and $\Cl_{0,s}$.
In Section~\ref{sec:r0 and r1} we construct a minimal admissible module of $\Cl_{r,1}$ with
an integral basis basing on the existence of integral basis 
for the algebra $\Cl_{r,0}$. Section~\ref{sec:rs} devoted to the construction of integral structures on admissible $\Cl_{r,s}$-modules for $0\leq r,s\leq 8$ with $r+s\leq 8$. In the section we actively develop a method of the simultaneous orthogonal decomposition of eigenspaces for a collection of mutually commuting isometric involutions. We also exploit results of Sections~\ref{sec:00n} and~\ref{sec:r0 and r1}. In Section~\ref{sec:tensor} we prove some theorems that allow to use the Bott periodicity of Clifford algebras for construction of integral structures. We also construct admissible modules with integral basis for $\Cl_{r+1,s+1}$ and $\Cl_{0,n+2}$ 
based on the admissible module of  $\Cl_{1,1}$ and $\Cl_{2,0}$.
This method shows that the tensor product representation with 
some modification gives us an admissible module with integral basis,
but it need not be minimal. It remains to decompose this
admissible module into minimal one's together with an integral basis. Section~\ref{sec:g9} deals with integral structures on admissible $\Cl_{r,s}$-modules for $r+s\geq 9$. In the last Section~\ref{fremark} we make some observations about the presented constructions. Appendix contains the table of Clifford algebras, where circled Clifford algebras have the admissible modules of double dimension compare with the irreducible modules. It is also easy to see the symmetry with respect to the axis $r-s=-1$  that allows to use the construction of Section~\ref{sec:00n}.

 
 \section{Properties of admissible $\Cl_{r,s}$-modules}\label{sec:aux}
 

We recall the basic properties of admissible $\Cl_{r,s}$-modules
when $s>0$, see also~\cite{Ciatti}. We say that $W$ is an admissible sub-module of an admissible module $(V,\langle\cdot\,,\cdot\rangle_V)$ if $W$ is a Clifford sub-module of $V$ and the restriction of $\langle\cdot\,,\cdot\rangle_V$ on $W$, denoted by $\langle\cdot\,,\cdot\rangle_W$, is an admissible scalar product. There are decompositions of a given admissible
$\Cl_{r,s}$-module, $s>0$, into
non-admissible sub-modules, see Remarks~\ref{rem:2} and~\ref{rem:4} for examples. In the following proposition we give conditions that ensure a decomposition of an admissible Clifford module into admissible sub-modules.

\begin{proposition}\label{prop:1}
Let $(V,\langle\cdot\,,\cdot\rangle_V)$ be an admissible
$\Cl_{r,s}$-module with $s>0$ and $J_{z_k}$, $k=1,\ldots,r+s$,
representations of the 
orthonormal generators $z_1,\ldots,z_{r+s}$ of the Clifford algebra $\Cl_{r,s}$. 

$(1)$ Then the scalar product space $(V,\langle\cdot\,,\cdot\rangle_V)$ is neutral, i.~e. the maximal dimension of subspaces where the restriction of $\langle\cdot\,,\cdot\rangle_V$ is positive or negative definite coincide and, particularly, the
dimension of $V$ can be only even.

$(2)$ If $W$ is an admissible sub-module of a Clifford $\Cl_{r,s}$-module, then $W^{\perp}$
is also an admissible
sub-module. 
Hence, we have the decomposition of an admissible
$\Cl_{r,s}$-module $(V,\langle\cdot\,,\cdot\rangle_V)$
into admissible sub-modules.
\end{proposition}
\begin{proof} Since the module $(V,\langle\cdot\,,\cdot\rangle_V)$ is admissible we have 
\[
\langle J_{z_k}v,v\rangle_V=0\quad\text{and}\quad \langle
J_{z_k}v,J_{z_k}v\rangle_V
= \langle z_k,z_k\rangle_{\mathbb R^{r,s}}\langle v,v\rangle_V, \quad\text{for all}\quad k.
\]
Particularly, $\langle J_{z_{k}}v,J_{z_{k}}v\rangle_V=
-\langle v,v\rangle_V$, 
where $k=r+1,\ldots,r+s$. Let $v\in V$ be such that $\langle v,v\rangle_V\not=0$ and $W=\spn\{v,J_{z_{r+1}}v\}$. 
The scalar product restricted to the subspace $W$ is non-degenerate and neutral. The scalar
product restricted to the orthogonal
complement to $W$ is also non-degenerate. Applying the same arguments to $W^{\perp}$, we decompose it into two dimensional neutral subspace and the orthogonal complement, that shows the first statement. 

If $W$ is an admissible sub-module of a Clifford $\Cl_{r,s}$-module, then the action of
each $J_{z_k}$ leaves invariant 
the orthogonal complement 
$W^{\perp}=\{v\in V\mid\  \langle w,v\rangle_V=0~\text{for any}~w\in
W\}$. 
Indeed, if $\widetilde w\in W^{\perp}$ and $w\in W$, then 
$$
\langle J_{z_k}\widetilde w,w\rangle_V=-\langle \widetilde w,\underbrace{J_{z_k}w}_{\in W}\rangle_V=0.
$$
Since the scalar product restricted to $W^{\perp}$ is 
non-degenerate,
$W^{\perp}$
is also an admissible
sub-module. 
\end{proof}

Further we show the existence of non-trivial scalar product 
satisfying two conditions (consequently all three) 
among~\eqref{eq:J_orth},~\eqref{eq:J_Clif}, and~\eqref{eq:J_skew}
for 
Clifford
$\Cl_{r,s}$-modules with $s>0$.

\begin{lemma}\label{prop:2}
Let $V$ be an irreducible module of a Clifford algebra $\Cl_{r,s}:$
\[
J\colon\Cl_{r,s}\to\End(V),\qquad
J_z^2=-\langle z,z\rangle_{\mathbb{R}^{r,s}}\Id_V,\quad z\in \mathbb{R}^{r,s},
\]
with a symmetric bilinear form 
$
\langle\cdot\,,\cdot\rangle_V\colon V\times V \to \mathbb{R}
$
which satisfies
\[
\langle J_zv,w\rangle_V+\langle v,J_zw\rangle_V=0\ \ \text{for any} \ \ z\in
\mathbb{R}^{r,s}\ \ \text{and any}\ \  v,w\in V.
\]
Then the scalar product $\langle\cdot\,,\cdot\rangle_V$
is non-degenerate or identically vanishes.
\end{lemma}

\begin{proof}
Let $N=\{v\in V\mid\ \langle v,w\rangle_V=0\ \ \text{for any}\ \ w\in V\}$. Then
for any $z\in\mathbb{R}^{r,s}$ and $v\in N$
\[
\langle J_zv,w\rangle_V+\langle v,J_zw\rangle_V=0, \ \ 
\text{and}\ \ \langle v,J_zw\rangle_V=0,\ \ \text{for all}\ \ w\in V.
\]
Hence $\langle J_zv,w\rangle_V=0$, which shows that the subspace $N$ is invariant
under the action of the Clifford algebra.
So once we have an element $v\in V$ with $\langle v,v\rangle_V\not= 0$ then $N$ must
be the trivial space $\{0\}$ or entire $V$ due to the irreducibility of the module $V$.
\end{proof}

As was mentioned above, in the case of classical $H$-type algebras,
there is an admissible 
inner product for any $\Cl_{r,0}$-module, 
particularly, any irreducible module
can be an admissible module with an inner product.
However, for $s>0$ not all
irreducible $\Cl_{r,s}$-modules can be 
admissible modules. For instance, the Clifford algebra $\Cl_{0,2}$ is
isomorphic to the algebra $\R(2)$ of $(2\times 2)$ real matrices, and the
irreducible module is 2 dimensional, whereas the admissible $\Cl_{0,2}$-module of minimal
dimension is isomorphic to $\R^{2,2}$, see
Section~\ref{sec:0n}. The table presented in Appendix shows the Clifford algebras, where the circled
Clifford algebras has admissible modules of double dimension comparing 
with irreducible modules. However the following properties are still hold.

\begin{proposition}\label{prop:4}{\em (}\cite{Ciatti}{\em )}
Let $V$ be a $\Cl_{r,s}$-module, then there is a
scalar product on $V$ or on $V\oplus V$ with respect to which the resulting
$\Cl_{r,s}$-module is an admissible module.  The representation on $V\oplus V$ should be
redefined in an obvious way. 
\end{proposition}  

Particularly for irreducible modules $V$ the following corollary holds.

\begin{corollary}\label{cor:1}
An irreducible $\Cl_{r,s}$-module can be either an
admissible module of minimal dimension or 
the double of the irreducible module is the admissible module of minimal dimension.
As a result, a generalized $H$-type algebra constructed from the minimal
dimensional
admissible module is unique up to isomorphism. 
\end{corollary}

Any irreducible $\Cl_{r,s}$-module $V$ can be generated
by a non-zero vector $v\in V$ and the subsequent actions by $J_z$, $z\in\mathbb R^{r,s}$ (or Clifford multiplication by $z$). Similarly,
we have the following statement. 
\begin{proposition}\label{prop:5}
Any minimal dimensional admissible module is generated by a
non-null vector.
\end{proposition}

\begin{lemma}\label{prop:6}
Let $\Cl_{r,s}$ be a Clifford algebra with orthonormal generators
$\{z_i\}$, $i=1,\ldots,r+s$ and $(V,\langle\cdot,\cdot\rangle_V)$ an admissible $\Cl_{r,s}$-module with an orthonormal basis $\{v_\alpha\}_{\alpha=1}^{\dim V}$. Then
$J_{z_i}v_\alpha \not= \pm J_{z_j}v_\alpha$ for $i\not =j$.
\end{lemma}
\begin{proof}
Let us assume that $J_{z_i}v_\alpha = \pm J_{z_j}v_\alpha$ for some $i\not= j$ and an element $v_\alpha$ of the orthonormal basis.
From the assumption we have $J_{z_i\pm z_j}v_\alpha=0$ and hence
$J^2_{z_i\pm z_j}v_\alpha=-\langle z_i\pm z_j,z_i\pm z_j\rangle_{\mathbb  R^{r,s}} v_\alpha=0$, which
implies that $\langle z_i\pm z_j,z_i\pm z_j\rangle_{\mathbb  R^{r,s}} = 0 = \langle z_i,z_i\rangle_{\mathbb  R^{r,s}}+\langle z_j,z_j\rangle_{\mathbb  R^{r,s}}$. So, 
one of $z_i$ or $z_j$ is spacelike, and the other is timelike.
The assumption also implies that
$J_{z_i}J_{z_j}v_\alpha=\pm v_\alpha$, which leads to 
$$
\langle\pm v_\alpha,\pm v_\alpha \rangle_V=\langle J_{z_i}J_{z_j} v_\alpha,J_{z_i}J_{z_j} v_\alpha \rangle_V =\langle z_i,z_i\rangle_{\mathbb  R^{r,s}}\langle z_j,z_j\rangle_{\mathbb  R^{r,s}}\langle v_\alpha,v_\alpha \rangle_V=-\langle v_\alpha,v_\alpha \rangle_V.
$$
This is a contradiction since $v_{\alpha}$ is a non-null vector. 
\end{proof}
\begin{corollary}\label{cor:2-1}
Under assumptions of Lemma~\ref{prop:6} if we additionally assume that
each operator $J_{z_i}$ permutes the basis $\{v_\alpha\}$ up to sign, 
that is $J_{z_i}v_{\alpha}=\pm v_\beta$ for any $i,\alpha$ and some $\beta$, then 
with 
the given orthonormal generators $\{z_i\}$ of the Clifford algebra 
the basis $\{v_\alpha,z_i\}$ of the general Lie algebra
$V\oplus_{\perp}\mathbb{R}^{r,s}$ has structure
constants
$A^k_{\alpha,\beta}$ equal $\pm 1$ or 0.
\end{corollary}

In fact, in the present work we show that for any minimal dimensional admissible module $(V\langle \cdot\,,\cdot\rangle_V)$, one can find a vector $w\in V$ with $\langle w,w\rangle_V=1$ such that
the $2^{r+s}$-vectors 
\begin{equation*}
w,J_1w,\ldots,J_{r+s}w, \quad J_1J_2w,\ldots,J_{r+s-1}J_{r+s}w,\quad\ldots
J_{i_1}J_{i_2}\ldots J_{i_{k}}w,\quad \ldots J_1J_2\cdots J_{r+s}w,
\end{equation*}
satisfy the property that
any two vectors 
$J_{i_1}J_{i_2}\cdots J_{i_{k}}w$ and $J_{j_1}J_{j_2}\cdots J_{j_{\ell}}w$
coincide up to sign or orthogonal and
therefore we can select an orthonormal basis
which is produced from $w\in V$ by action of generators
$J_i=J_{z_i}$. 
Moreover $J_{z_i}$ act on the obtained basis as permutations up to sign.
Hence we prove that any general H-type algebra
have an integral structure according Lemma~\ref{prop:6} and Corollary~\ref{cor:2-1}.
 
Further we collect some auxiliary technical lemmas, 
that will be used later. We also say that a linear operator 
$\Omega\colon V\to V$ on a scalar product space
$(V,\langle\cdot\,,\cdot\rangle_V)$ 
is {\it symmetric} if $\langle \Omega v,w\rangle_V=\langle v,\Omega w\rangle_V$ for all $v,w\in V$.
 
\begin{lemma}\label{orthogonalization}
Let $(V,\langle\cdot\,,\cdot\rangle_V)$ be an admissible module, 
$\Omega\colon V\to V$ a symmetric linear operator such that 
$\Omega^2=-\Id_V$. Then for any $w\in V$ with $\langle w,w\rangle_V=1$ 
there is $\lambda\in\mathbb R$ such that $\tilde w=w+\lambda\Omega(w)$ satisfies:
$$
\langle\tilde w,\Omega\tilde w\rangle_V=0,\quad\text{and}\quad\langle\tilde w,\tilde w\rangle_V=1.
$$
\end{lemma}
\begin{proof}

First we clame that $\langle\Omega w,\Omega w\rangle_V=-\langle w,w\rangle_V$. Indeed 
$$
\langle\Omega w,\Omega w\rangle_V=\langle\Omega^2 w, w\rangle_V=-\langle w, w\rangle_V.
$$

Let $w\in V$ be a vector such that $\langle w,w\rangle_V=1$ and assume 
$\langle w,\Omega w\rangle_V=a$. If $a=0$, then we choose
$\lambda=0$. Thus, 
we can assume that $\langle w,\Omega w\rangle_V=a\neq 0$. Then, by solving the equation 
$$
\langle \tilde{w}, \Omega \tilde{w}\rangle_{V} 
=\langle w+\lambda \Omega w, -\lambda w+ \Omega w \rangle_{V} = -\lambda + a -\lambda^2 a -\lambda=-\big(a\lambda^2 + 2\lambda-a\big)=0,
$$
we find the solutions of this equation as
\[
\lambda=-\frac{1}{a}\pm\sqrt{1+\frac{1}{a^2}}\quad\text{for}\quad a\neq 0.
\]
For this $\lambda$ we get
\[
\langle \tilde w,\tilde w\rangle_V  =\langle w+\lambda \Omega w,w+\lambda \Omega w\rangle_V
=1-\lambda^2+2\lambda a
=2\Big(a+\frac{1}{a}\Big)\left(-\frac{1}{a}\pm\sqrt{1+\frac{1}{a^2}}\right)
=2\lambda\frac{a^2+1}{a}\neq 0.
\]
If $a>0$ then we choose
$\lambda=-\frac{1}{a}+\sqrt{1+\frac{1}{a^2}}>0$ 
and if $a<0$ then we choose
$\lambda=-\frac{1}{a}-\sqrt{1+\frac{1}{a^2}}<0$. 
This choice makes the product $\langle \tilde w,\tilde w\rangle_V$ strictly positive. Normalizing $\tilde w$
we get $\langle \tilde w,\tilde w\rangle=1$.
\end{proof}

The next lemma is a generalizaition of the previous one.

\begin{lemma}\label{orthogonal}
Let $(V,\langle\cdot\,,\cdot\rangle_V)$ be an admissible module, $\Omega_1,\ldots,\Omega_l$ symmetric linear operators on $V$ such that 
\begin{itemize}
\item[1)] $\Omega^2_k=-\Id_V$, $k=1,\ldots,l$;
\item[2)] $\Omega_k\Omega_j=-\Omega_j\Omega_k$ for all $k,j=1,\ldots,l$.
\end{itemize}
Then for any $w\in V$ with $\langle w,w\rangle_V=1$ there is a vector $\tilde w$ satisfying:
$$
\langle\tilde w,\Omega_k\tilde w\rangle_V=0,\quad\text{and}\quad\langle\tilde w,\tilde w\rangle_V=1,\ \ k=1,\ldots,l.
$$
\end{lemma}

\begin{proof}
Notice that symmetry of operators and property 2) imply 
\begin{equation}\label{eq:orthogonal1}
\langle\Omega_k v,\Omega_jv\rangle_V=0
\end{equation}
for any $v\in V$. Indeed
$
\langle\Omega_k v,\Omega_jv\rangle_V=\langle\Omega_j\Omega_k v,v\rangle_V=-\langle\Omega_k\Omega_j v,v\rangle_V=-\langle\Omega_j v,\Omega_kv\rangle_V,
$
that shows~\eqref{eq:orthogonal1}.

To prove Lemma~\ref{orthogonal} we choose $w\in V$ with $\langle w,w\rangle_{V}=1$. Apply Lemma~\ref{orthogonalization} to $w$ and $\Omega_1$ and construct $w_1\in V$ such that 
\begin{equation}\label{eq:w1}
\langle w_1,w_1\rangle_{V}=1,\quad \langle w_1,\Omega_1 w_1\rangle_{V}=0.
\end{equation}
Then we define $w_2=w_1+\lambda_2\Omega_2w_1$ and find that for suitable $\lambda_2\in\mathbb R$ we have
$$
\langle w_2,w_2\rangle_{V}=1,\quad \langle w_2,\Omega_2 w_2\rangle_{V}=0.
$$
Moreover, 
\begin{eqnarray*}
\langle w_2,\Omega_1w_2\rangle_{V} 
& = & \langle w_1,\Omega_1w_1\rangle_{V}+ 
\lambda_2\langle w_1,\Omega_1\Omega_2w_1\rangle_{V}+ 
\lambda_2\langle \Omega_2w_1,\Omega_1w_1\rangle_{V} \nonumber
\\
& + &
\lambda_2^2\langle\Omega_2 w_1,\Omega_1\Omega_2w_1\rangle_{V} =0
\end{eqnarray*}
by assumptions of Lemma~\ref{orthogonal} and properties~\eqref{eq:orthogonal1} and~\eqref{eq:w1}.
Now, applying Lemma~\ref{orthogonalization} we assume that the vector $w_k=w_{k-1}+\lambda_k\Omega_kw_{k-1}$, $2<k<l$, is chosen and satisfies
$$
\langle w_k,w_k\rangle_{V}=1,\quad \langle \Omega_k w_k,w_k\rangle_{V}=0,
$$
where it was shown that 
$$
\langle w_{k-1},w_{k-1}\rangle_{V}=1,\quad \langle w_{k-1},\Omega_j w_{k-1}\rangle_{V}=0, \ \ j=1,\ldots, k-1.
$$
Then 
\begin{eqnarray*}
\langle w_k,\Omega_jw_k\rangle_{V} 
& = & \langle w_{k-1},\Omega_jw_{k-1}\rangle_{V}+ 
\lambda_k\langle w_{k-1},\Omega_j\Omega_kw_{k-1}\rangle_{V}+ 
\lambda_k\langle \Omega_kw_{k-1},\Omega_jw_{k-1}\rangle_{V} \nonumber
\\
& + &
\lambda_k^2\langle\Omega_k w_{k-1},\Omega_{j}\Omega_kw_{k-1}\rangle_{V} =0,\quad\text{for any}\quad j=1,\ldots, k-1.
\end{eqnarray*}
by~\eqref{eq:orthogonal1} and assumption on operators $\Omega_j$. Denoting the last vector $w_l$ by $\widetilde w$, we finish the proof of Lemma~\ref{orthogonal}.
\end{proof}

\begin{corollary}\label{cor:orthogonal}
Let $(V,\langle\cdot\,,\cdot\rangle_V)$ and $\Omega_1,\ldots,\Omega_l$ satisfies the conditions of Lemma~\ref{orthogonal}. Let $P$ be a linear operator on $V$ such that 
$
P^2=\Id_V$, and $P\Omega_k=\Omega_kP$, $k=1,\ldots,l.
$
If $w\in V$ satisfies $Pw=w$ and $\langle w,w\rangle_{V}=1$, then the vector $\widetilde w$ constructed in Lemma~\ref{orthogonal} is also eigenvector of $P$: $P\widetilde w=\widetilde w$.
 \end{corollary}
 \begin{proof}
Let $w\in V$ satisfies $Pw=w$ and $\langle w,w\rangle_{V}=1$. Then for the vector $w_1=w+\lambda_1\Omega_1w$ we calculate
$
Pw_1=Pw+\lambda P\Omega_1w=w+\lambda\Omega_1 Pw=w_1
$.
Thus, we proceed further by induction and proof the Corollary.
\end{proof}

One of the principal parts in our construction is a presence of an operator having orthogonal eigenspaces. The following lemma describes some of them.

\begin{lemma}\label{rem:041}
Let $(V,\langle\cdot\,,\cdot\rangle_V)$ be a neutral scalar product space. Assume that an involution $P\colon V\to V$ is either symmetric or isometric operator. We denote its eigenspaces
$
E_+=\{v\in V\mid Pv=v\},\quad E_-=\{v\in V\mid Pv=-v\}.
$
Then the decomposition $P=E_+\oplus E_-$ is orthogonal with respect to $\langle\cdot\,,\cdot\rangle_V$.
\end{lemma}
\begin{proof} 
Let $w\in E_+$ and $v\in E_-$ be arbitrary vectors. If involution $P$ is symmetric then we argue as follows
$
\langle w,v\rangle_V=\langle Pw,v\rangle_V=\langle w,Pv\rangle_V=-\langle w,v\rangle_V$.

Let $P$ be an isometry. Then
$
\langle w,v\rangle_V=\langle Pw,Pv\rangle_V=-\langle w,v\rangle_V$,
where in the first equality we used the isometry property of $P$ and in the second the definition of eigenvectors. 
\end{proof}
 
From now on we fix the notation $E_\pm$ for eigenspaces of an involution $P$ corresponding to eigenvalues $\pm 1$. 
\begin{lemma}\label{lem:PT1}
Let $(V,\langle\cdot\,,\cdot\rangle_V)$ be a neutral scalar product space and $P\colon V \to V$ be an isometric involution. Then we have the following cases.
\begin{itemize} 
\item[1)] If a linear map $T\colon V\to V$ is an isometry such that $PT=-TP$, then each of eigenspaces 
$E_{\pm}$ of $P$ is a neutral scalar product space with respect to the 
restriction of the scalar product $\langle\cdot\,,\cdot\rangle_V$ on each of $E_{\pm}$.
\item[2)] If a linear map $T\colon V\to V$ is an anti-isometry such that  $PT=-TP$, 
then the restriction of $\langle\cdot\,,\cdot\rangle_V$ on each of
$E_{\pm}$ is non-degenerate neutral or sign definite,
\item[3)] If a linear map $T\colon V\to V$ is an anti-isometry 
such that $PT=TP$, then the restriction 
of $\langle\cdot\,,\cdot\rangle_V$ on each of $E_{\pm}$ is non-degenerate neutral.
\end{itemize}
\end{lemma}
\begin{proof} 
To show the first statement we observe that the isometry $T$ acts as an isometry from $E_+$ to $E_-$.
Since the eigenspaces $E_{\pm}$ are orthogonal, the scalar product $\langle\cdot\,,\cdot\rangle_V$
restricted to each of $E_{\pm}$ is non-degenerate. If the scalar product
restricted to $E_+$ would be positive definite, then the scalar product
restricted to $E_-$ would be also positive definite, since the map $T$
is an isometry that contradicts the assumption that space
$(V,\langle\cdot\,,\cdot\rangle_V)$ is neutral. The same arguments
shows that the restriction to $E_+$ could not be negative definite. 
So the scalar product restricted to $E_+$ and therefore to $E_-$
should be neutral.

In order to prove the second statement, we observe that since 
$T\colon E_+\to E_-$ is an anti-isometry, the restriction of 
$\langle\cdot\,,\cdot\rangle_V$ to $E_+$ can be sign definite 
and the restriction of $\langle\cdot\,,\cdot\rangle_V$ to $E_-$ 
will have opposite sign due to neutral nature of $(V,\langle\cdot\,,\cdot\rangle_V)$.

In the third case since the eigenspaces $E_{\pm}$ are invariant under
$T$ 
but contains spacelike and timelike vectors, 
they decompose into subspaces of equal 
dimensions where the restriction of $\langle\cdot\,,\cdot\rangle_V$ sign definite but of opposite signs.
\end{proof}

Note that assumptions made in the case 2 of Lemma~\ref{lem:PT1}, 
does not guarantee the existence of $w\in V$ such that 
\begin{equation}\label{eq:norm_eigen}
Pw=w,\qquad\langle w,w\rangle_V=1.
\end{equation}
The following lemma contains a benchmark example for our work, describing one of possible solutions of this problem. 

\begin{lemma}\label{lemma:04}
Let $(V,\langle\cdot\,,\cdot\rangle_V)$ be an admissible $\Cl_{r,s}$-module, $z_{1},\ldots,z_{r+s}$  orthonormal generators of the Clifford algebra $\Cl_{r,s}$, and $J_{z_i}\in\End(V)$, $i=1,\ldots, r+s$ are representations for the Clifford algebra. Assume that
$P=J_{z_{i_1}}J_{z_{i_2}}J_{z_{i_3}}J_{z_{i_4}}$, 
$i_1\neq i_2\neq i_3\neq i_4$, 
is an isometric involution and $T\colon V\to V$ is an anti-isometry such that $PT=-TP$. Then there is a vector $w\in V$ satisfying~\eqref{eq:norm_eigen}
or $P$ can be modified to other isometric involution $\hat P$ such that~\eqref{eq:norm_eigen} holds for $\hat P$.
\end{lemma}
\begin{proof}
First we notice that operator $P=J_{z_{i_1}}J_{z_{i_2}}J_{z_{i_3}}J_{z_{i_4}}$, $i_1\neq i_2\neq i_3\neq i_4$ is also symmetric.

We apply Lemma~\ref{lem:PT1}, item 2). If the restriction of $\langle\cdot\,,\cdot\rangle_V$ on $E_+$ is positive definite or neutral, then we are done. If the restriction is negative definite, then we define the operator $\hat P=J_{z_{i_2}}J_{z_{i_1}}J_{z_{i_3}}J_{z_{i_4}}$ and denote by $\hat E_{\pm}$ its eigenspaces corresponding to eigenvalues $\pm 1$. Thus if $w\in E_+$, then $\hat Pw=-Pw=-w$ and therefore $w\in\hat E_-$. Continuing to argue in the same way, we conclude that $E_+=\hat E_-$ and $E_-=\hat E_+$. So, we change the operator $P$ and its eigenspaces $E_\pm$ to the operator $\hat P$ and the corresponding eigenvectors $\hat E_\mp$ to satisfy~\eqref{eq:norm_eigen}. 
\end{proof}

To ensure existence of $w\in V$ satisfying~\eqref{eq:norm_eigen} for a general isometric involution $P$ we need to have one more operator commuting with $P$. More precisely we state the following generalisation of Lemma~\ref{lem:PT1}. 

\begin{lemma}\label{lem:PT2}
Let $(V,\langle\cdot\,,\cdot\rangle_V)$ be a neutral scalar product space.
Let $P$ be an isometric involution and assume that there are two anti-isometric operators $T_i\colon V\to V$, $i=1,2$, such that
$$
T_1P=-PT_1,\quad\text{and}\quad T_2P=PT_2.
$$
Then the eigenspaces $E_{\pm}$ of $P$ are non-trivial and the scalar product $\langle\cdot\,,\cdot\rangle_V$
restricted to each of $E_{\pm}$ is non-degenerate and neutral.
\end{lemma}
\begin{proof}
We only need to explain why the scalar product $\langle\cdot\,,\cdot\rangle_V$ restricted to each of $E_{\pm}$
is neutral. The non-degeneracy of the restriction of $\langle\cdot\,,\cdot\rangle_V$ to $E_{\pm}$ is shown as in Lemma~\ref{lem:PT1}.

The presence of the operator $T_1$ ensures that the restriction of $\langle\cdot\,,\cdot\rangle_V$ to $E_{\pm}$ is neutral or sign definite. Actually the restriction of $\langle\cdot\,,\cdot\rangle_V$ to $E_{+}$ can not be sign definite. Indeed, since $T_2$ preserves $E_{+}$ and it is an anti-isometry, the space $E_{+}$ contains both spacelike and timelike vectors forming subspaces of an equal dimension as was shown in the proof of Proposition~\ref{prop:1}. The same arguments, applied to $E_-$, finish the proof.
\end{proof}

\begin{definition}
Let $P_1,\ldots, P_m$ be isometric mutually commuting involutions defined on a neutral scalar product space $(V,\langle\cdot\,,\cdot\rangle_V)$. Then the collection $T_1,\ldots, T_m, T_{m+1}$ of linear operators on $V$ is called complementary operators to the family $P_1,\ldots, P_m$ if 
$$
\begin{array}{lllllll}
&P_1T_1=-T_1P_1, & P_1T_2=T_2P_1, &\ldots & P_1T_m=T_mP_1, & P_1T_{m+1}=T_{m+1}P_1,
\\
&&P_2T_2=-T_2P_2,&\ldots& P_2T_m=T_mP_2, &P_2T_{m+1}=T_{m+1}P_2,
\\
&&&&\vdots
\\
&&&&P_mT_m=-T_mP_m,& P_mT_{m+1}=T_{m+1}P_m.
\end{array}
$$ 
\end{definition}
\begin{remark}
In some of situations the operator $T_{m+1}$ can be omitted, but we still call the system of operators $T_1,\ldots, T_m$ complementary to  $P_1,\ldots, P_m$.
\end{remark}

Based on Lemmas~\ref{lem:PT1} and~\ref{lem:PT2}
we construct an integral structure 
by giving an explicit simultaneous eigenspace decomposition of a given
admissible module by a family of isometric involutions and their complementary operators.
Simultaneously, we calculate the dimension of the minimal admissible modules for all
cases.


\section{Isomorphism preserving  admissibility}\label{sec:00n}

There are several types of isomorphisms between Clifford
algebras. Among them the periodicity with the period $8$
$$
\Cl_{r+8,s}\cong \Cl_{r+4,s+4}\cong\Cl_{r,s}\otimes\mathbb R(16),\qquad
\Cl_{r,s+8}\cong \Cl_{r+4,s+4}\cong\Cl_{r,s}\otimes\mathbb R(16)
$$
are basic and used 
to construct an integral basis for all
the cases $\Cl_{r,s}$ after we prove the existence 
of the integral basis for $\Cl_{r,s}$ of $0\leq r,s\leq 7$ 
and $\Cl_{8,0}\cong\mathbb{R}(8)$, $\Cl_{0,8}\cong\mathbb{R}(8)$, 
see Theorems~\ref{th:rs8},~\ref{th:r8s} and~\ref{th:r4s4}. 

Not all isomorphisms of the Clifford algebras lead 
to the isometric admissible modules, for instance, the isomorphism $\Cl_{r,s+4}\cong \Cl_{s,r+4}$ does not
preserve the admissibility in general, since in particular, 
the isomorphism $\Cl_{0,4}\cong \Cl_{4,0}$
does not directly give us a required scalar product from the positive one for $\Cl_{4,0}$-module to
a neutral one for $\Cl_{0,4}$-module. 
In this section we show that the isomorphism $\Cl_{r,s+1}\cong \Cl_{s,r+1}$
preserves the admissibility. 

We recall an isomorphism between $\Cl_{r,s+1}$ and $\Cl_{s,r+1}$. It
is given as follows:
let $z_1,\ldots,z_{r}$, $ \zeta_1,\ldots,\zeta_{s+1}$ be the
orthonormal generators of $\Cl_{r,s+1}$ with the property
\[
\langle z_i,z_i\rangle_{\mathbb R^{r,s+1}}=1,\qquad\langle \zeta_j,\zeta_j\rangle_{\mathbb R^{r,s+1}}=-1.
\]
Likewise let $a_1,\ldots,a_{s},b_1,\ldots,b_{r+1}$ be orthonormal generators of
$\Cl_{s,r+1}$ with
$$
\langle a_i,a_i\rangle_{\mathbb R^{s,r+1}}=1,\qquad\langle b_j,b_j\rangle_{\mathbb R^{s,r+1}}=-1.
$$
We define a correspondence $\Phi\colon\mathbb{R}^{r,s+1}\to \Cl_{s,r+1}$ by
$$
\begin{array}{lllll}
&z_1 \longmapsto b_1b_{r+1},\qquad & \zeta_1\longmapsto a_1b_{r+1},\\
& z_2 \longmapsto b_2b_{r+1}, &\cdots\qquad \cdots,\\
&\cdots\qquad\cdots,&\zeta_s\longmapsto a_sb_{r+1},\\
&z_r\longmapsto  b_rb_{r+1},&\zeta_{s+1}\longmapsto b_{r+1}. 
\end{array}
$$
Then we have $\Phi(z_i)^2=-1$, $\Phi(\zeta_j)^2=1$ and 
\begin{align*}
&\Phi(z_i)\Phi(\zeta_j)+\Phi(\zeta_j)\Phi(z_i)=0\quad\text{for any}\quad i=1,\ldots,r,\ \  j=1,\ldots s,\qquad~\text{and}\\
&\Phi(z_{i_1})\Phi(z_{i_2})+\Phi(z_{i_2})\Phi(z_{i_1})=0,\ \
\Phi(\zeta_{j_1})\Phi(\zeta_{j_2})+\Phi(\zeta_{j_2})\Phi(\zeta_{j_1})=0\ \text{for any}\ 
i_k\not=i_l,\  j_k\not=j_l.
\end{align*}
Since the vectors
$\{\Phi(z_i),\Phi(\zeta_j)\}_{i,j}$
are linearly independent, one can extend the isomorphism to an isomorphism
of the Clifford algebras $\Cl_{r,s+1}$ and $\Cl_{s,r+1}$.

\begin{theorem}\label{prop:7}
Let $(V,\langle\cdot\,,\cdot\rangle_V)$ be an admissible module of the Clifford algebra
$\Cl_{s,r+1}$ and $J\colon\Cl_{s,r+1}\to \End (V)$ its representation. Then the Clifford module $(V,\langle\cdot\,,\cdot\rangle_V)$ with representation
$$J\circ\Phi\colon\Cl_{r,s+1}\to \End(V)$$
is admissible. 
\end{theorem}
\begin{proof}
The skew symmetry condition
\[
\langle J_{\Phi(z_i)}v,w\rangle_V+\langle v,J_{\Phi(z_i)}w\rangle _V=0 \quad\text{for any}\quad v,w\in V
\]
holds by the following 
\begin{align*}
\langle J_{\Phi(z_i)}v,w\rangle_V&= \langle J_{b_i}J_{b_{r+1}}v,w\rangle_V
=-\langle J_{b_{r+1}}v,J_{b_i}w\rangle_V
\\
&= \langle v,J_{b_{r+1}}J_{b_i}w\rangle_V
=-\langle v,J_{b_{i}}J_{b_{r+1}}w\rangle_V
= -\langle v,J_{\Phi(z_i)}w\rangle_V\quad\text{for}\quad i\leq r.
\end{align*}
The rest of the cases can be shown in a similar way.
\end{proof}

\begin{corollary}\label{cor:3}
If an integral basis for an admissible $\Cl_{s,r+1}$-module satisfies the assumptions of
Lemma~\ref{prop:6} and 
Corollary~\ref{cor:2-1}, 
then it is also
an integral basis of the admissible $\Cl_{r,s+1}$-module.
\end{corollary}
\begin{proof}
Indeed, let $\{v_{\alpha}\}$ be an integral basis for $\Cl_{s,r+1}$-module, then
\begin{align*}
\langle J_{\Phi(z_i)}v_\alpha,v_{\beta}\rangle_V=\langle J_{b_i}J_{b_{r+1}}v_\alpha,v_\beta\rangle_V=-\langle J_{b_{r+1}}v_{\alpha},J_{b_{i}}v_{\beta}\rangle_V=\pm 1\ \text{or}\ 0,
\end{align*}
since the vectors $J_{b_{r+1}}v_{\alpha}$ and $J_{b_{i}}v_{\beta}$ are also
ones of the basis vectors up to sign by the assumption.
\end{proof}

{\it This Lemma~\ref{prop:6} and Corollary~\ref{cor:3} allows to reduce the construction of integral bases from $64$ to $42$ cases, and
we shall construct such bases for the Clifford algebras
\[
\Cl_{r,s},\quad\text{for}\quad 0\leq r\leq s\leq 7\qquad\text{and}\qquad\Cl_{r,r+1}\quad \text{for}\quad r=0,\ldots,6. 
\] }

 
 \section{Integral structure on admissible $\Cl_{0,s}$-modules }\label{sec:0n}
 

In this section we present admissible integral $\Cl_{0,s}$-modules for
$s=1,\ldots,8$, constructing them directly. 


\subsection{Integral structure on admissible $\Cl_{0,1}$-module}


In this case the Clifford algebra is isomorphic 
to the space $\mathbb R\oplus \mathbb R$, where the isomorphism is given by
$$
1_{\Cl_{0,1}}\mapsto (1,1),\qquad z_1\mapsto (1,-1)
$$
and then it is continued to the algebra isomorphism.
The generator $z_1$ of $\Cl_{0,1}$ 
satisfies $\langle z_1,z_1\rangle_{\mathbb R^{0,1}}=-1$.  Let $(V,\langle \cdot\,,\cdot\rangle_{V})$ be 
an admissible module and the endomorphism
$J_{z_1}$ be such that $J^2_{z_1}=\Id_V$. 
We pick up $w\in V$ with $\langle w,w\rangle_{V}=1$. Then 
$
\langle J_{z_1}w,w\rangle_{V}=0,
$
that gives us two linearly independent vectors. We can choose the basis 
$$
v_1=w,\qquad v_2=J_{z_1}w.
$$ Then, since
$$
\langle v_1,v_1\rangle_{V}=1,\quad \langle v_2,v_2\rangle_{V}
=\langle z_1,z_1\rangle_{\mathbb R^{0,1}}\langle w,w\rangle_{V}=-1
$$
the basis is orthonormal and we obtain one non-vanishing relation
$
\langle J_{z_1}v_1,v_2\rangle_{V}=-1$.
It gives the following commutation relation
$$
\langle J_{z_1}v_1,v_2\rangle_V=\langle z_1,[v_1,v_2]\rangle_{\mathbb R^{0,1}}\quad\Longrightarrow\quad [v_1,v_2]=z_1.
$$ 
This Lie algebra $\big(\mathbb{R}^{1,1}\oplus_{\perp}\mathbb{R}^{0,1},[\cdot\,,\cdot]\big)$ 
is the Heisenberg algebra. The natural choice of the metric of index
$(1,1)$ on the space $V\cong \mathbb{R}^{1,1}$ 
leads to studying of the Lorentzian Heisenberg groups, see~\cite{G, KM}. 
The constructed admissible module is of minimal dimension, 
but it is not irreducible, since the irreducible module has dimension 1.

\begin{remark}\label{rem:2}
Let $u=w+J_{z_1}w$, then $u\not=0$ and $\langle{ u},u\rangle_V=0$, also
$J_{z_1}u=u$. Hence the subspace $W=\spn\{u\}$ of $V$
is a $\Cl_{0,1}$-module with the trivial scalar product.
So, in general a sub-module of an admissible $\Cl_{r,s}$-module need not be
an admissible module, as was observed at the beginning of Section~\ref{sec:aux}. 
\end{remark}

\begin{remark}\label{three dim Heisenberg}
The Clifford algebra $\Cl_{1,0}$ is isomorphic to the field $\mathbb{C}$ of complex numbers
 and the corresponding H-type Lie algebra $\big(\mathbb{R}^{2,0}\oplus_{\perp}\mathbb{R}^{1,0},[\cdot\,,\cdot]\big)$ is also
three dimensional Heisenberg algebra with the Euclidean metric as a
natural choice on the horizontal space. 
We see that different Clifford algebras $\Cl_{1,0}$ and $\Cl_{0,1}$ 
lead to different general $H$-type Lie algebras, having isomorphic 
underlying Lie algebras, if we discard the presence of the metric.
\end{remark}


\subsection{Integral structure on admissible $\Cl_{0,2}$-module }


The Clifford algebra $\Cl_{0,2}$ is isomorphic to the space 
$\mathbb R(2)$ of $(2\times 2)$-matrices with real entries.
Let $z_1,z_2$ be generators of $\Cl_{0,2}$ with $\langle
z_i,z_i\rangle_{\mathbb R^{0,2}}=-1$ 
and $(V,\langle \cdot\,,\cdot\rangle_{V})$ an admissible $\Cl_{0,2}$-module. Then 
$
J_i^2:=J^2_{z_i}=\Id_V$. Choose $w\in V$ such that $\langle w,w\rangle_{V}=1$. 
The basis 
$$
v_1=w,\quad v_2=J_1J_2w,\quad v_3=J_1w,\quad v_4=J_2w
$$ is orthonormal 
and satisfies 
$
\langle v_1,v_1\rangle_{V}=\langle v_2,v_2\rangle_{V}=-\langle v_3,v_3\rangle_{V}=-\langle v_4,v_4\rangle_{V}=1$.
This implies that the scalar product restricted to the subspace spanned
by the four vectors $\{v_1,v_2,v_3,v_4\}$ is non-degenerate.
Moreover, since the action of $J_j$, $j=1,2$, on the basis
$\{v_{\alpha}\}_{\alpha=1}^4$ 
permutes this basis up to sign, the basis is integral that gives the following non-vanishing commutation relations
$$
[v_1,v_3]=[v_2,v_4]=z_1.\qquad [v_1,v_4]=-[v_2,v_3]=z_2.
$$

\begin{remark}\label{rem:4}
Let $u_1=v_1+v_3$ and $u_2=v_2-v_4$, then
\[
J_1u_1=u_1,\ \  J_1u_2=-u_2,\ ~\text{and}~\  J_2u_1=u_2,\ \  J_2u_2=-u_1.
\]
The scalar product on the subspace $W=\spn\{u_1,u_2\}$ vanishes, 
so the irreducible sub-module $W$ is not admissible. 
\end{remark}

We emphasise that 
the construction of admissible integral modules 
some times gives the irreducible admissible
integral module, 
but in other cases the resulting module exceeds the 
dimension of the irreducible module twice.


\subsection{Integral structure on admissible $\Cl_{0,3}$-module}


The Clifford algebra $\Cl_{0,3}$ is isomorphic to the space 
$\mathbb C(2)$ of $(2\times 2)$-matrices with complex entries.
Let $z_1,z_2,z_3$ be generators of $\Cl_{0,3}$ with 
$\langle z_i,z_i\rangle_{\mathbb R^{0,3}}=-1$, $i=1,2,3$, and $(V,\langle \cdot\,,\cdot\rangle_{V})$ an admissible
$\Cl_{0,3}$-module. Then 
$
J_i^2:=J^2_{z_i}=\Id_V$, $ i=1,2,3$.
Choose $w\in V$ such that 
$\langle w,w\rangle_{V}=1$. 
In general, the scalar product
$\langle w, J_1J_2J_3w\rangle_V$ does not vanish. Nevertheless, 
since the operator $\Omega=J_1J_2J_3$ satisfies the conditions of 
Lemma~\ref{orthogonalization} we can choose the initial vector making 
the product $\langle w, J_1J_2J_3w\rangle_V$ equal to $0$.
Fix such a vector $w\in V$ and pick up the vectors 
$$
\begin{array}{llllll}
&v_1=w,\quad &v_2=J_1J_2w,\quad &v_3=J_1J_3w,\quad &v_4=J_2J_3w,
\\
&v_5=J_1w,\quad &v_6=J_2w,\quad &v_7=J_3w,\quad &v_8=J_1J_2J_3w,
\end{array}
$$
which satisfy $\langle v_{\alpha},v_{\alpha}\rangle_{V}=1$, $\alpha=1,2,3,4$ and $\langle v_{\alpha},v_{\alpha}\rangle_{V}=-1$ for
$\alpha=5,6,7,8$. 
Moreover, since we have 
$$
0=\langle v_{1},v_{8}\rangle_{V}=-\langle v_{2},v_{7}\rangle_{V}=\langle v_{3},v_{6}\rangle_{V}=-\langle v_{4},v_{5}\rangle_{V}
$$
the basis $\{v_1,\ldots,v_8\}$ is orthonormal with respect
to the scalar product $\langle\cdot\,,\cdot\rangle_V$ by the choice of 
initial vector $w$ and admissibility of the scalar product. 

It is clear that the action of $J_j$, $j=1,2,3$, on the basis $\{v_{\alpha}\}_{\alpha=1}^8$ permutes it up to the sign. 
Thus we conclude $\langle J_jv_\alpha,v_\beta\rangle_{V}=\pm 1$ or 
$\langle J_jv_\alpha,v_\beta\rangle_{V}=0$ for any $j=1,2,3$ and any
$\alpha,\beta=1,\ldots,8$. 
Precisely, we obtain the following non-vanishing commutators 
$$
[v_1,v_5]=[v_2,v_6]=[v_3,v_7]=[v_4,v_8]=z_1,\quad [v_1,v_6]=-[v_2,v_5]=-[v_3,v_8]=[v_4,v_7]=z_2
$$
$$
[v_1,v_7]=[v_2,v_8]=-[v_3,v_5]=-[v_4,v_6]=z_3.
$$
We apply Proposition~\ref{prop:1}, part (2) and note, that $W=\spn\{v_{1},\ldots, v_8\}$ 
is invariant under the action of the algebra $\Cl_{0,3}$ 
and the scalar product $\langle\cdot\,,\cdot\rangle_V$ restricted to
$W$
(we denote it by  $\langle\cdot\,,\cdot\rangle_W$)
is non-degenerate. So we constructed an admissible sub-module
$(W, \langle\cdot\,,\cdot\rangle_W)$. The orthogonal complement
$W^{\perp}=\{x\in V~|~ \langle x, v\rangle_V=0~\text{for any}~v\in W\}$
is also invariant under the action of the Clifford algebra $\Cl_{0,3}$ and 
the scalar product restricted to $W^{\perp}$ is non-degenerate. This procedure
implies
that the given admissible $\Cl_{0,3}$-module
$(V,\langle\cdot\,,\cdot\rangle)$
is decomposed into a finite sum of the minimal $8$-dimensional admissible
modules $(W,\langle\cdot\,,\cdot\rangle_{W})$ of $\Cl_{0,3}$.


\subsection{Integral structure on admissible $\Cl_{0,4}$-module}
\label{Cl04}


The Clifford algebra $\Cl_{0,4}$ is isomorphic to the space $\mathbb H(2)$.
Let $z_j$, $j=1,2,3,4$ be generators 
of $\Cl_{0,4}$ with $\langle z_j,z_j\rangle_{\mathbb R^{0,4}}=-1$, 
and $(V,\langle \cdot\,,\cdot,\rangle_{V})$ an admissible $\Cl_{0,4}$-module. Then 
$
J_j^2:=J^2_{z_j}=\Id_V$, $j=1,2,3,4$.
The operator $P=J_1J_2J_3J_4$ is isometric involution and $T=J_1$ is an anti-isometry such that $PT=-TP$. Applying Lemma~\ref{lemma:04} we find $w\in V$ such that 
$Pw=w$ and $\langle w,w\rangle_{V}=1$. 
Then we get relations
\begin{align}\label{eq:Cl04_2}
&J_1w= J_2J_3J_4w,\quad J_2w=-J_1J_3J_4w,\quad J_3w=J_1J_2J_4w,\quad J_4w=- J_1J_2J_3w,\nonumber
\\
&J_1J_2w=-J_3J_4w, \quad J_1J_3w=J_2J_4w, \quad J_1J_4w=-J_2J_3w,
\end{align}
that allows us to choose 8 orthogonal vectors
\begin{align}\label{basis 04}
\begin{array}{lllllll}
& v_1=w,\quad & v_2=J_1J_2w,\quad  & v_3=J_1J_3w,\quad & v_4=J_1J_4w,
\\
&v_5=J_1w, & v_6=J_2w, & v_7=J_3w, & v_8=J_4w,
\end{array}
\end{align}
satisfying $
\langle v_\alpha,v_\alpha\rangle_{V}=1$, $\alpha=1,2,3,4$ and $\langle v_\alpha,v_\alpha\rangle_{V}=-1$, for $\alpha=5,6,7,8$.

The relations~\eqref{eq:Cl04_2} shows that the action 
of $J_j$, $j=1,2,3,4$, on the basis $\{v_{\alpha}\}_{\alpha=1}^8$ 
permutes elements of the basis up to sign. 
We conclude $\langle J_jv_\alpha,v_\beta\rangle_{V}=\pm 1$ 
or $0$ for
$\alpha,\beta=1,\ldots,8$. 

The subspace spanned by vectors~\eqref{basis 04} is invariant
under the action of the Clifford algebra $\Cl_{0,4}$.
Using Proposition~\ref{prop:1}, (2) we conclude that 
constructed an admissible sub-module $W=\spn\{v_1,\ldots,v_8\}$, is of minimal dimension admissible integral
module. 
This module is irreducible.


\subsection{Integral structure on admissible $\Cl_{0,5}$-module}
\label{1module05}


The Clifford algebra $\Cl_{0,5}$ is isomorphic to the direct sum 
$\mathbb H(2)\oplus\mathbb H(2)$ of spaces of $(2\times 2)$-matrices with quaternion entries.
Let $z_j$, $j=1,\ldots,5$, be generators of $\Cl_{0,5}$ with $\langle
z_j,z_j\rangle_{\mathbb R^{0,5}}=-1$ 
and $(V,\langle \cdot\,,\cdot\rangle_{V})$ an admissible
$\Cl_{0,5}$-module. Then 
$
J_j^2:=J^2_{z_j}=\Id_V$, $j=1,\ldots,5$.
Consider the isometric involution $P_1=J_1J_2J_3J_4$ and operator $T=J_5$. Then the restriction of $\langle \cdot\,,\cdot\rangle_{V}$ to the eigenspace $E_{+}$ of the involution $P$ is neutral by Lemma~\ref{lem:PT1}. Thus we can find $w\in E_{+}$ such that $\langle w,w\rangle_{V}=1$. 

Recalling the relations~\eqref{eq:Cl04_2} and taking into account the invarianse of $E_{\pm}$ under $J_5$, we present an eigenspace decomposition of the involution $P_1$.
\begin{table}[ht]\label{tab:1}
\begin{center}
\caption{Eigenspace decomposition: $\Cl_{0,5}$ case}
\begin{tabular}{|l|c|c|} \hline
Involutions&\multicolumn{2}{|c|}{Eigenvalues}
\\ 
\hline
$P_1$&\multicolumn{1}{|c|}{$+1$} &\multicolumn{1}{|c|}{$-1$}\\\hline\hline
{Eigenvector}
&$w$,\ $J_1J_2w$,\ $J_1J_3w$,\ $J_1J_4w$ 
&$J_1w$,\ $J_2w$,\ $J_3w$,\ $J_4w$
\\
&$J_5w$,\ $J_1J_2J_5w$,\ $J_1J_3J_5w$,\ $J_1J_4J_5w$
&$J_1J_5w$,\ $J_2J_5w$,\ $J_3J_5w$,\ $J_4J_5w$
\\
\hline
\end{tabular}
\end{center}
\end{table}

It allows to choose the linear independent elements
\begin{equation} \label{eq:basis05}
\begin{array}{lllllllllll}
& v_1=w,\quad & v_2=J_1J_2w,\quad & v_3=J_1J_3w,\quad  & v_4=J_1J_4w,
\\ 
& v_5=J_1J_5,& v_6=J_2J_5w, & v_7=J_3J_5w, & v_8=J_4J_5w,
\\
&v_9=J_1w,\quad &v_{10}=J_2w,\quad &v_{11}=J_3w,\quad &v_{12}=J_4w,
\\ 
&v_{13}=J_5w,\ &v_{14}=J_1J_2J_5w,& v_{15}=J_1J_3J_5w,& v_{16}=J_1J_4J_5w.
\end{array}
\end{equation}
satisfying $\langle v_\alpha,v_\alpha\rangle_{V}=1$, $\alpha=1,\ldots, 8$ and $\langle v_\alpha,v_\alpha\rangle_{V}=-1$, $\alpha=9,\ldots,16$.
Unfortunately not all the vectors in~\eqref{eq:basis05} are orthogonal for an arbitrary choice of $w\in V$. The possible non-vanishing relations are
$$
\begin{array}{llllllllllllllllll}
& \langle v_1,v_{14}\rangle_{V} & = \langle v_4,v_{15}\rangle_{V} & =\langle v_5,v_{10}\rangle_{V} & =\langle v_8,v_{11}\rangle_{V}=a,
\\
& \langle v_2,v_{13}\rangle_{V} & = \langle v_3,v_{16}\rangle_{V} & =\langle v_6,v_{9}\rangle_{V} & =\langle v_7,v_{12}\rangle_{V}=-a,
\\
& \langle v_1,v_{15}\rangle_{V} & = \langle v_2,v_{16}\rangle_{V} & =\langle v_5,v_{11}\rangle_{V} & =\langle v_6,v_{12}\rangle_{V}=b,
\\
& \langle v_3,v_{13}\rangle_{V} & = \langle v_4,v_{14}\rangle_{V} & =\langle v_7,v_{9}\rangle_{V} & =\langle v_8,v_{10}\rangle_{V}=-b,
\\
& \langle v_1,v_{16}\rangle_{V} & = \langle v_3,v_{14}\rangle_{V} & =\langle v_5,v_{12}\rangle_{V} & =\langle v_7,v_{10}\rangle_{V}=c,
\\
& \langle v_2,v_{15}\rangle_{V} & = \langle v_4,v_{13}\rangle_{V} & =\langle v_6,v_{11}\rangle_{V} & =\langle v_8,v_{9}\rangle_{V}=-c.
\end{array}
$$
If we denote $\Omega_1=J_1J_2J_5$, $\Omega_2=J_1J_3J_5$, and $\Omega_3=J_1J_4J_5$, then we apply Lemma~\ref{orthogonal} and Corollary~\ref{cor:orthogonal} and find $\widetilde w$ making all the vectors in~\eqref{eq:basis05} orthogonal.

Moreover the relations~\eqref{eq:Cl04_2} and Table~1 shows that the action of $J_j$, $j=1,\ldots, 5$ permutes the basis up to sign. The constructed admissible integral sub-module 
$(W,\langle \cdot\,,\cdot\rangle_{W})$, $W=\spn\{v_1,\ldots,v_{16}\}$, 
is of minimal dimension, 
but is not irreducible, since the irreducible module of the Clifford
algebra $\Cl_{0,5}\cong\mathbb H(2)\oplus\mathbb H(2)$ is of dimension $8$. 


\subsection{Integral structure on admissible $\Cl_{0,6}$-module}\label{06}


The Clifford algebra $\Cl_{0,6}$ is isomorphic to the space $\mathbb H(4)$.
Let $z_j$, $j=1,\ldots,6$ be orthogonal generators of $\Cl_{0,6}$ with 
$\langle z_j,z_j\rangle_{\mathbb R^{0,6}}=-1$ and $(V,\langle\cdot\,
,\cdot\rangle_{V})$ 
an admissible $\Cl_{0,6}$-module. Then 
$
J_j^2:=J^2_{z_j}=\Id_V$, $j=1,\ldots,6$. 
Consider two isometric and mutually commuting involutions $P_1=J_1J_2J_3J_4$ and $P_2=J_1J_2J_5J_6$. 
In order to construct a collection of 
complementary operators we present the 
table of commuting relations with generators $J_j$, $j=1,\ldots,6$.

\smallskip

\centerline {Commutation relations: $\Cl_{0,6}$ case}
\begin{equation}\label{eq:Cl06Gene}
\begin{array}{|l|c|c|c|c|c|c|c|}\hline
\text{Involution}\backslash \text{Generator}&J_{1}&J_{2}&J_{3}&J_{4}&J_{5}&J_{6},\\\hline
P_1=J_{1}J_{2}J_{3}J_{4}&a&a&a&a&c&c\\\hline
P_2=J_{1}J_{2}J_{5}J_{6}&a&a&c&c&a&a\\\hline
\end{array}
\end{equation}
Here "a" denotes that the corresponding operators anti-commute and "c"
means that they are commuting. 
Then we give the table of complementary operators, where we use 
the sign $(+\to +)$ to emphasise that the operator is an isometry 
and $(+\to -)$ to point out the anti-isometric operator.

\begin{center}
{Complementary operator}: $\Cl_{0,6}$ case
\end{center}
\begin{equation*}\label{eq:Cl06comp}
\begin{array}{|l|c|c|c|c|}\hline
\text{Involution}\backslash \text{Comp. op.}
&J_{1}(+\to -)&J_{5}(+\to -)&J_2J_3J_5(+\to -)\\\hline
P_1=J_{1}J_{2}J_{3}J_{4}&a&c&c\\\hline
P_2=J_{1}J_{2}J_{5}J_{6}&&a&c\\\hline
\end{array}
\end{equation*}

From these two tables we deduce, that since $P_1J_1=-J_1P_1$ and
$P_1J_5=J_5P_1$ the eigenspaces $E_{1\pm}$ of the operator $P_1$ are
neutral spaces by Lemma~\ref{lem:PT2}. 
Then we apply  Lemma~\ref{lem:PT2} to the neutral space $E_{1+}$, the
operator $P_2$ and anti-isometries $T_1=J_5$, $T_2=J_2J_3J_5$ and
conclude that $E_{1+}$ is decomposed into 
two eigenspaces $E_{2\pm}$ of the operator $P_2$ 
which both are neutral. The same arguments 
are applied to the neutral space $E_{1-}$ and operators $P_2$ and $T_1=J_5$, $T_2=J_2J_3J_5$.
Thus we can find $w\in E_{2+}\cap E_{1+}\subset V$ such that 
$$
P_1w=J_1J_2J_3J_4w=w,\quad P_2w=J_1J_2J_5J_6w=w,\quad\text{and}\quad\langle w,w\rangle_V=1.
$$

Then we fix linear dependent vectors
\begin{equation}\label{eq:Cl06_2}
\begin{array}{llllllllllllllll}
&J_1w = J_2J_3J_4w = J_2J_5J_6w,\qquad &J_1J_2w = -J_3J_4w = -J_5J_6w, 
\\
&J_2w = -J_1J_3J_4w  = -J_1J_5J_6w,\qquad &J_1J_3w = J_2J_4w,
\\ 
& J_3w = J_1J_2J_4w  = -J_4J_5J_6w,
\qquad &J_1J_4w = -J_2J_3w,
\\ 
&J_4w = -J_1J_2J_3w  = J_3J_5J_6w,
\qquad &J_1J_5w = J_2J_6w, 
\\
&J_5w = J_1J_2J_6w = -J_3J_4J_6w,\qquad &J_1J_6w = -J_2J_5w,
\\
& J_6w = -J_1J_2J_5w = J_3J_4J_5w,
\qquad &J_3J_5w = -J_4J_6w,
\\ 
&&J_3J_6w = J_4J_5w.
\\
 &J_1J_3J_5w = -J_1J_4J_6w = J_2J_3J_6w =J_2J_4J_5w,
 \\ 
 &J_1J_3J_6w = J_1J_4J_5w = -J_2J_3J_5w  =J_2J_4J_6w.
 \end{array}
\end{equation}
We present the simultaneous eigenspace decomposition 
by the involutions $P_1$ and $P_2$ taking into account~\eqref{eq:Cl06Gene} and~\eqref{eq:Cl06_2}:

{\small
\begin{table}[ht]
\begin{center}
\caption{Eigenspace decomposition: $\Cl_{0,6}$ case}
\begin{tabular}{|l|c|c|c|c|} \hline
Involutions&\multicolumn{4}{|c|}{Eigenvalues}\\ \hline
$P_1=J_{1}J_{2}J_{3}J_{4}$&\multicolumn{2}{|c|}{$+1$} &\multicolumn{2}{|c|}{$-1$}\\\hline
$P_2=J_{1}J_{2}J_{5}J_{6}$&\multicolumn{1}{|c|}{$+1$} &\multicolumn{1}{|c|}{$-1$}
&\multicolumn{1}{|c|}{$+1$} &\multicolumn{1}{|c|}{$-1$}\\\hline\hline
{Eigenvector}
&$w,\ J_1J_2w$&$J_1J_3w,\ J_1J_4w$&$J_1J_5w,\ J_1J_6w$& $J_3J_5w,\ J_3J_6w$
\\    
&$J_1J_3J_5w,\ J_1J_3J_6w$&$J_5w,\ J_6w$&$J_3w,\ J_4w$&$J_1w,\ J_2w$
\\\hline
\end{tabular}
\end{center}
\end{table}
}

We choose the vectors 
\begin{equation}\label{eq:basis06}
\begin{array}{lllllllll}
&v_1=w,\quad &v_2=J_1J_2w,\quad & v_3=J_1J_3w,\quad & v_4=J_1J_4w,
\\ 
&v_5=J_1J_5w,& v_6=J_1J_6w, & v_7=J_3J_5w,& v_8=J_3J_6w
\\
&v_9=J_1w,\quad &v_{10}=J_2w,\quad &v_{11}=J_3w,\quad &v_{12}=J_4w,
\\ 
&v_{13}=J_5w, &v_{14}=J_6w, & v_{15}=J_1J_3J_5w, & v_{16}=J_1J_3J_6w,
\end{array}
\end{equation}
satisfying $\langle v_{\alpha},v_{\alpha}\rangle_{V}=1$, $\alpha=1,\ldots,8$ and $\langle v_{\alpha},v_{\alpha}\rangle_{V}=-1$, $\alpha=9,\ldots,16$. Since not all of them can be orthogonal, we apply Lemma~\ref{orthogonal} to operators $\Omega_1=J_1J_3J_5$, and $\Omega_2=J_1J_3J_6$ in order to choose correct vector $w\in E_{1+}\cap E_{2+}\subset V$. We need to use the orthogonalisation only for these operators, since other triplets are linear dependent with $\Omega_1$ and $\Omega_2$. Thus the basis~\eqref{eq:basis06} is orthonormal and 
the action of any $J_j$, $j=1,\ldots, 6$ permutes it 
up to sign, by~\eqref{eq:Cl06_2}. We conclude that the constructed 16 dimensional 
sub-module $(W,\langle \cdot\,,\cdot\rangle_{W})$ is of minimal
dimension. 
Moreover it is irreducible. 

\subsection{Integral structure on admissible $\Cl_{0,7}$-module}
\label{07}

The Clifford algebra $\Cl_{0,7}$ is isomorphic to the space $\mathbb C(8)$.
Let $z_j$, $j=1,\ldots,7$, be generators of $\Cl_{0,7}$ with $\langle
z_j,z_j\rangle_{\mathbb R^{0,7}}=-1$ 
and $(V,\langle\cdot\, ,\cdot\rangle_{V})$ an admissible $\Cl_{0,7}$-module. Then 
$
J_j^2:=J^2_{z_j}=\Id_V$, $j=1,\ldots,7$.
In this case we fix three mutually commuting symmetric isometric involutions
$P_1=J_1J_2J_3J_4$, $P_2=J_1J_2J_5J_6$ and $P_3=J_2J_3J_6J_7$. 
This gives the table of commuting relations with generators and the collection of 
complementary operators.

\begin{center}
{Commutation relations: $\Cl_{0,7}$ case}
\end{center}
\begin{equation}\label{eq:Cl07Gene}
\begin{array}{|l|c|c|c|c|c|c|c|}\hline
\text{Involution}\backslash \text{Generator}&J_{1}&J_{2}&J_{3}&J_{4}&J_{5}&J_{6}&J_{7},\\\hline
P_1=J_{1}J_{2}J_{3}J_{4}&a&a&a&a&c&c&c\\\hline
P_2=J_{1}J_{2}J_{5}J_{6}&a&a&c&c&a&a&c\\\hline
P_3=J_{2}J_{3}J_{6}J_{7}&c&a&a&c&c&a&a\\\hline
\end{array}
\end{equation}

\begin{center}
{Complementary operator}: $\Cl_{0,7}$ case
\end{center}
\begin{equation*}\label{eq:Cl07comp}
\begin{array}{|l|c|c|c|c|}\hline
\text{Involution}\backslash \text{Comp. op.}&J_{1}(+\to-)&J_{5}(+\to-)&J_{7}(+\to-)&J_5J_6J_7(+\to-)\\\hline
P_1=J_{1}J_{2}J_{3}J_{4}&a&c&c&c\\\hline
P_2=J_{1}J_{2}J_{5}J_{6}&&a&c&c\\\hline
P_3=J_{2}J_{3}J_{6}J_{7}&&&a&c\\\hline
\end{array}
\end{equation*}
From these relations, applying Lemma~\ref{lem:PT2}, we can choose a vector $w\in V$ such that
\[
P_1w=w,\quad P_2w=w,\quad P_3w=w,\quad\text{and}\quad\langle w ,w\rangle_{V}=1, 
\]
since the common eigenspaces
of all three involutions are all neutral spaces. 
Then we have a simultaneous eigenspace decomposition of a subspace  $W\subset V$
spanned by the $16$ eigenvectors 
$$
\begin{array}{lllllll}
& v_1=w, \quad & v_2=J_1J_2w,\quad & v_3=J_1J_3w,\quad &v_4=J_1J_4w,
\\
& v_5=J_1J_5w, & v_6=J_1J_6w, & v_7=J_1J_7w, & v_8=J_3J_6w
\\
& v_9=J_1w,\quad & v_{10}=J_2w,\quad &v_{11}=J_3w,\quad &v_{12}=J_4w,
\\ 
& v_{13}=J_5w,\ & v_{14}=J_6w, & v_{15}=J_7w, &v_{16}=J_1J_3J_6w,
\end{array}
$$
satisfying  $\langle v_{\alpha},v_{\alpha}\rangle_{V}=1$, $\alpha=1,\ldots,8$ and $\langle v_{\alpha},v_{\alpha}\rangle_{V}=-1$, $\alpha=9,\ldots,16$.
{\small
\begin{table}[ht]
\begin{center}
\caption{Eigenspace decomposition: $\Cl_{0,7}$ case}
\begin{tabular}{|l|c|c|c|c|c|c|c|c|} \hline
Involutions&\multicolumn{8}{|c|}{Eigenvalues}\\ \hline
$P_1$&\multicolumn{4}{|c|}{$+1$} &\multicolumn{4}{|c|}{$-1$}\\\hline
$P_2$&\multicolumn{2}{|c|}{$+1$}
&\multicolumn{2}{|c|}{$-1$}&\multicolumn{2}{|c|}{$+1$} 
&\multicolumn{2}{|c|}{$-1$}\\\hline
$P_3$&\multicolumn{1}{|c|}{$+1$}
&\multicolumn{1}{|c|}{$-1$}&\multicolumn{1}{|c|}{$+1$} 
&\multicolumn{1}{|c|}{$-1$}
&\multicolumn{1}{|c|}{$+1$} &\multicolumn{1}{|c|}{$-1$}
&\multicolumn{1}{|c|}{$+1$} &\multicolumn{1}{|c|}{$-1$}\\
\hline\hline
{Eigenvectors}
&$w$&$J_1J_2w$&$J_1J_4w$& $J_1J_3w$&$J_1J_5w$&$J_1J_6w$&$J_1J_3J_6w$&$J_1J_7w$
\\
&$J_3J_6w$&$J_7w$&$J_5w$& $J_6w$&$J_4w$&$J_3w$&$J_1w$&$J_2w$
\\\hline
\end{tabular}
\end{center}
\end{table}
}
Where we used~\eqref{eq:Cl07Gene} and the following linear dependence relations.
\begin{align}\label{eq:Cl07_1}
&\begin{array}{llllllllllll}
& J_1w = J_2J_3J_4w = J_2J_5J_6w = J_3J_5J_7w = -J_4J_6J_7w,
\\
&J_2w =  -J_1J_3J_4w = -J_1J_5J_6w = J_3J_6J_7w = J_4J_5J_7w,
\\
& J_3w = J_1J_2J_4w = -J_1J_5J_7w = -J_2J_6J_7w = -J_4J_5J_6w,
\\
& J_4w = -J_1J_2J_3w = J_1J_6J_7w = -J_2J_5J_7w =J_3J_5J_6w,
\\ 
&J_5w = J_1J_2J_6w = J_1J_3J_7w = J_2J_4J_7w = -J_3J_4J_6w,
\\
& J_6w = -J_1J_2J_5w = -J_1J_4J_7w = J_2J_3J_7w = J_3J_4J_5w,
\\
&J_7w = J_1J_4J_6w = -J_1J_3J_5w = -J_2J_3J_6w = -J_2J_4J_5w,
\end{array}
\\
&\begin{array}{llll}
& J_1J_2w = -J_3J_4w = -J_5J_6w,     &J_1J_6w = -J_2J_5w = -J_4J_7w,
\\
& J_1J_3w = J_2J_4w = -J_5J_7w,       &J_1J_7w = -J_3J_5w = J_4J_6w,
\\
& J_1J_4w = -J_2J_3w = J_6J_7w,     & J_3J_6w = -J_2J_7w = J_4J_5w,
\\
& J_1J_5w = J_2J_6w = J_3J_7w, 
\end{array}\nonumber
\end{align}
$$
J_1J_3J_6w=-J_1J_2J_7w=J_1J_4J_5w=-J_2J_3J_5w=J_2J_4J_6w=J_3J_4J_7w=J_5J_6J_7w.
$$

At the final step, we apply Lemma~\ref{orthogonalization} for
the operator $\Omega=J_1J_3J_6w$ to make the basis orthogonal. The relations~\eqref{eq:Cl07_1} shows that the action of $J_i$, $i=1,\ldots,7$ permutes the orthonormal basis $\{v_{\alpha}\}_{\alpha=1}^{16}$ up to sign. We conclude that the constructed 16 dimensional sub-module $(W,\langle \cdot\,,\cdot\rangle_{W})$ is of minimal dimension. Moreover it is irreducible.

\subsection{Integral structure on admissible $\Cl_{0,8}$-module}
\label{08}


The Clifford algebra $\Cl_{0,8}$ is isomorphic to the space $\mathbb R(16)$.
Let $z_j$, $j=1,\ldots,8$, be orthonormal generators of $\Cl_{0,8}$
with 
$\langle z_j,z_j\rangle_{\mathbb R^{0,8}}=-1$ 
and $(V,\langle\cdot,\,\cdot\rangle_{V})$ an admissible $\Cl_{0,8}$-module. Then 
$
J_j^2:=J^2_{z_j}=\Id_V$, $j=1,\ldots,8$.
We fix four isometric mutually commuting involutions
$P_i$:
$$
P_1=J_1J_2J_3J_4,\quad
P_2=J_1J_2J_5J_6,\quad P_3=J_2J_3J_5J_7,\quad\text{and}\quad P_4=J_1J_2J_7J_8.
$$
In this case to find 5 anti-isometric complementary operators is impossible and we chose the different strategy. The tables of commutation relations of the
involutions with the generators and the family of complementary operators are the following:

\begin{center}
{Commutation relations: $\Cl_{0,8}$ case}
\end{center}
\begin{equation*}\label{eq:Cl08Gene}
\begin{array}{|l|c|c|c|c|c|c|c|c|}\hline
\text{Involution}\backslash \text{Generator}&J_{1}&J_{2}&J_{3}&J_{4}&J_{5}&J_{6}&J_{7}&J_{8}
\\\hline
P_1=J_{1}J_{2}J_{3}J_{4}&a&a&a&a&c&c&c&c\\\hline
P_2=J_{1}J_{2}J_{5}J_{6}&a&a&c&c&a&a&c&c\\\hline
P_3=J_2J_3J_5J_7         &c&a&a&c&a&c&a&c\\\hline
P_4=J_1J_2J_7J_8         &a&a&c&c&c&c&a&a\\\hline
\end{array}
\end{equation*}

\begin{center}
{Complementary operator}: $\Cl_{0,8}$ case
\end{center}
{\small
\begin{equation*}\label{eq:Cl08comp}
\begin{array}{|l|c|c|c|c|}\hline
\text{Involution}\backslash \text{Comp. op.}&J_{1}J_5(+\to+)&J_1J_3(+\to+)&J_1J_2(+\to+)&J_{8}(+\to-)
\\\hline
P_1=J_{1}J_{2}J_{3}J_{4}&a&c&c&c\\\hline
P_2=J_{1}J_{2}J_{5}J_{6}& &a&c&c\\\hline
P_3=J_2J_3J_5J_7         & & &a&c\\\hline
P_4=J_1J_2J_7J_8         & & & &a\\\hline
\end{array}
\end{equation*}
}

We apply Lemma~\ref{lem:PT1}, item 1) three times to pairs $(P_1, T_1=J_{1}J_5)$,  $(P_2, T_2=J_{1}J_3)$, and $(P_3, T_3=J_{1}J_2)$ and conclude that the common 1-eigenspace of $P_1,P_2,P_3$ is neutral. We decompose it into two orthogonal eigenspaces $E_{4\pm}$ of $P_4$. If the space $E_{4+}$ is negative definite, we use Lemma~\ref{lemma:04} and modify the operator $P$. From these relations we can choose a vector $w\in V$ such that
$$
P_1w= P_2w= P_3w=P_4w=w\quad\text{and}\quad \langle w ,w\rangle_{V}=1.
$$
We act successively by $J_j$ on $w$ and get a simultaneous eigenspace decomposition spanned by 16 orthogonal vectors 
$$
\begin{array}{lllllll}
& v_1=w, \quad & v_2=J_1J_2w,\quad & v_3=J_1J_3w,\quad &v_4=J_1J_4w,
\\
& v_5=J_1J_5, & v_6=J_1J_6, & v_7=J_1J_7, & v_8=J_1J_8w
\\
& v_9=J_1w,\quad & v_{10}=J_2w,\quad &v_{11}=J_3w,\quad &v_{12}=J_4w,
\\ 
& v_{13}=J_5w,\ & v_{14}=J_6w, & v_{15}=J_7w, &v_{16}=J_8w,
\end{array}
$$
with $\langle v_{\alpha},v_{\alpha}\rangle_{V}=1$, $\alpha=1,\ldots,8$ and $\langle v_{\alpha},v_{\alpha}\rangle_{V}=-1$, $\alpha=9,\ldots,16$. 

{\small
\begin{table}[ht]
\begin{center}
\caption{Eigenspace decomposition: $\Cl_{0,8}$ case}
\begin{tabular}{|l|c|c|c|c|c|c|c|c|} \hline
\text{Involution}&\multicolumn{8}{|c|}{\text{Eigenvalue}}\\\hline
$P_1$&\multicolumn{8}{|c|}{$+1$}\\\hline
$P_2$&\multicolumn{4}{|c|}{$+1$}&\multicolumn{4}{|c|}{$-1$}\\\hline
$P_3$&\multicolumn{2}{|c|}{$+1$} &\multicolumn{2}{|c|}{$-1$}
&\multicolumn{2}{|c|}{$+1$} &\multicolumn{2}{|c|}{$-1$}\\\hline
$P_4$&\multicolumn{1}{|c|}{$+1$} &\multicolumn{1}{|c|}{$-1$}
&\multicolumn{1}{|c|}{$+1$} &\multicolumn{1}{|c|}{$-1$}
&\multicolumn{1}{|c|}{$+1$} &\multicolumn{1}{|c|}{$-1$}
&\multicolumn{1}{|c|}{$+1$} &\multicolumn{1}{|c|}{$-1$}\\\hline\hline
{Eigenvectors}&$w$&$J_8w$&$J_1J_2w$&
$J_7w$&$J_6w$&$J_1J_4w$&$J_5w$&$J_1J_3w$
\\
\hline
\end{tabular}
\end{center}
\end{table}
}
{\small
\begin{table}[h]
\begin{center}
\begin{tabular}{|l|c|c|c|c|c|c|c|c|} \hline
\text{Involution}&\multicolumn{8}{|c|}{\text{Eigenvalue}}\\\hline
$P_1$&\multicolumn{8}{|c|}{$-1$}\\\hline
$P_2$&\multicolumn{4}{|c|}{$+1$}&\multicolumn{4}{|c|}{$-1$}\\\hline
$P_3$&\multicolumn{2}{|c|}{$+1$} &\multicolumn{2}{|c|}{$-1$}
&\multicolumn{2}{|c|}{$+1$} &\multicolumn{2}{|c|}{$-1$}\\\hline
$P_4$&\multicolumn{1}{|c|}{$+1$} &\multicolumn{1}{|c|}{$-1$}
&\multicolumn{1}{|c|}{$+1$} &\multicolumn{1}{|c|}{$-1$}
&\multicolumn{1}{|c|}{$+1$} &\multicolumn{1}{|c|}{$-1$}
&\multicolumn{1}{|c|}{$+1$} &\multicolumn{1}{|c|}{$-1$}\\\hline\hline
{Eigenvectors}
&$J_4w$&$J_1J_6w$&$J_3w$& $J_1J_5w$&$J_1J_8w$&$J_1w$&$J_1J_7w$&$J_2w$\\\hline
\hline
\end{tabular}
\end{center}
\end{table}
}
The presence of new involution $P_4$ eliminate additional linear dependent vectors and as in the previous cases it can be shown that generators acts by permutation on the basis. Therefore, we constructed an integral structure
of a minimal admissible and irreducible $\Cl_{0,8}$-module. For the completeness we present all linear relations up to sign.
\begin{equation}\label{eq:Cl08_all}
\begin{array}{lll}
& J_1J_2=\pm J_3J_4=\pm J_5J_6=\pm J_7J_8,
\\
& J_1J_3=\pm J_2J_4=\pm J_5J_8=\pm J_6J_7,
\\
& J_1J_4=\pm J_2J_3=\pm J_5J_7=\pm J_6J_8,
\\
& J_1J_5=\pm J_2J_6=\pm J_3J_8=\pm J_4J_7,
\\
& J_1J_6=\pm J_2J_5=\pm J_3J_7=\pm J_4J_8,
\\
& J_1J_7=\pm J_2J_8=\pm J_3J_6=\pm J_4J_5,
\\
& J_1J_8=\pm J_2J_7=\pm J_3J_5=\pm J_4J_6,
\end{array}
\end{equation}
$$
\begin{array}{lll}
& J_1=\pm J_2J_3J_4=\pm J_2J_5J_6=\pm J_2J_7J_8=\pm J_3J_5J_8=\pm J_3J_6J_7=\pm J_4J_5J_7=\pm J_4J_6J_8,
\\
& J_2=\pm J_1J_3J_4=\pm J_1J_5J_6=\pm J_1J_7J_8=\pm J_3J_5J_7=\pm J_3J_6J_8=\pm J_4J_5J_8=\pm J_4J_6J_7,
\\
& J_3=\pm J_1J_2J_4=\pm J_1J_5J_8=\pm J_1J_6J_8=\pm J_2J_5J_7=\pm J_2J_6J_8=\pm J_4J_5J_6=\pm J_4J_7J_8,
\\
& J_4=\pm J_1J_2J_3=\pm J_1J_5J_7=\pm J_1J_6J_8=\pm J_2J_5J_8=\pm J_2J_6J_7=\pm J_3J_5J_6=\pm J_3J_7J_8,
\\
& J_5=\pm J_1J_2J_6=\pm J_1J_3J_8=\pm J_1J_4J_7=\pm J_2J_3J_7=\pm J_2J_4J_8=\pm J_3J_4J_6=\pm J_6J_7J_8,
\\
& J_6=\pm J_1J_2J_5=\pm J_1J_3J_7=\pm J_1J_4J_8=\pm J_2J_3J_8=\pm J_2J_4J_7=\pm J_3J_4J_5=\pm J_5J_7J_8,
\\
& J_7=\pm J_1J_2J_8=\pm J_1J_3J_6=\pm J_1J_4J_5=\pm J_2J_3J_5=\pm J_2J_4J_6=\pm J_3J_4J_8=\pm J_5J_6J_8,
\\
& J_8=\pm J_1J_2J_7=\pm J_1J_3J_5=\pm J_1J_4J_6=\pm J_2J_3J_6=\pm J_2J_4J_5=\pm J_3J_4J_7=\pm J_5J_6J_7.
\end{array}
$$

\section{Integral structure on admissible $\Cl_{r,1}$-modules}
\label{sec:r0 and r1}


In this section we show the existence of the integral structure for $\Cl_{r,1}$-modules based on the existence of those for $\Cl_{r,0}$-modules. 
In the next section we will deal with the classical cases $\Cl_{r,0}$ directly.

Let $(V,\langle\cdot\,,\cdot\rangle_V)$ be an admissible
$\Cl_{r,1}$-module. Denote by $J_i=J_{z_i}$, $i=1,\ldots, r+1$ 
the representations of orthonormal generators of the algebra $\Cl_{r,1}$ such that
$$
J_i^2=J_{z_i}^2=-\Id_V,\quad i=1,\ldots,r,\qquad J_{r+1}^2=J_{z_{r+1}}^2=\Id_V.
$$ 
Let $U$ be a subspace of $V$ invariant under the action of
$J_i=J_{z_i}$, $i=1,\ldots, r$, 
and let $U=\spn\{u_1,\ldots,u_l\}$,  be a basis such that 
\begin{itemize}
\item{$\{u_{\alpha}\}_{\alpha=1}^{l}$ is an orthonormal basis of $U$;}
\item{maps $J_i=J_{z_i}$, $i=1,\ldots, r$, permute the basis $\{u_{\alpha}\}_{\alpha=1}^{l}$ up to sign;}
\item{$\langle J_iu_{\alpha},u_{\beta}\rangle_V=\pm 1$ or $0$.}
\end{itemize}
Denote $\tilde u_{\alpha}=J_{r+1}u_{\alpha}$, $\alpha=1,\ldots, l$, $\widetilde U=\spn\{\tilde u_1,\ldots,\tilde u_l\}$, and set the space $$W=\spn\{u_1,\ldots,u_l,\tilde u_1,\ldots,\tilde u_l\}.$$
Then
\begin{theorem}\label{prop:r0tor1}
In the notations above if the decomposition $W=U\oplus\widetilde U$ is orthogonal with respect to $\langle\cdot\,,\cdot\rangle_V$, then $(W,\langle\cdot\,,\cdot\rangle_W)$, where $\langle\cdot\,,\cdot\rangle_W$ is the restriction of $\langle\cdot\,,\cdot\rangle_V$ onto $W$ is an admissible integral $\Cl_{r,1}$-module. If $U$ has minimal dimension, then $W$ is an admissible integral $\Cl_{r,1}$-module of minimal dimension.
\end{theorem}
\begin{proof}
Observe that the restriction of $\langle\cdot\,,\cdot\rangle_V$ to $U$ will be positive definite or negative definite since
$$
\langle J_iu_{\alpha},J_iu_{\alpha}\rangle_V=\langle z_i,z_i\rangle_{\mathbb R^{r,1}}\langle u_{\alpha},u_{\alpha}\rangle_V=\langle u_{\alpha},u_{\alpha}\rangle_V,\quad i=1,\ldots,r. 
$$
Let us assume that it is positive definite. 
Then $\widetilde U=\spn\{\tilde u_{1},\ldots,\tilde u_{l}\}$ is isomorphic to $U$ and the restriction of $\langle\cdot\,,\cdot\rangle_V$ on $\widetilde U$ is negative definite since
$$
\langle \tilde u_{\alpha},\tilde u_{\alpha}\rangle_V=\langle J_{r+1}u_{\alpha},J_{r+1}u_{\alpha}\rangle_V=\langle z_{r+1},z_{r+1}\rangle_{\mathbb R^{r,1}}\langle u_{\alpha},u_{\alpha}\rangle_V=-\langle u_{\alpha},u_{\alpha}\rangle_V.
$$
We conclude that $J_{r+1}\colon U\to\widetilde U$ defines an
anti-isometry with respect to restrictions of the scalar product
$\langle\cdot\,,\cdot\rangle_V$ onto spaces $U$ and $\widetilde U$ and
the space $(W,\langle\cdot\,,\cdot\rangle_W)$, where
$W=U\oplus\widetilde U$ and $\langle\cdot\,,\cdot\rangle_W$ 
is the restriction of $\langle\cdot\,,\cdot\rangle_V$ onto $W$ is neutral.  

The space $W$ is invariant under the action of $J_i$, $i=1,\ldots,r+1$, and all maps $J_i$, $i=1,\ldots,r+1$ permute the basis $\{v_{\alpha},\tilde v_{\alpha}\}_{\alpha=1}^{l}$ up to sign. We conclude that the sub-module $(W,\langle\cdot\,,\cdot\rangle_W)$ is admissible integral and has a minimal possible dimension if the space $U$ has minimal dimension.
\end{proof}
\begin{corollary}\label{cor:2}
The dimension of the minimal admissible module for $\Cl_{r,1}$ is always twice
of the $\Cl_{r,0}$-module.
\end{corollary}


\section{Integral structure on admissible $\Cl_{r,s}$-modules with $r+s\leq 8$}
\label{sec:rs}


We show the existence of an admissible integral
$\Cl_{r,s}$-module for $r+s=2,\ldots,8$ by direct construction. In these constructions we start
from an admissible module $(V,\langle\cdot\,,\cdot\rangle_{V})$ and
find an orthonormal set $\{v_{\alpha}\}_{\alpha=1}^{l}$ such that the
sub-module $(W,\langle\cdot\,,\cdot\rangle_{W})$, where
$W=\spn\{v_1,\ldots,v_l\}$ and the scalar product
$\langle\cdot\,,\cdot\rangle_{W}$ is the restriction
of $\langle\cdot\,,\cdot\rangle_{V}$ onto $W$, is an
admissible integral and has the minimal possible dimension. 
In some cases $W$ is not an irreducible sub-module. 


\subsection{Integral structure on $\Cl_{r,s}$-admissible modules with $r+s=2$}\label{111}



\subsubsection{Integral structure on admissible $\Cl_{2,0}$-module}\label{20}


The Clifford algebra $\Cl_{2,0}$ is isomorphic to the space $\mathbb H$.
Let $z_1,z_2$ be orthonormal generators of $\Cl_{2,0}$ with
$( z_i,z_i)_{\mathbb R^{2,0}}=1$, $i=1,2$. 
Let $(V,(\cdot\,,\cdot)_{V})$ be an admissible $\Cl_{2,0}$-module with positive definite inner product, then 
$
J_i^2:=J^2_{z_i}=-( z_i,z_i)_{\mathbb R^{2,0}}\Id_V=-\Id_V$, for $i=1,2$.
Choose $w\in V$ such that $(w,w)_{V}=1$. Then the basis
\begin{equation}\label{eq:basis11}
v_1=w,\quad v_2=J_1w,\quad v_3=J_2w,\quad v_4=J_2J_1w
\end{equation}  
is orthogonal and satisfies 
$
(v_\alpha,v_\alpha)_{V}=1$ for $\alpha=1,2,3,4$.
It is easy to see that the admissible module is integral.  The
admissible module is irreducible as in all case corresponding to $\Cl_{r,0}$ algebras. 


\subsubsection{Integral structure on admissible $\Cl_{1,1}$-module}\label{11}


The Clifford algebra $\Cl_{1,1}$ is isomorphic to the space $\mathbb R(2)$ of $(2\times 2)$-matrices with real entries.
Let $z_1,z_2$ be orthonormal generators of $\Cl_{1,1}$ with
$\langle z_1,z_1\rangle_{\mathbb R^{1,1}}=1$ and $\langle
z_2,z_2\rangle_{\mathbb R^{1,1}}=-1$. 
Let $(V,\langle\cdot\,,\cdot\rangle_{V})$ be an admissible $\Cl_{1,1}$-module, then 
$$
J_1^2:=J^2_{z_1}=-\langle z_1,z_1\rangle_{\mathbb R^{1,1}}\Id_V=-\Id_V,\quad J_2^2:=J_{z_2}^2=-\langle z_2,z_2\rangle_{\mathbb R^{1,1}}\Id_V=\Id_V.
$$
Choose $w\in V$ such that $\langle w,w\rangle_{V}=1$. Then the basis~\eqref{eq:basis11}
is orthogonal and satisfies 
$
\langle v_1,v_1\rangle_{V}=\langle v_2,v_2\rangle_{V}=-\langle v_3,v_3\rangle_{V}=-\langle v_4,v_4\rangle_{V}=1$.
It is easy to see that the admissible module is integral.  The
admissible module is not irreducible in this case since the dimension of an irreducible module is $2$. 

{\it Notice that we could show the existence of the integral structure by making use of the isomorphism
$
\Cl_{r,s+1}=\Cl_{0,2}\cong\Cl_{1,1}=\Cl_{s,r+1}$, see Theorem~\ref{prop:7},
but we prefer the direct construction, since we will use this construction further.}


\subsubsection{Integral structure on admissible $\Cl_{0,2}$-module}\label{02}


The admissible integral $\Cl_{0,2}$-module was constructed in Section~\ref{sec:0n}.


\subsection{Integral structure on admissible $\Cl_{r,s}$-modules with $r+s=3$}\label{301}



\subsubsection{Integral structure on admissible $\Cl_{3,0}$-module}\label{30}


The Clifford algebra $\Cl_{3,0}$ is isomorphic to the space 
$\mathbb H\oplus \mathbb H$.
Let $z_1,z_2,z_3$ be orthogonal generators of $\Cl_{3,0}$ 
with $(z_i,z_i)_{\mathbb R^{3,0}}=1$, $i=1,2,3$,  
and $(V,(\cdot\, ,\cdot)_{V})$ an admissible $\Cl_{3,0}$-module with positive definite metric. Then 
$$
J_i^2:=J^2_{z_i}=-(z_i,z_i)_{\mathbb R^{3,0}}\Id_V=-\Id_V.
$$
Consider the isometric involution $P=J_1J_2J_3$ and pick up a vector $w\in V$ such that $J_1J_2J_3w=w$ and
$(w,w)_{V}=1$. It is possible, since the 1-eigenspace is an inner product space.
We get the following linear dependent vectors
\begin{equation}\label{eq:30}
J_1w=-J_2J_3w,\quad J_2w=J_1J_3w,\quad J_3w=J_1J_2w.
\end{equation}
In this case we choose the basis 
$v_1=w$, $v_2=J_1w$, $v_3=J_2w$, $v_4=J_3w$. 
It is orthogonal due to the skew symmetry of $J_i$, $i=1,2,3$.
Moreover $(v_\alpha,v_\alpha)_{V}=1$, $\alpha=1,2,3,4$.
It is easy to see that the sub-module $W=\spn\{v_1,\ldots,v_4\}$, 
is integral due to the orthogonality of generators, skew symmetry of
$J_k$ and the condition $J_1J_2J_3w=w$. The admissible integral
sub-module is irreducible
with an inner product.


\subsubsection{Integral structure on admissible $\Cl_{2,1}$-module}\label{21}


The Clifford algebra $\Cl_{2,1}$ is isomorphic to the space $\mathbb C(2)$ of $2\times 2$ matrices with complex entries.
Let $z_1,z_2,z_3$ be generators of $\Cl_{2,1}$ with $\langle z_i,z_i\rangle_{\mathbb R^{2,1}}=1$, $i=1,2$ and $\langle z_3,z_3\rangle_{\mathbb R^{2,1}}=-1$. Let $(V,\langle\cdot\, ,\cdot\rangle_{V})$ be an admissible $\Cl_{2,1}$-module. Then 
$
J_i^2:=J^2_{z_i}=-\Id_V$, $i=1,2$, $J_3^2:=J^2_{z_3}=\Id_V$.
We apply Theorem~\ref{prop:r0tor1} and chose the following basis 
$$
\begin{array}{lllll}
&v_1=w, \quad &v_2=J_1w,\quad &v_3=J_2w,\quad &v_4=J_1J_2w,
\\
&v_5=J_3w,\quad &v_6=J_1J_3w,\quad &v_7=J_2J_3w,\quad &v_8=J_1J_2J_3w.
\end{array}
$$  
If the vectors $w$ and $J_1J_2J_3w$ are not orthogonal, we apply Lemma~\ref{orthogonalization} and find the correct vector $\widetilde w$ making basis orthogonal.
Moreover 
$
\langle v_\alpha,v_\alpha\rangle_{V}=1$, $\alpha=1,\ldots,4$, $\langle v_\alpha,v_\alpha\rangle_{V}=-1$, $\alpha=5,\ldots,8$.
It is easy to see that the module is integral since $J_i$, $i=1,2,3$, permute the basis up to sign. The admissible integral sub-module is not irreducible since the dimension of the irreducible module is equal to $4$.

\begin{remark}
Let $v\in V$ be a non-null vector and assume that one of $J_iJ_jv$ with $i<j$ is
independent from $v$, $J_1v $, $J_2v$, $J_3v$. Then the dimension of the admissible subspace
is at least 8. Together with the construction above the irreducible
module can not be admissible with any choice of non-null vector.
\end{remark}

We also can use isomorphism $\Cl_{0,3}\cong\Cl_{2,1}$ and Theorem~\ref{prop:7} to prove the existence of integral structure on admissible $\Cl_{2,1}$-module.


\subsubsection{Integral structure on admissible $\Cl_{1,2}$-module}\label{12}


The Clifford algebra $\Cl_{1,2}$ is isomorphic to the space $\mathbb R(2)\times\mathbb R(2)$.
Let $z_1,z_2,z_3$ be generators of $\Cl_{1,2}$ with $\langle z_1,z_1\rangle_{\mathbb R^{1,2}}=1$, $\langle z_i,z_i\rangle_{\mathbb R^{1,2}}=-1$, $i=2,3$, and $(V,\langle\cdot\, ,\cdot\rangle_{V})$ an admissible $\Cl_{1,2}$-module. Then 
$$
J_1^2:=J^2_{z_1}=-\Id_V,\qquad
J_i^2:=J^2_{z_i}=\Id_V \quad i=2,3.
$$
Consider the isometric involution $P=J_1J_2J_3$. The operator $P$ commutes with the anti-isometry $T=J_2$. Thus the eigenspace $E_+$ of $P$ is a neutral space by Lemma~\ref{lem:PT1}, case 3).  We pick up a vector $w\in V$ such that $Pw=w$ and $\langle w,w\rangle_{V}=1$. 
It gives the linear dependent vectors~\eqref{eq:30}.
In this case we choose the basis $v_1=w$, $v_2=J_1w$, $v_3=J_2w$, $v_4=J_3w$, that is orthogonal as it was shown in the case of $\Cl_{3,0}$-module. Moreover 
$
\langle v_1,v_1\rangle_{V}=\langle v_2,v_2\rangle_{V}=-\langle v_3,v_3\rangle_{V}=-\langle v_4,v_4\rangle_{V}=1$.
It is easy to see that the module is integral due to the orthogonality
of generators, skew symmetry of $J_k$ and the condition
$J_1J_2J_3w=w$. The admissible integral module is not irreducible 
since the dimension of the irreducible module is equal to~$2$. 


\subsubsection{Integral structure on admissible $\Cl_{0,3}$-module}\label{03}


The admissible integral $\Cl_{0,3}$-module was constructed in Section~\ref{sec:0n}.


\subsection{Integral structure on admissible $\Cl_{r,s}$-modules with $r+s=4$}\label{401}



\subsubsection{Integral structure on admissible $\Cl_{4,0}$-module}\label{40}


The Clifford algebra $\Cl_{4,0}$ is isomorphic to the space $\mathbb H(2)$.
Let $z_1,z_2,z_3,z_4$ be orthogonal 
generators of $\Cl_{4,0}$ with $(z_i,z_i)_{\mathbb R^{4,0}}=1$, 
$i=1,2,3,4$, and let $(V,(\cdot\, ,\cdot)_{V})$ 
be an admissible $\Cl_{4,0}$-module with positive definite metric. Then 
$J_i^2:=J^2_{z_i}=-\Id_V$, $i=1,2,3,4$.
As in the case of $\Cl_{0,4}$-module we consider the isometric involution $P=J_1J_2J_3J_4$ and pick up
$w\in V$ such that $Pw=w$ and $(w,w)_{V}=1$. 
It leads to the linear dependence relations similar to~\eqref{eq:Cl04_2}, where some of them can change sign.
The basis~\eqref{basis 04} is orthonormal with respect to $(\cdot\,
,\cdot)_{V}$. The maps $J_i$, $i=1,2,3,4$, permute the basis up to
sign. The admissible integral sub-module is irreducible with the inner
product.


\subsubsection{Integral structure on admissible $\Cl_{3,1}$-module}\label{secCl31}


The Clifford algebra $\Cl_{3,1}$ is isomorphic to the space $\mathbb H(2)$.
The integral structure on 
admissible $\Cl_{3,1}$-modules exists 
according to Theorem~\ref{prop:7} and Corollary~\ref{cor:3} or Theorem~\ref{prop:r0tor1}.


\subsubsection{Integral structure on admissible $\Cl_{2,2}$-module}\label{22}


The Clifford algebra $\Cl_{2,2}$ is isomorphic to the space $\mathbb R(4)$.
Let $z_1,z_2,z_3,z_4$ be orthogonal 
generators of $\Cl_{2,2}$ with 
$\langle z_i,z_i\rangle_{\mathbb R^{2,2}}=1$, 
$i=1,2$, $\langle z_j,z_j\rangle_{\mathbb R^{2,2}}=-1$, $j=3,4$, 
and $(V,\langle\cdot\, ,\cdot\rangle_{V})$ an admissible $\Cl_{2,2}$-module. Then 
$$
J_i^2:=J^2_{z_i}=-\Id_V, \quad i=1,2,\quad J_j^2:=J^2_{z_j}=\Id_V,\quad j=3,4.
$$
We consider the isometric involution $P=J_1J_2J_3J_4$ and the isometry $T=J_1$ such that $PT=-TP$. 
We choose a vector $w\in V$ such
that $Pw=w$ and $\langle w ,w\rangle_{V}=1$ by Lemma~\ref{lem:PT1}, part 1). 
The linear dependence relations~\eqref{eq:Cl04_2} still hold up to sign.
We write the basis~\eqref{basis 04} in the form
$$
\begin{array}{lllllll}
&v_1=w,\quad &v_2=J_1w,\quad &v_3=J_2w,\quad &v_4=J_1J_2w,
\\
&v_5=J_3w,\quad &v_6=J_4w,\quad &v_7=J_1J_3w,\quad &v_8=J_1J_4w.
\end{array}
$$
It is orthonormal and 
$
\langle v_\alpha,v_\alpha\rangle_{V}=1$, $\alpha=1,2,3,4$, $\langle v_\alpha,v_\alpha\rangle_{V}=-1$, $ \alpha=5,6,7,8$.
Since $J_i$, $i=1,2,3,4$, permute basis vectors up to sign the
sub-module $W=\spn\{v_1,\ldots,v_8\}$ is integral. The constructed
module is not irreducible because the dimension of an irreducible
module is $4$ and with any choice of a non-null vector $v\in V$ the five
vectors $v$, $J_1v$, $J_2v$, $J_3v$, $J_4v$ are already linear independent.

We give an alternative construction of an admissible $\Cl_{2,2}$-module. 
As it was mentioned the $\pm 1$-eigenspaces $E_{\pm}$ of $P$
have equal dimension. Consider two operators $\hat J_1=J_1J_2$ and $\hat J_2=J_1J_3$. Since they commute with $P$ they leave invariant $E_+$. Moreover
$$
\hat J_1^2=-\Id_{E_+},\ \ \hat J_2^2=\Id_{E_+},\qquad \hat J_1\hat J_2=-\hat J_2\hat J_1.
$$
Thus, the algebra generated by
$\hat J_1$ and $\hat J_2$ in $\End(E_+)$ is isomorphic to the
Clifford algebra $\Cl_{1,1}$ and the representation is admissible,
since, for example,
\[
\langle\hat J_1u,v\rangle_{E_+}  = \langle J_1J_2u,v\rangle_{E_+}=-\langle J_2u,J_1v\rangle_{E_+}=\langle u,J_2J_1v\rangle_{E_+}=-\langle u,J_1J_2v\rangle_{E_+}
=-\langle u,\hat J_1v\rangle_{E_+}.
\]
The same arguments valid for $E_-$.
Because $\dim(E_+)=4$ and $E_{+}\perp E_-$, we have an integral structure on $V=E_+\oplus_{\perp}E_-$ inherited from that of $\Cl_{1,1}$. 


\subsubsection{Integral structure on admissible $\Cl_{1,3}$-module}
\label{13}


The Clifford algebra $\Cl_{1,3}$ is isomorphic to the space $\mathbb R(4)$.
In this case we can use Theorem~\ref{prop:7} since the Clifford algebra $\Cl_{r,s+1}=\Cl_{2,2}$ is isomorphic to $\Cl_{s,r+1}=\Cl_{1,3}$. The orthogonal basis changed to
$$
\begin{array}{lllllll}
&v_1=w,\quad &v_2=J_2J_4w,\quad &v_3=J_3J_4w,\quad &v_4=-J_2J_3w,
\\
&v_5=J_1J_4w,\quad &v_6=J_4w,\quad &v_7=J_1J_2w,\quad &v_8=J_2w,
\end{array}
$$
with
$
\langle v_\alpha,v_\alpha\rangle_{V}=1$, $\alpha=1,2,3,4$, $\langle v_\alpha,v_\alpha\rangle_{V}=-1$, $ \alpha=5,6,7,8$.
Then the constructed sub-module $W=\spn\{v_1,\ldots,v_8\}$ is admissible and
integral, but is not irreducible because the dimension of an
irreducible module is $4$.


\subsubsection{Integral structure on admissible $\Cl_{0,4}$-module}


The admissible integral $\Cl_{0,4}$-module was constructed in Section~\ref{sec:0n}.


\subsection{Integral structure on admissible $\Cl_{r,s}$-modules with $r+s=5$}
\label{sec:Cl051}



\subsubsection{Integral structure on admissible $\Cl_{5,0}$-module}
\label{sec:Cl05}


The Clifford algebra $\Cl_{5,0}$ is isomorphic to the space $\mathbb C(4)$.
Let $z_1,\dots,z_5$ be orthonormal 
generators of $\Cl_{5,0}$ with 
$(z_i,z_i)_{\mathbb{R}^{5,0}}=1$, $i=1,\ldots 5$, 
and $(V,\langle\cdot\,,\cdot\rangle_{V})$ an admissible $\Cl_{5,0}$-module with a neutral product. Then 
$
J_i^2:=J^2_{z_i}=-\Id_V$ for $i=1,\ldots,5$.
In this case we fix two mutually commuting isometric involutions
$P_1=J_1J_2J_3J_4$ and $P_2=J_1J_2J_5$.
Since we can find complementary isometries $T_1=J_1$ and
$T_2=J_2J_3$ which
satisfy the relations 
\begin{equation}\label{eq:PT50}
P_1T_1=-T_1P_1,\quad P_1T_2=T_2P_1\qquad P_2T_2=-T_2P_2,
\end{equation}
we may pick up a vector $w\in V$ such that $P_1w=P_2w=w$ 
and $\langle w,w\rangle_{V}=1$ by applying Lemma~\ref{lem:PT1}, part 1) twice. 
Then the simultaneous eigenspace decomposition, presented in ????? 5,
{\small
\begin{table}[h]
\begin{center}
\caption{Eigenspace decomposition: $\Cl_{5,0}$ case}
\begin{tabular}{|l|c|c|c|c|}
\hline
\text{Involution}&\multicolumn{4}{|c|}{\text{Eigenvalue}}\\
\hline
$P_1$&\multicolumn{2}{|c|}{$+1$}&\multicolumn{2}{|c|}{$-1$}\\
\hline
$P_2$&\multicolumn{1}{|c|}{$+1$}&\multicolumn{1}{|c|}{$-1$}
&\multicolumn{1}{|c|}{$+1$}&\multicolumn{1}{|c|}{$-1$}\\\hline\hline
{Eigenvectors}
&$w, J_5w$&$J_1J_3w, J_1J_4w$& $J_1w, J_2w$&$J_3w,J_4w$\\\hline
\end{tabular}
\end{center}
\end{table}
}
allows to chose the orthonormal basis 
$$
\begin{array}{llllllll}
&v_1=w,\quad &v_2=J_1w,\quad &v_3=J_2w,\quad &v_4=J_3w,
\\
&v_5=J_4w,\quad &v_6=J_5w,\quad &v_7=J_1J_3w,\quad &v_8=J_1J_4w,
\end{array}
$$
with $\langle v_\alpha,v_\alpha\rangle_V=1$, $\alpha=1,\ldots,8$.
Relations~\eqref{eq:Cl04_2} and those coming from $J_1J_2J_5w=w$ shows that the action of $J_j$, $j=1\ldots,5$, permutes basis up to sign.
Hence we have an integral structure of a minimal dimensional
admissible sub-module of $\Cl_{5,0}$-module with an inner product, see Remark~\ref{rem:product} and an observation in Section~\ref{fremark}. The admissible module is irreducible.


\subsubsection{Integral structure on admissible $\Cl_{4,1}$-module}
\label{41}


The admissible $\Cl_{4,1}$-module is integral by
Theorem~\ref{prop:r0tor1}. If the decomposition on $U$ spanned
by~\eqref{basis 04} and the image $\widetilde J_5(U)$ is not
orthogonal, 
then we change the vector $w$ to $\widetilde w$ by
Lemma~\ref{orthogonal} and Corollary~\ref{cor:orthogonal}. We also can
use Theorem~\ref{prop:7} and Corollary~\ref{cor:3} 
to show the existence of an integral structure.


\subsubsection{Integral structure on admissible $\Cl_{3,2}$-module}


The Clifford algebra $\Cl_{3,2}$ is isomorphic to the space $\mathbb C(4)$.
Let $z_1,\dots,z_5$ be orthogonal 
generators of $\Cl_{3,2}$, and $(V,\langle\cdot\,,\cdot\rangle_{V})$ an admissible $\Cl_{3,2}$-module. Then 
$J_i^2:=J^2_{z_i}=-\Id_V$, $i=1,2,3$, $J_j^2:=J^2_{z_j}=\Id_V$, for $j=4,5$.
The construction is essentially the same as for $\Cl_{5,0}$-module. 
In this case we choose the mutually commuting isometric 
involutions $P_1=J_2J_3J_4J_5$ and $P_2=J_1J_2J_3$ and 
complementary isometries $T_1=J_2$ and $T_2=J_2J_4$, that satisfy relations~\eqref{eq:PT50}. 
We pick up  a vector $w\in V$ such that $P_1w=P_2w=w$ and $\langle
w,w\rangle_{V}=1$, which existence is guaranteed by Lemma~\ref{lem:PT1}, part1). 
The orthonormal basis is
$$
\begin{array}{lllll}
&v_1=w,\quad &v_2=J_1w,\quad &v_3=J_2w,\quad &v_4=J_3w,
\\
&v_5=J_4w,\quad &v_6=J_5w,\quad &v_7=J_2J_4w,\quad &v_8=J_2J_5w,
\end{array}
$$
with $\langle v_{\alpha},v_{\alpha}\rangle_{V}=1$, $\alpha=1,\ldots,4$ and $\langle v_{\alpha},v_{\alpha}\rangle_{V}=-1$, $\alpha=5,\ldots,8$. In this case we will have enough linear dependent relations showing that action of $J_i$, $i=1,\ldots,5$, permutes basis vectors up to sign. These relations are analogous to~\eqref{eq:Cl04_2} and those arising from $P_2w=w$. The constructed module is admissible, integral, and irreducible.


\subsubsection{Integral structure on admissible $\Cl_{2,3}$-module}


The Clifford algebra $\Cl_{2,3}$ is isomorphic to the space $\mathbb R(4)\oplus\mathbb R(4)$.
Let $z_1,\dots,z_5$ be generators of $\Cl_{2,3}$, 
and $(V,\langle\cdot\,,\cdot\rangle_{V})$ an admissible $\Cl_{2,3}$-module. Then 
$
J_i^2:=J^2_{z_i}=-\Id_V$, $i=1,2$, $J_j^2:=J^2_{z_j}=\Id_V$, $j=3,4,5$.

We fix two mutually commuting isometric involutions $P_1=J_1J_2J_3J_4$ and
$P_2=J_1J_4J_5$ and two complementary isometries 
$T_1=J_1$ and $T_2=J_1J_2$ which satisfy~\eqref{eq:PT50}. The
common eigenspaces of $P_1$ and $P_2$ are neutral spaces by Lemma~\ref{lem:PT1}, part 1). So 
we find a common eigenvector $w\in E_{1+}\cap E_{2+}\subset V$ such that $P_1w=P_2w=w$ with $\langle w,w\rangle_V=1$.
We have the following
simultaneous eigenspace decomposition of the representation space $V$ by $P_1$ and $P_2$ with mutually
orthogonal eigenvectors:
{\small
\begin{table}[h]
\begin{center}
\caption{Eigenspace decomposition: $\Cl_{2,3}$ case}
\begin{tabular}{|l|c|c|c|c|}
\hline
\text{Involution}&\multicolumn{4}{|c|}{\text{Eigenvalue}}\\
\hline
$P_1$&\multicolumn{2}{|c|}{$+1$}&\multicolumn{2}{|c|}{$-1$}\\
\hline
$P_2$&\multicolumn{1}{|c|}{$+1$}&\multicolumn{1}{|c|}{$-1$}
&\multicolumn{1}{|c|}{$+1$}&\multicolumn{1}{|c|}{$-1$}\\\hline\hline
{Eigenvectors}
&$w, J_1J_4w$&$J_5w, J_1J_2w$& $J_1w, J_4w$&$J_2w, J_3w$\\\hline
\end{tabular}
\end{center}
\end{table}
}
It gives the orthonormal basis 
$$
\begin{array}{lllll}
&v_1=w,\quad &v_2=J_1w,\quad &v_3=J_2w,\quad &v_4=J_1J_2w,
\\
&v_5=J_3w,\quad &v_6=J_4w,\quad &v_7=J_5w,\quad &v_8=J_1J_4w,
\end{array}
$$
with $\langle v_{\alpha},v_{\alpha}\rangle_{V}=1$, $\alpha=1,\ldots,4$ and $\langle v_{\alpha},v_{\alpha}\rangle_{V}=-1$, $\alpha=5,\ldots,8$.
As in previous cases we can show that $J_i$, $i=1,\ldots,5$ permute basis vectors up to sign.
The constructed module is admissible integral, but not irreducible because
we know that with any choice of non-null vector $v\in V$ the 6 vectors 
$v$, $J_1v$, $J_2v$, $J_3v$, $J_4v$, $J_5v$ are linearly independent but the
dimension of the irreducible module is $4$.


\subsubsection{Integral structure on admissible $\Cl_{1,4}$-module}


The Clifford algebra $\Cl_{1,4}$ is isomorphic to the space $\mathbb
C(4)$.
We can apply Theorem~\ref{prop:7} and Corollary~\ref{cor:3}, since $\Cl_{1,4}\cong \Cl_{3,2}$.


\subsubsection{Integral structure on admissible $\Cl_{0,5}$-module}


The admissible integral $\Cl_{0,5}$-module was constructed in Section~\ref{sec:0n}.


\subsection{Integral structure on admissible $\Cl_{r,s}$-modules with $r+s=6$}



\subsubsection{Integral structure on admissible $\Cl_{6,0}$-module}


The Clifford algebra $\Cl_{6,0}$ is isomorphic to the space $\mathbb R(8)$.
Let $z_1,\dots,z_6$ be orthonormal generators of $\Cl_{6,0}$, 
and $(V,\langle\cdot\,,\cdot\rangle_{V})$ an admissible $\Cl_{6,0}$-module. Then 
$
J_i^2:=J^2_{z_i}=-\Id_V$, $i=1,\ldots,6$.
We consider three mutually commuting isometric involutions
$$
P_1=J_1J_2J_3J_4,\quad P_2=J_1J_2J_5J_6,\quad P_3=J_1J_4J_5. 
$$
We start from the common eigenvector $w\in V$:
\begin{align}\label{equ60}
P_1w=P_2w=P_3w=w, \quad \langle w,w\rangle_{V}=1.
\end{align}
The existence of such $w$ is guaranteed by the complementary isometries 
$T_1=J_1$, $T_2=J_5$ and $T_3=J_5J_6$ and Lemma~\ref{lem:PT1}, part 1). We sum up the relations between these two families of operators in the following tables
$$
\begin{array}{|l|c|c|c|c|c|c|}\hline
\text{\small{Involutions}}\backslash{\text{\small{Generators}}} &J_1&J_2&J_3&J_4&J_5&J_6\\\hline
P_1=J_1J_2J_3J_4&a&a&a&a&c&c\\\hline
P_2=J_1J_2J_5J_6&a&a&c&c&a&a\\\hline
P_3=J_1J_4J_5     &c&a&a&c&c&a\\\hline
\end{array}
$$
$$
\begin{array}{|l|c|c|c|}\hline
\text{Involutions}
\backslash{\text{Comp. op.}} &J_1(+\to +)&J_5(+\to +)&J_5J_6(+\to +)\\\hline
P_1=J_1J_2J_3J_4&a&c&c\\\hline
P_2=J_1J_2J_5J_6&&a&c\\\hline
P_3=J_1J_4J_5&&&a\\\hline
\end{array}
$$

{\small
\begin{table}[h]
\begin{center}
\caption{Eigenspace decomposition: $\Cl_{6,0}$ case}
\vspace{-3.6mm}
\begin{tabular}{|l|c|c|c|c|c|c|c|c|} \hline
Involutions&\multicolumn{8}{|c|}{Eigenvalues}\\ \hline
$P_1$&\multicolumn{4}{|c|}{$+1$} &\multicolumn{4}{|c|}{$-1$}\\\hline
$P_2$&\multicolumn{2}{|c|}{$+1$}&\multicolumn{2}{c}{$-1$}&\multicolumn{2}{|c|}{$+1$}&\multicolumn{2}{|c|}{$-1$}\\\hline
$P_3$&{$+1$}&{$-1$}&{$+1$}&$-1$&{$+1$}&{$-1$}&$+1$&$-1$\\\hline\hline
{Eigenvector}&$w$ &$J_{1}J_{2}w$&$J_{5}w$&$J_6w$&$J_4w$&$J_{3}w$&$J_{1}w$&$J_2w$\\\hline
\end{tabular}
\end{center}
\end{table}
}
We get 8 orthonormal vectors
$$
\begin{array}{lllll}
&v_1=w,\quad &v_2=J_1w,\quad &v_3=J_2w,\quad &v_4=J_3w,
\\
&v_5=J_4w,\quad &v_6=J_5w,\quad &v_7=J_6w,\quad &v_8=J_1J_2w,
\end{array}
$$
with $\langle v_{\alpha},v_{\alpha}\rangle_V=1$.
Any two of these 8 vectors does not belong 
to a common eigenspace of $P_1$, $P_2$ and $P_3$, which
says that they are orthogonal. Relations~\eqref{equ60} lead to relations~\eqref{eq:Cl06_2}, that still hold up to sign, and additional ones, arising from $P_3w=w$. It shows that each
operator $J_i$ $i=1,\ldots,6$ permutes these basis up to the sign.
The sub-module spanned by theses 8 vectors is admissible integral irreducible sub-module with positive definite inner product, see Remark~\ref{rem:product} and observation in Section~\ref{fremark}.


\subsubsection{Integral structure on admissible $\Cl_{5,1}$-module}


The Clifford algebra $\Cl_{5,1}$ is isomorphic to the space $\mathbb
H(4)$. We can apply Theorem~\ref{prop:r0tor1} or Theorem~\ref{prop:7} and Corollary~\ref{cor:3} to show the existence of an integral structure.


\subsubsection{Integral structure on admissible $\Cl_{4,2}$-module}


The Clifford algebra $\Cl_{4,2}$ is isomorphic to the space $\mathbb H(4)$.
Let $z_1,\dots,z_6$ be generators of $\Cl_{4,2}$, and $(V,\langle\cdot\,,\cdot\rangle_{V})$ an admissible $\Cl_{4,2}$-module. Then 
$
J_i^2:=J^2_{z_i}=-\Id_V$, $i=1,\ldots,4$, $J_j^2:=J^2_{z_j}=\Id_V$ for $j=5,6$.
We fix two mutually commuting isometric involutions $P_1=J_1J_2J_3J_4$ and
$P_2=J_1J_2J_5J_6$. Then we have two complementary isometries
$T_1=J_1$, $T_2=J_2J_3$
which satisfy relations~\eqref{eq:PT50}.
So, the simultaneous eigenspace $E_{1+}\cap E_{2+}$ of $P_1$  and $P_2$ is neutral by Lemma~\ref{lem:PT1}, part 1).
Hence we find a common eigenvector $w\in V$ satisfying $P_1w=P_2w=w$ with $\langle w,w\rangle_V=1$.
It gives a simultaneous eigenspace
decomposition of the representation space $V$, generated by Clifford multiplication of $w$, and which is presented in Table 8.
{\tiny{\small
\begin{table}[h]
\begin{center}
\caption{Eigenspace decomposition: $\Cl_{4,2}$ case}
\begin{tabular}{|l|c|c|c|c|} \hline
Involutions&\multicolumn{4}{|c|}{Eigenvalues}\\ \hline
$P_1$&\multicolumn{2}{|c|}{$+1$} &\multicolumn{2}{|c|}{$-1$}\\\hline
$P_2$&\multicolumn{1}{|c|}{$+1$}&\multicolumn{1}{c}{$-1$}&\multicolumn{1}{|c|}{$+1$}&\multicolumn{1}{|c|}{$-1$}\\\hline
{Eigenvector}
&$w, J_1J_2w$&$J_{5}w, J_{6}w$&$J_{3}w, J_4w$&$J_{1}w, J_{2}w$\\
&$J_1J_3J_5w, J_2J_3J_5w$&$J_{1}J_3w, J_{1}J_4w$&$J_1J_5w, J_{1}J_{6}w$&$J_3J_{5}w, J_4J_{6}w$\\\hline
\end{tabular}
\end{center}
\end{table}
}}

We chose the basis
$$
\begin{array}{lllll}
&v_1=w,\quad &v_2=J_1w,\quad &v_3=J_2w,\quad &v_4=J_3w,
\\
&v_5=J_4w,\quad &v_6=J_1J_2w,\quad &v_7=J_1J_3w,\quad &v_8=J_1J_4w,
\\
&v_9=J_5w,\quad &v_{10}=J_6w,\quad &v_{11}=J_1J_5w,\quad &v_{12}=J_1J_6w,
\\
&v_{13}=J_3J_5w,\quad &v_{14}=J_4J_6w,\quad &v_{15}=J_1J_3J_5w,\quad &v_{16}=J_2J_3J_5w
\end{array}
$$
with $\langle v_{\alpha},v_{\alpha}\rangle_V=1$, $\alpha=1,\ldots,8$ and $\langle v_{\alpha},v_{\alpha}\rangle_V=-1$ for $\alpha=9,\ldots,16$.
The vectors can be made orthogonal if we apply Lemma~\ref{orthogonal} to operators 
$J_1J_3J_5$ and $J_2J_3J_5$. The relations~\eqref{eq:Cl06_2} show that all other relations will be made also orthogonal. They also proves that the operators $J_j$, $j=1,\ldots,6$, permute the basis up to sign.  
Hence we constructed a minimal admissible integral sub-module of $\Cl_{4,2}$-module of the
dimension $16$, which is irreducible.


\subsubsection{Integral structure on admissible $\Cl_{3,3}$-module}\label{sec:Cl33}


The Clifford algebra $\Cl_{3,3}$ is isomorphic to the space $\mathbb R(8)$.
Let $z_1,\dots,z_6$ be generators of $\Cl_{3,3}$, 
and $(V,\langle\cdot\,,\cdot\rangle_{V})$ an admissible $\Cl_{3,3}$-module. Then 
$
J_i^2:=J^2_{z_i}=-\Id_V$, $i=1,2,3$, $J_j^2:=J^2_{z_j}=\Id_V$ for $j=4,5,6$.
We argue similar to the case $\Cl_{6,0}$ but consider different isometric mutually commuting involutions
$$
P_1=J_1J_2J_4J_5,\quad P_2=J_2J_3J_5J_6,\quad P_3=J_1J_2J_3.
$$
Then we have two tables of commutation relations with the generators
$J_i$ and complementary operators.
{\small
$$
\begin{array}{|l|c|c|c|c|c|c|}\hline
\text{\small{Involutions}}\backslash{\text{\small{Generators}}} &J_1&J_2&J_3&J_4&J_5&J_6\\\hline
P_1=J_1J_2J_4J_5&a&a&c&a&a&c\\\hline
P_2=J_2J_3J_5J_6&c&a&a&c&a&a\\\hline
P_3=J_1J_2J_3&c&c&c&a&a&a\\\hline
\end{array}
$$
$$
\begin{array}{|l|c|c|c|}\hline
\text{Involutions}
\backslash{\text{Comp. op.}} &J_1(+\to +)&J_3(+\to +)&J_1J_4(+\to -)\\\hline
P_1=J_1J_2J_4J_5&a&c&c\\\hline
P_2=J_2J_3J_5J_6&&a&c\\\hline
P_3=J_1J_2J_3&&&a\\\hline
\end{array}
$$
}

The tables show that the common eigenspace $E_{1+}\cap E_{2+}$ of operators $P_1$ and
$P_2$ is neutral scalar product space by Lemma~\ref{lem:PT1}, part 1). If the restriction of $\langle\cdot\,,\cdot\rangle_V$ on $E_{1+}\cap E_{2+}\cap E_{3+}$ is negative definite, we apply procedure of Lemma~\ref{lemma:04} and change the operator $P_3=J_1J_2J_3$ to the operator $\hat P_3=J_2J_1J_3$. So we can find a vector $w\in E_{1+}\cap E_{2+}\cap E_{3+}\subset V$ with properties
$
P_1w=P_2w=P_3w=w$ and $\langle w,w\rangle_V=1$. 
Then we obtain Table 9, expressing the simultaneous eigenspace decomposition by the
involutions $P_1, P_2$ and $P_3$.
{\small
\begin{table}[h]
\begin{center}
\caption{Eigenspace decomposition: $\Cl_{3,3}$ case}
\begin{tabular}{|l|c|c|c|c|c|c|c|c|} \hline
Involutions&\multicolumn{8}{|c|}{Eigenvalues}\\ \hline
$P_1$&\multicolumn{4}{|c|}{$+1$} &\multicolumn{4}{|c|}{$-1$}\\\hline
$P_2$&\multicolumn{2}{|c|}{$+1$}&\multicolumn{2}{c}{$-1$}&\multicolumn{2}{|c|}{$+1$}&\multicolumn{2}{|c|}{$-1$}\\\hline
$P_3$&{$+1$}&{$-1$}&{$+1$}&$-1$&{$+1$}&{$-1$}&$+1$&$-1$\\\hline\hline
{Eigenvector}&$w$ &$J_{1}J_{4}w$&$J_3w$&$J_6w$&$J_{1}w$&$J_{4}w$&$J_2w$&$J_5w$\\\hline
\end{tabular}
\end{center}
\end{table}
}
It allows us to choose the orthonormal basis~\eqref{eq:basis33}.
\begin{equation}\label{eq:basis33}
\begin{array}{lllll}
&v_1=w,\quad &v_2=J_1w,\quad &v_3=J_2w,\quad &v_4=J_3w,
\\
&v_5=J_4w,\quad &v_6=J_5w,\quad &v_7=J_6w,\quad &v_8=J_1J_4w,
\end{array}
\end{equation}
with $\langle v_{\alpha},v_{\alpha}\rangle_V=1$, $\alpha=1,\ldots,4$ and $\langle v_{\alpha},v_{\alpha}\rangle_V=-1$ for $\alpha=5,\ldots,8$.
Moreover, as in the previous case of $\Cl_{0,6}$-module, relations $P_1w=P_2w=P_3w=w$ leave only 8 linear independent vectors and shows that operators $J_j$, $j=1,\ldots,6$ permute the basis up to sign. Finally we conclude that the minimal admissible integral sub-module of $\Cl_{3,3}$-module is of dimension
$8$ and it is irreducible.


\subsubsection{Integral structure on admissible $\Cl_{2,4}$-module}


The Clifford algebra $\Cl_{2,4}$ is isomorphic to the space $\mathbb
R(8)$. The integral structure exists according to Theorem~\ref{prop:7} and Corollar~\ref{cor:3}, since $\Cl_{3,3}\cong\Cl_{2,4}$.


\subsubsection{Integral structure on admissible $\Cl_{1,5}$-module}


The Clifford algebra $\Cl_{1,5}$ is isomorphic to the space $\mathbb
H(4)$ and the integral structure exists according to Theorem~\ref{prop:7} and Corollar~\ref{cor:3} by the isomorphism
$\Cl_{4,2}\cong\Cl_{1,5}$.


\subsubsection{Integral structure on admissible $\Cl_{0,6}$-module}


The integral admissible $\Cl_{0,6}$-module was constructed in Section~\ref{sec:0n}.


\subsection{Integral structure on admissible $\Cl_{r,s}$-modules with $r+s=7$}\label{sec:Cl700}



\subsubsection{Integral structure on admissible $\Cl_{7,0}$-module}\label{sec:Cl70}


The Clifford algebra $\Cl_{7,0}$ is isomorphic to the space $\mathbb R(8)\oplus \mathbb R(8)$.
Let $z_1,\dots,z_7$ be a set of orthonormal generators of $\Cl_{7,0}$, 
and $(V,\langle\cdot\,,\cdot\rangle_{V})$ an admissible $\Cl_{7,0}$-module. Then 
$
J_i^2:=J^2_{z_i}=-\Id_V$, $i=1,\ldots,7$.
We consider four isometric involutions commuting with each other:
$$
P_1=J_1J_2J_3J_4,\quad P_2=J_1J_2J_5J_6,\quad P_3=J_2J_3J_6J_7,\quad\text{and}\quad P_4=J_1J_4J_5.
$$
We start from the common eigenvector $w\in V$ of $P_i$, $i=1,2,3$:
$
P_1w=P_2w=P_3w=w$.
The existence of three isometric complementary to $P_1,P_2,P_3$ isometries
\[
T_1=J_1,\quad T_2=J_5, \quad\text{and}\quad T_3=J_7
\]
guarantees that the space $E:=E_{1+}\cap E_{2+}\cap E_{3+}\subset V$ is neutral by Lemma~\ref{lem:PT1}, part 1).
Since the isometric involution $P_4$ commutes with $P_1$, $P_2$ and $P_3$, their common
eigenspace $E$ is $P_4$-invariant. We write $E=E_{4+}\oplus_{\perp}E_{4-}$. If the restriction of $\langle\cdot\,,\cdot\rangle_V$ to $E_{+4}$ is not negative definite we does not need to do anything and we find $w\in E\cap E_{4+}$ such that
\begin{equation}\label{eq:70P4}
P_1w=P_2w=P_3w=P_4w=w\quad\text{and}\quad \langle w, w\rangle_V=1.
\end{equation}
If the restriction of $\langle\cdot\,,\cdot\rangle_V$ to $E_{+4}$ is negative definite, then the restriction of $\langle\cdot\,,\cdot\rangle_V$ to $E_{-4}$ is positive definite. We change the operator $P_4=J_1J_4J_5$ to the operator $\hat P_4=J_4J_1J_5$. Then for $w\in E_{4-}$ we have $\hat P_4w=-P_4w=w$. Thus we can find $w\in E\cap \hat E_{4+}$, where $\hat E_{4+}=E_{4-}$ satisfying~\eqref{eq:70P4}, where the operator $P$ is changed to $\hat P$.
To the linear relations~\eqref{eq:Cl07_1}, that still hold up to sign, we add new ones coming from $P_4w=w$ and choose the orthonormal basis 
\begin{equation}\label{eq:basis07}
\begin{array}{lllll}
&v_1=w,\quad &v_2=J_1w,\quad &v_3=J_2w,\quad &v_4=J_3w,
\\
&v_5=J_4w,\quad &v_6=J_5w,\quad &v_7=J_6w,\quad &v_8=J_7w,
\end{array}
\end{equation}
with $\langle v_{\alpha},v_{\alpha}\rangle_V=1$, $\alpha=1,\ldots,8$.
These relations also shows that operators $J_j$, $j=1,\ldots,7$ permutes this basis up to sign.

\begin{remark}\label{rem:product}
Note that we constructed  
from the given admissible module $V$, where we assumed that the scalar product is neutral
a minimal admissible sub-module of the Clifford algebra $\Cl_{7,0}$
whose dimension is $8$ and the restriction of the given neutral scalar
product to this subspace is {\it positive definite}. 
The sub-module is irreducible. The same constructions were done for the $\Cl_{5,0}$- and $\Cl_{6,0}$-modules.
\end{remark}

%
%

{\small
\begin{table}[h]
\begin{center}
\caption{Eigenspace decomposition: $\Cl_{7,0}$ case}
\begin{tabular}{|l|c|c|c|c|c|c|c|c|} \hline
Involutions&\multicolumn{8}{|c|}{Eigenvalues}\\ \hline
$P_1$&\multicolumn{4}{|c|}{$+1$} &\multicolumn{4}{|c|}{$-1$}\\\hline
$P_2$&\multicolumn{2}{|c|}{$+1$}&\multicolumn{2}{c}{$-1$}&\multicolumn{2}{|c|}{$+1$}&\multicolumn{2}{|c|}{$-1$}\\\hline
$P_3$&{$+1$}&{$-1$}&{$+1$}&$-1$&{$+1$}&{$-1$}&$+1$&$-1$\\\hline\hline
$P_4$&{$+1$}&{$-1$}&{$+1$}&$-1$&{$+1$}&{$-1$}&$+1$&$-1$\\\hline\hline
{Eigenvectors}&$w$ &$J_7w$&$J_3w$&$J_6w$&$J_1w$&$J_4w$&$J_2w$&$J_5w$\\\hline
\end{tabular}
\end{center}
\end{table}
}


\subsubsection{Integral structure on admissible $\Cl_{6,1}$-module}


The Clifford algebra $\Cl_{6,1}$ is isomorphic to the space $\mathbb
C(8)$.  We can apply Theorem~\ref{prop:r0tor1} or Theorem~\ref{prop:7} and Corollary~\ref{cor:3} to show the existence of an integral structure.


\subsubsection{Integral structure on admissible $\Cl_{5,2}$-module}


The Clifford algebra $\Cl_{5,2}$ is isomorphic to the space $\mathbb H(4)\oplus\mathbb H(4)$.
Let $z_1,\dots,z_7$ be orthonormal generators of $\Cl_{5,2}$, and
$(V,\langle\cdot\,,\cdot\rangle_{V})$ an admissible
$\Cl_{5,2}$-module. 
Then 
$
J_i^2:=J^2_{z_i}=-\Id_V$, $i=1,\ldots,5$, $J_j^2:=J^2_{z_j}=\Id_V$ for $j=6,7$.
The isometric involutions
$$
P_1=J_1J_2J_3J_4,\quad P_2=J_1J_2J_6J_7,\quad P_3=J_5J_6J_7.
$$
mutually commute. 
In this case we have two complementary isometric operators 
$T_1=J_1$, $T_2=J_2J_3$. 
Here we present the tables of their commutation relations
{\small
$$
\begin{array}{|l|c|c|c|c|c|c|c|}\hline
\text{\small{Involutions}}\backslash{\text{\small{Generators}}} &J_1&J_2&J_3&J_4&J_5&J_6&J_7\\\hline
P_1=J_1J_2J_3J_4&a&a&a&a&c&c&c\\\hline
P_2=J_1J_2J_6J_7&a&a&c&c&c&a&a\\\hline
P_3=J_5J_6J_7     &a&a&a&a&c&c&c\\\hline
\end{array}
$$
$$
\begin{array}{|l|c|c|}\hline
\text{Involutions}
\backslash{\text{Comp. op.}} &J_1(+\to +)&J_2J_3(+\to +)\\\hline
P_1=J_1J_2J_3J_4&a&c\\\hline
P_2=J_1J_2J_6J_7&&a\\\hline
P_3=J_5J_6J_7&&\\\hline
\end{array}
$$
}
Since there is no a complementary isometric operator with the property that
it commutes with $P_1$ and $P_2$ and anti-commutes with $P_3$ we argue as in the case of $\Cl_{7,0}$ and use the property that isometry $P_3$ commutes with $P_1$ and $P_2$ and if it is necessary we change $P_3$ to $\hat P_3$. We find a vector $w\in V$ such that $P_1w=P_2w=P_3w=w$ and $\langle w,w\rangle_V=1$. This relations will give the orthonormal basis and show that $J_j$ acts by permutation on this basis.
The eigenspace decomposition is given by Table 11.
{\small
\begin{table}[h]
\begin{center}
\caption{Eigenspace decomposition: $\Cl_{5,2}$ case}
\begin{tabular}{|l|c|c|c|c|c|c|c|c|} \hline
Involutions&\multicolumn{8}{|c|}{Eigenvalues}\\ \hline
$P_1$&\multicolumn{4}{|c|}{$+1$} &\multicolumn{4}{|c|}{$-1$}\\\hline
$P_2$&\multicolumn{2}{|c|}{$+1$}&\multicolumn{2}{c}{$-1$}&\multicolumn{2}{|c|}{$+1$}&\multicolumn{2}{|c|}{$-1$}\\\hline
$P_3$&{$+1$}&{$-1$}&{$+1$}&$-1$&{$+1$}&{$-1$}&$+1$&$-1$\\\hline\hline
{Eigenvector}&$w, J_5w$&&$J_6w, J_7w$&&&$J_{3}w, J_4w$&&$J_1w, J_{2}w$
\\
&$J_1J_3J_6w, J_1J_{3}J_7w$&&$J_1J_3w, J_1J_4w$&&&$J_{1}J_6w, J_1J_7w$&&$J_3J_6w, J_3J_7w$
\\
\hline
\end{tabular}
\end{center}
\end{table}
}
The basis is given in~\eqref{eq:basis52}.
\begin{equation}\label{eq:basis52}
\begin{array}{lllll}
&v_1=w,\quad &v_2=J_1w,\quad &v_3=J_2w,\quad &v_4=J_3w,
\\
&v_5=J_4w,\quad &v_6=J_5w,\quad &v_7=J_1J_3w,\quad &v_8=J_1J_4w,
\\
&v_9=J_5w,\quad &v_{10}=J_6w,\quad &v_{11}=J_1J_6w,\quad &v_{12}=J_1J_7w,
\\
&v_{13}=J_3J_6w,\quad &v_{14}=J_3J_7w,\quad &v_{15}=J_1J_3J_6w,\quad &v_{16}=J_1J_3J_7w
\end{array}
\end{equation}
with $\langle v_{\alpha},v_{\alpha}\rangle_V=1$, $\alpha=1,\ldots,8$ and $\langle v_{\alpha},v_{\alpha}\rangle_V=-1$ for $\alpha=9,\ldots,16$. Finally, since the vectors $w$ and $J_1J_3J_6w$  and $J_1J_3J_7w$
need not be orthogonal, we apply Lemma~\ref{orthogonal}
to change the vector $w$ to $\tilde{w}$. 
We constructed an integral structure in
the admissible sub-module, which is irreducible.


\subsubsection{Integral structure on admissible $\Cl_{4,3}$-module}


The Clifford algebra $\Cl_{4,3}$ is isomorphic to the space $\mathbb C(8)$.
Let $z_1,\dots,z_7$ be generators of $\Cl_{4,3}$, and
$(V,\langle\cdot\,,\cdot\rangle_{V})$ 
an admissible $\Cl_{4,3}$-module Then 
$
J_i^2:=J^2_{z_i}=-\Id_V$, $i=1,\ldots,4$, $J_j^2:=J^2_{z_j}=\Id_V$, for $j=5,6,7$.
We use the same mutually commuting isometries as in the case of $\Cl_{0,7}$-module
$$
P_1=J_1J_2J_3J_4,\quad P_2=J_1J_2J_5J_6,\quad P_3=J_2J_3J_6J_7.
$$
We start from the common eigenvector $w\in V$: $
P_1w=P_2w=P_3w=w$, $\langle w,w\rangle_{V}=1$.
The existence of such a common eigenvector $w$ is guaranteed by Lemma~\ref{lem:PT1}, part 1) and presence of three
complementary isometries 
$T_1=J_1$, $T_2=J_2J_3$ and $T_3=J_1J_2$ with the
commutation
relations
$$
\begin{array}{|l|c|c|c|}\hline
\text{Involutions}
\backslash{\text{Comp. op.}} &J_1(+\to +)&J_2J_3(+\to +)&J_1J_2(+\to +)\\\hline
P_1=J_1J_2J_3J_4&a&c&c\\\hline
P_2=J_1J_2J_5J_6&&a&c\\\hline
P_3=J_2J_3J_6J_7&&&a\\\hline
\end{array}
$$
Then we have simultaneous eigenspace decomposition of $P_i$ showed in Table 12.
{\small
\begin{table}[h]
\begin{center}
\caption{Eigenspace decomposition: $\Cl_{4,3}$ case}
\begin{tabular}{|l|c|c|c|c|c|c|c|c|} \hline
Involutions&\multicolumn{8}{|c|}{Eigenvalues}\\ \hline
$P_1$&\multicolumn{4}{|c|}{$+1$} &\multicolumn{4}{|c|}{$-1$}\\\hline
$P_2$&\multicolumn{2}{|c|}{$+1$}&\multicolumn{2}{c}{$-1$}&\multicolumn{2}{|c|}{$+1$}&\multicolumn{2}{|c|}{$-1$}\\\hline
$P_3$&\multicolumn{1}{|c|}{$+1$}&\multicolumn{1}{c}{$-1$}&\multicolumn{1}{|c|}{$+1$}&\multicolumn{1}{|c|}{$-1$}
&\multicolumn{1}{|c|}{$+1$}&\multicolumn{1}{c}{$-1$}&\multicolumn{1}{|c|}{$+1$}&\multicolumn{1}{|c|}{$-1$}\\\hline
{Eigenvector}
&$w$&$J_7w$&$J_5w$&$J_6w$&$J_4w$&$J_3w$&$J_1w$&$J_2w$
\\
&$J_1J_3J_6w$&$J_1J_2w$&$J_1J_4w$&$J_1J_3w$&$J_1J_5w$&$J_1J_6w$&$J_3J_6w$&$J_1J_7w$\\\hline
\end{tabular}
\end{center}
\end{table}
}
The eigenvectors listed in Table 12 form a basis. Since the elements $w$ and $J_1J_3J_6w$ are not necessarily
orthogonal, we apply Lemma~\ref{orthogonalization} 
to make them
orthogonal. 
The relations~\eqref{eq:Cl07_1} show 
that $J_i$, $i=1,\ldots,7$, permute basis vectors up to sign. 
The constructed sub-module of admissible $\Cl_{4,3}(w)$-module
 is admissible, integral, and irreducible.


\subsubsection{Integral structure on admissible $\Cl_{3,4}$-module}


The Clifford algebra $\Cl_{3,4}$ is isomorphic to the space $\mathbb R(8)\oplus \mathbb R(8)$.
Let $z_1,\dots,z_7$ be orthonormal generators of $\Cl_{3,4}$, 
and $(V,\langle\cdot\,,\cdot\rangle_{V})$ an admissible $\Cl_{3,4}$-module. Then 
$
J_i^2:=J^2_{z_i}=-\Id_V$, $i=1,\ldots,3$, $J_j^2:=J^2_{z_j}=\Id_V$ for $j=4,\ldots,7$.
We fix mutually commuting isometric involutions
$$
P_1=J_1J_2J_4J_5,\quad P_2=J_2J_3J_5J_6,\quad P_3=J_1J_2J_6J_7,\quad\text{and}\quad P_4=J_3J_4J_5.
$$
They have the following commutation relations with representations of generators.
{\small
$$
\begin{array}{|l|c|c|c|c|c|c|c|}\hline
\text{\small{Involutions}}\backslash{\text{\small{Generators}}} &J_1&J_2&J_3&J_4&J_5&J_6&J_7\\\hline
P_1=J_1J_2J_4J_5&a&a&c&a&a&c&c\\\hline
P_2=J_2J_3J_5J_6&c&a&a&c&a&a&c\\\hline
P_3=J_1J_2J_6J_7&a&a&c&c&c&a&a\\\hline
P_4=J_3J_4J_5     &a&a&c&c&c&a&a\\\hline
\end{array}
$$
}
Define the complementary operators $T_1=J_1$, $T_2=J_3$ and $T_3=J_7$.
{\small
$$
\begin{array}{|l|c|c|c|}\hline
\text{Involutions}
\backslash{\text{Comp. op.}} &J_1(+\to +)&J_3(+\to +)&J_7(+\to -)\\\hline
P_1=J_1J_2J_4J_5&a&c&c\\\hline
P_2=J_2J_3J_5J_6&&a&c\\\hline
P_3=J_1J_2J_6J_7&&&a\\\hline
P_4=J_{3}J_{4}J_{5}&&&\\\hline
\end{array}
$$
}
Since there are no complementary isometric operator with the property that
it commutes with $P_1$ and $P_2$ and anti-commutes with $P_3$,
we only know that the
simultaneous eigenspaces of $P_1$ and $P_2$ are neutral spaces by Lemma~\ref{lem:PT1}, part 1). The complementary anti-isometry $T_3=J_7$ and Lemma~\ref{lem:PT1}, part 2) guaranties that the space $E=E_{1+}\cap E_{2+}\cap E_{3+}$ either neutral or sign definite. We consider both possibilities.

Let the restriction of $\langle\cdot\,,\cdot\rangle_V$ on $E$ be neutral, then we argue as in the case of $\Cl_{0,7}$ algebra and find 
\begin{equation}\label{eq:70P41}
P_1w=P_2w=P_3w=P_4w=w\quad\text{and}\quad \langle w, w\rangle_V=1.
\end{equation}
In this case we can directly proceed further and construct an othonormal basis. 

Let the restriction of $\langle\cdot\,,\cdot\rangle_V$ on $E$ be negative definite, then we change $P_3$ to $\hat P_3=J_2J_1J_6J_7$, that will make the space $\hat E=E_{1+}\cap E_{2+}\cap \hat E_{3+}$ positive definite space.

If the restriction of $\langle\cdot\,,\cdot\rangle_V$ on $E$ is positive definite, we do nothing.

Thus, from now on we can assume that the restriction of $\langle\cdot\,,\cdot\rangle_V$ on $E=E_{1+}\cap E_{2+}\cap E_{3+}$ is positive definite.
Hence it can happen only two cases for $w\in E$:
\begin{itemize}
\item[(1)] $\langle w+P_4w,w+P_4w\rangle_V=0$, that is
$P_4w=-w$. We change the operator $P_4=J_3J_4J_5$ to the operator $\hat P_4=J_4J_3J_5$, then 
the vector $w$ is a common eigenvector of all four involutions $P_i $ with the eigenvalue $1$, i.~e. $w$ satisfies~\eqref{eq:70P41}. 
\item[(2)] $\langle w+P_4w,w+P_4w\rangle_V>0$. In this case we get the eigenvector $\hat w=w+P_4w$ of $P_4$ with the eigenvalue
$1$. Normalising the vector $\hat w$ 
we obtain that $\hat w$ satisfies~\eqref{eq:70P41}.
\end{itemize}

So in the both cases the admissible sub-module of the Clifford algebra
$\Cl_{3,4}$ is 8 dimensional and
decomposed into $8$ common eigenspaces as shown in Table 13. The orthogonal basis is given in~\eqref{eq:basis07} with $\langle v_{\alpha},v_{\alpha}\rangle_V=1$, $\alpha=1,\ldots,4$ and $\langle v_{\alpha},v_{\alpha}\rangle_V=-1$ for $\alpha=5,\ldots,8$.
The relations~\eqref{eq:70P41} show also that operators $J_j$, $j=1,\ldots,7$, permute the basis up to sign. 
{\small
\begin{table}[h]
\begin{center}
\caption{Eigenspace decomposition: $\Cl_{3,4}$ case}
\vspace{-3.5mm}
\begin{tabular}{|l|c|c|c|c|c|c|c|c|} \hline
Involutions&\multicolumn{8}{|c|}{Eigenvalues}\\ \hline
$P_1$&\multicolumn{4}{|c|}{$+1$} &\multicolumn{4}{|c|}{$-1$}\\\hline
$P_2$&\multicolumn{2}{|c|}{$+1$}&\multicolumn{2}{c}{$-1$}&\multicolumn{2}{|c|}{$+1$}&\multicolumn{2}{|c|}{$-1$}\\\hline
$P_3$&\multicolumn{1}{|c|}{$+1$}&\multicolumn{1}{c}{$-1$}&\multicolumn{1}{|c|}{$+1$}&\multicolumn{1}{|c|}{$-1$}
&\multicolumn{1}{|c|}{$+1$}&\multicolumn{1}{c}{$-1$}&\multicolumn{1}{|c|}{$+1$}&\multicolumn{1}{|c|}{$-1$}
\\\hline
$P_4$&\multicolumn{1}{|c|}{$+1$}&\multicolumn{1}{c}{$-1$}&\multicolumn{1}{|c|}{$+1$}&\multicolumn{1}{|c|}{$-1$}
&\multicolumn{1}{|c|}{$+1$}&\multicolumn{1}{c}{$-1$}&\multicolumn{1}{|c|}{$+1$}&\multicolumn{1}{|c|}{$-1$}\\\hline\hline
{Eigenvectors}&$w$&$J_7w$ &$J_3w$ &$J_6w$ &$J_4w$&$J_1w$&$J_5w$&$J_2w$\\\hline
\end{tabular}
\end{center}
\end{table}
}
The constructed module is admissible integral, and irreducible.


\subsubsection{Integral structure on admissible $\Cl_{2,5}$-module}


The integral structure on the admissible $\Cl_{2,5}$-module exists according to Theorem~\ref{prop:7} and Corollar~\ref{cor:3}, since $\Cl_{2,5}\cong\Cl_{4,3}$.


\subsubsection{Integral structure on admissible $\Cl_{1,6}$-module}


The integral structure on the admissible $\Cl_{1,6}$-module exists according to Theorem~\ref{prop:7} and Corollar~\ref{cor:3}, since $\Cl_{1,6}\cong\Cl_{5,2}$.


\subsubsection{Integral structure on admissible $\Cl_{0,7}$-module}


The admissible integral $\Cl_{0,7}$-module was constructed in Section~\eqref{sec:0n}.



\subsection{Integral structure on admissible $\Cl_{r,s}$-modules with $r+s=8$}


\subsubsection{Integral structure on admissible $\Cl_{8,0}$-module}


The Clifford algebra $\Cl_{8,0}$ is isomorphic to the space $\mathbb R(16)$.
Let $z_1,\dots,z_8$ be generators of $\Cl_{8,0}$, and $(V,(\cdot\,,\cdot)_{V})$ an admissible $\Cl_{8,0}$-module with the positive definite product $(\cdot\,,\cdot)_{V}$. Note that here we use the known fact that $\Cl_{r,0}$-modules are admissible with a positive definite product. Then 
$
J_i^2:=J^2_{z_i}=-\Id_V$, $i=1,\ldots,8$.
In this case we argue as in the case of $\Cl_{0,8}$-modules and consider the mutually commuting isometric involution
$$
P_1=J_1J_2J_3J_4,\quad P_2=J_1J_2J_5J_6,\quad P_3=J_2J_3J_5J_7,\quad P_4=J_1J_2J_7J_8.
$$
We start from the common eigenvector $w\in V$: $
P_1w=P_2w=P_3w=P_4w=w$ with $\langle w,w\rangle_{V}=1$.
The relations~\eqref{eq:Cl08_all}
are still true and they left
16 elements of orthonormal basis:
$$
\begin{array}{llllll}
&v_1=w,\quad &v_2=J_1J_2w,\quad &v_3=J_1J_3w,\quad &v_4=J_1J_4w,
\\
&v_5=J_1J_5w,\quad &v_6=J_1J_6w,\quad &v_7=J_1J_7w,\quad &v_8=J_1J_8w,
\\
&v_9=J_1w,\quad &v_{10}=J_2w,\quad &v_{11}=J_3w,\quad &v_{12}=J_4w,
\\
&v_{13}=J_5w,\quad &v_{14}=J_6w,\quad &v_{15}=J_7w,\quad &v_{16}=J_8w,
\end{array}
$$
with $(v_{\alpha},v_{\alpha})_{V}=1$, $\alpha=1,\ldots,16$. We see from~\eqref{eq:Cl08_all} that $J_i$, $i=1,\ldots,8$, permute basis vectors up to sign. The constructed 16 dimensional sub-module is admissible, integral, and irreducible.


\subsubsection{Integral structure on admissible $\Cl_{7,1}$-module}


The Clifford algebra $\Cl_{7,1}$ is isomorphic to the space $\mathbb
R(16)$.
We can apply Theorem~\ref{prop:r0tor1} or Theorem~\ref{prop:7} and Corollary~\ref{cor:3} to show the existence of an integral structure. The admissible module is irreducible


\subsubsection{Integral structure on $\Cl_{6,2}$-admissible module}


The Clifford algebra $\Cl_{6,2}$ is isomorphic to the space $\mathbb H(8)$.
Let $z_1,\dots,z_8$ be orthonormal generators of $\Cl_{6,2}$ and 
$(V,\langle\cdot\,,\cdot\rangle_{V})$ be an admissible $\Cl_{6,2}$-module. Then 
$
J_i^2:=J^2_{z_i}=-\Id_V$, $i=1,\ldots,6$, $J_j^2:=J_{z_j}^2=\Id_V$ for $j=7,8$.
We consider the three mutually commuting isometric
involutions
\[
P_1=J_1J_2J_3J_4,\quad P_2=J_1J_2J_5J_6, \quad\text{and}\quad P_3=J_1J_2J_7J_8,
\]
and four complementary operators: $
T_1=J_1,\quad T_2= J_5,\quad T_3=J_7,\quad\text{and}\quad T_4=J_1J_3J_5J_7$.
Then we have commutation relations between
involutions $P_i$ and complementary operators $T_j$.
{\small
$$
\begin{array}{|l|c|c|c|c|}\hline
\text{Involutions}
\backslash{\text{Comp. op.}} &J_1(+\to +)&J_5(+\to +)&J_7(+\to -)&J_1J_3J_5J_7(+\to -)\\\hline
P_1=J_1J_2J_3J_4&a&c&c&c\\\hline
P_2=J_1J_2J_5J_6&&a&c&c\\\hline
P_3=J_1J_2J_7J_8&&&a&c\\\hline
\end{array}
$$}
From these relations
the common eigenspace $E_{1+}\cap E_{2+}$ of the first two involutions
is neutral space by Lemma~\ref{lem:PT1}, part 1). Then we use Lemma~\ref{lem:PT2} and conclude that 
the common eigenspace $E_{1+}\cap E_{2+}\cap E_{3+}$ of all three involutions $P_i$
is a neutral space. So,
we may find an element $w$ such that
$
P_1w=P_2w=P_3w=w$, and $\langle w,w\rangle_V=1$.
The eigenspace decomposition presented in Table 14.
By a direct calculations, we know that
each $4$ vectors belong to a common eigenspace, especially
the vectors
$w$, $J_{1}J_{2}w$, $J_{1}J_{3}J_{5}J_{7}w$, and $J_{1}J_{3}J_{5}J_{8}w$
are in the comon eigenspace with eigenvalue of $1$.
First two are spacelike and orthonormal
and also last two are timelike and orthonormal. 
Unfortunately, first two and last two need not be orthogonal, so we
need to apply Lemma~\ref{orthogonal}.
Since the operators $J_{1}J_{3}J_{5}J_{7}w$ and $J_{1}J_{3}J_{5}J_8w$
are anti-involutions and anti-commute each other, they satisfy
the conditions of Lemma~\ref{orthogonal}. Hence we can
obtain a new
common eigenvector $\tilde{w}$:
$P_1\tilde{w}=P_2\tilde{w}=P_3\tilde{w}=\tilde{w}$, such that
$\langle \tilde{w},J_{1}J_{3}J_{5}J_{7}\tilde{w}\rangle_V=\langle\tilde{w},J_{1}J_{3}J_{5}J_8\tilde{w}\rangle_V=0$. 
The linear dependence relations arising from $P_1\tilde w=P_2 \tilde w=P_3\tilde w=\tilde w$ shows that all other vectors presented in Table 14 become orthogonal and moreover $J_j$, $j=1,\ldots,8$, permute them up to sign. The vectors listed in Table 14 form an integral basis of 32 dimensional admissible sub-module of $\Cl_{6,2}$-module.
Moreover this sub-module is irreducible. 
{\small
\begin{table}[h]
\begin{center}
\caption{Eigenspace decomposition: $\Cl_{6,2}$ case}
\begin{tabular}{|l|c|c|c|c|c|c|c|c|} \hline
Involutions&\multicolumn{8}{|c|}{Eigenvalues}\\ \hline
$P_1$&\multicolumn{4}{|c|}{$+1$} &\multicolumn{4}{|c|}{$-1$}\\\hline
$P_2$&\multicolumn{2}{|c|}{$+1$}&\multicolumn{2}{c}{$-1$}&\multicolumn{2}{|c|}{$+1$}&\multicolumn{2}{|c|}{$-1$}\\\hline
$P_3$&{$+1$}&{$-1$}&{$+1$}&$-1$&{$+1$}&{$-1$}&$+1$&$-1$\\\hline\hline
{Eigenvectors}
&$w$&$J_{7}w$&$J_5w$&$J_{1}J_{3}w$&$J_{3}w$&$J_1J_{5}w$&$J_{1}J_7w$&$J_{1}w$
\\
&$J_1J_2w$ &$J_8w$&$J_6w$&$J_1J_4w$&$J_4w$&$J_1J_6w$&$J_1J_8w$&$J_2w$
\\
&$J_1J_3J_5J_7w$&$J_2J_{3}J_5w$&$J_2J_{3}J_7w$&$J_{5}J_{7}w$&$J_{1}J_{5}J_7w$&$J_{3}J_7w$&$J_{3}J_5w$&$J_{3}J_5J_7w$
\\
&$J_1J_3J_5J_8w$&$J_{2}J_{3}J_6w$&$J_{2}J_{3}J_8w$&$J_5J_8w$&$J_1J_{5}J_8w$&$J_3J_8w$&$J_3J_6w$&$J_3J_5J_8w$\\\hline
\end{tabular}
\end{center}
\end{table}
}


\subsubsection{Integral structure on admissible $\Cl_{5,3}$-module}


The Clifford algebra $\Cl_{5,3}$ is isomorphic to the space 
$\mathbb{H}(8)$. 
Let $z_1,\dots,z_8$ be orthonormal generators of $\Cl_{5,3}$ and 
$(V,\langle\cdot\,,\cdot\rangle_{V})$ be an admissible $\Cl_{5,3}$-module. Then 
$
J_i^2:=J^2_{z_i}=-\Id_V$, $i=1,\ldots,5$, $J_j^2:=J_{z_j}^2=\Id_V$ for $j=6,7,8$.
We consider the three mutually commuting isometric
involutions
\[
P_1=J_1J_2J_3J_4,\quad P_2=J_1J_2J_6J_7, \quad\text{and}\quad P_3=J_2J_3J_7J_8.
\]
In this case, we choose four complementary operators:
$$
T_1=J_1,\quad T_2= J_1J_3,\quad T_3=J_8,\quad\text{and}\quad T_4=J_1J_3J_5J_7.
$$
Then we have commutation relations between
involutions $P_i$ and complementary operators $T_j$.
{\small
$$
\begin{array}{|l|c|c|c|c|}\hline
\text{Involutions}
\backslash{\text{Comp. op.}} &J_1(+\to +)&J_1J_3(+\to +)&J_8(+\to -)&J_1J_3J_5J_7(+\to -)\\\hline
P_1=J_1J_2J_3J_4&a&c&c&c\\\hline
P_2=J_1J_2J_6J_7&&a&c&c\\\hline
P_3=J_2J_3J_7J_8&&&a&c\\\hline
\end{array}
$$}
From these relations
the common eigenspace $E_{1+}\cap E_{2+}$ of the first two involutions
is neutral space by Lemma~\ref{lem:PT1}, part 1). Then we use Lemma~\ref{lem:PT2} and conclude that 
the common eigenspace $E_{1+}\cap E_{2+}\cap E_{3+}$ of all three involutions $P_i$
is a neutral space. So,
we may find an element $w$ such that $
P_1w=P_2w=P_3w=w$ and $\langle w,w\rangle_V=1$.
The eigenspace decomposition presented in Table 15.
{\tiny
\begin{table}[h]
\begin{center}
\caption{Eigenspace decomposition: $\Cl_{5,3}$ case}
\begin{tabular}{|l|c|c|c|c|c|c|c|c|} \hline
Involutions&\multicolumn{8}{|c|}{Eigenvalues}\\ \hline
$P_1$&\multicolumn{4}{|c|}{$+1$} &\multicolumn{4}{|c|}{$-1$}\\\hline
$P_2$&\multicolumn{2}{|c|}{$+1$}&\multicolumn{2}{c}{$-1$}&\multicolumn{2}{|c|}{$+1$}&\multicolumn{2}{|c|}{$-1$}\\\hline
$P_3$&{$+1$}&{$-1$}&{$+1$}&$-1$&{$+1$}&{$-1$}&$+1$&$-1$\\\hline\hline
{Eigenvectors}
&$w$&$J_8w$&$J_6w$&$J_7w$&$J_4w$&$J_3w$&$J_1w$&$J_2w$
\\
&$J_5w$ &$J_1J_2w$&$J_1J_4w$&$J_1J_3w$&$J_1J_6w$&$J_1J_7w$&$J_1J_5w$&$J_1J_8w$
\\
&$J_1J_2J_8w$&$J_5J_8w$&$J_5J_6w$&$J_{5}J_{7}w$&$J_4J_{5}w$&$J_{3}J_5w$&$J_{3}J_7w$&$J_2J_5w$
\\
&$J_1J_2J_5J_8w$&$J_1J_2J_5w$&$J_1J_4J_5w$&$J_1J_3J_5w$&$J_1J_5J_6w$&$J_1J_5J_7w$&$J_3J_5J_7w$&$J_1J_5J_8w$
\\\hline
\end{tabular}
\end{center}
\end{table}
}
We need to apply Lemma~\ref{orthogonal} to operators $\Omega_1=J_1J_2J_8$ and $\Omega_2=J_1J_2J_5J_8$ to make the vectors in $E_{1+}\cap E_{2+}\cap E_{3+}$ orthogonal. It also makes all other vectors orthogonal by relations
$P_1w=P_2w=P_3w=w$. The same relations show that $J_j$, $j=1,\ldots,8$ permute the basis up to sign. It proves that basis listed in Table 15 is integral.
Finally we notice that the constructed sub-module is irreducible,
  since its dimension is $32$.


\subsubsection{Integral structure on admissible $\Cl_{4,4}$-module}

The Clifford algebra $\Cl_{4,4}$ is isomorphic to the space $\mathbb R(16)$.
Let $z_1,\dots,z_8$ be generators of $\Cl_{4,4}$, 
and $(V,\langle\cdot\,,\cdot\rangle_{V})$ an admissible $\Cl_{4,4}$-module. Then 
$
J_i^2:=J^2_{z_i}=-\Id_V$, $i=1,\ldots,4$, $J_j^2:=J^2_{z_j}=\Id_V$ for $j=5,\ldots,8$.
Choose mutually commuting isometric involutions
$$
P_1=J_1J_2J_3J_4,\quad P_2=J_1J_2J_5J_6,\quad P_3=J_2J_3J_5J_7,\quad\text{and}\quad P_4=J_1J_2J_7J_8,
$$
and four complementary operators $T_1=J_1$, $T_2=J_1J_3$, $T_3=J_1J_2$, and $T_4=J_8$.
Here are the tables of the commutation relations with generators and 
complementary operators:
{\small
\begin{center}
{Commutation relations: $\Cl_{4,4}$ case}
\end{center}
\begin{equation*}\label{eq:Cl44Gene}
\begin{array}{|l|c|c|c|c|c|c|c|c|}\hline
\text{Involution}\backslash \text{Generator}&J_{1}&J_{2}&J_{3}&J_{4}&J_{5}&J_{6}&J_{7}&J_8,\\\hline
P_1=J_{1}J_{2}J_{3}J_{4}&a&a&a&a&c&c&c&c\\\hline
P_2=J_{1}J_{2}J_{5}J_{6}&a&a&c&c&a&a&c&c\\\hline
P_3=J_{2}J_{3}J_{5}J_{7}&c&a&a&c&a&c&a&c\\\hline
P_4=J_{1}J_{2}J_{7}J_{8}&a&a&c&c&c&c&a&a\\\hline
\end{array}
\end{equation*}
}
\begin{center}
{Complementary operator}: $\Cl_{4,4}$ case
\end{center}
{\small
\begin{equation*}\label{eq:Cl44comp}
\begin{array}{|l|c|c|c|c|}\hline
\text{Involution}\backslash \text{Comp. op.}&J_{1}(+\to +)&J_{1}J_{3}(+\to +)&J_{1}J_{2}(+\to
+)&J_8(+\to -)\\\hline
P_1=J_{1}J_{2}J_{3}J_{4}&a&c&c&c\\\hline
P_2=J_{1}J_{2}J_{5}J_{6}&&a&c&c\\\hline
P_3=J_2J_{3}J_5J_7       &&&a&c\\\hline
P_4=J_1J_2J_7J_8         &&&&a\\\hline
\end{array}
\end{equation*}
}
From these relations we can choose a vector $w\in V$ such that
$$
P_1w=P_2w=P_3w=P_4w=w,\quad\text{and}\quad\langle w ,w\rangle_{V}=1,
$$
by Lemma~\ref{lem:PT1}, part 1) and Lemma~\ref{lemma:04}.
Hence by the relations~\eqref{eq:Cl08_all}
we have a simultaneous eigenspace decomposition of a subspace in $V$
spanned by the $16$ common eigenvectors that form the basis 
$$
\begin{array}{llllll}
&v_1=w,\quad &v_2=J_1w,\quad &v_3=J_2w,\quad &v_4=J_3w,
\\
&v_5=J_4w,\quad &v_6=J_1J_2w,\quad &v_7=J_1J_3w,\quad &v_8=J_1J_4w,
\\
&v_9=J_5w,\quad &v_{10}=J_6w,\quad &v_{11}=J_7w,\quad &v_{12}=J_8w,
\\
&v_{13}=J_1J_5w,\quad &v_{14}=J_1J_6w,\quad &v_{15}=J_1J_7w,\quad &v_{16}=J_1J_8w,
\end{array}
$$
with $\langle v_{\alpha},v_{\alpha}\rangle_{V}=1$, $\alpha=1,\ldots,8$ and $\langle v_{\alpha},v_{\alpha}\rangle_{V}=-1$ for $\alpha=9,\ldots,16$.
{\small
\begin{table}[h]
\begin{center}
\caption{Eigenspace decomposition: $\Cl_{4,4}$ case}
{\small
\begin{tabular}{|l|c|c|c|c|c|c|c|c|} \hline
Involutions&\multicolumn{8}{|c|}{Eigenvalues}\\ \hline
$P_1$&\multicolumn{8}{|c|}{$+1$}\\\hline
$P_2$&\multicolumn{4}{|c|}{$+1$} &\multicolumn{4}{|c|}{$-1$}\\\hline
$P_3$&\multicolumn{2}{|c|}{$+1$}
&\multicolumn{2}{|c|}{$-1$}&\multicolumn{2}{|c|}{$+1$}
&\multicolumn{2}{|c|}{$-1$}\\
\hline
$P_4$&\multicolumn{1}{|c|}{$+1$}&\multicolumn{1}{|c|}{$-1$}&\multicolumn{1}{|c|}{$+1$}&\multicolumn{1}{|c|}{$-1$}
&\multicolumn{1}{|c|}{$+1$}&\multicolumn{1}{|c|}{$-1$}&\multicolumn{1}{|c|}{$+1$}&\multicolumn{1}{|c|}{$-1$}
\\\hline\hline
{Eigenvectors}
&$w$&$J_8w$&$J_1J_2w$& $J_7w$&$J_6w$&$J_1J_4w$&$J_5w$&$J_1J_3w$\\\hline
\end{tabular}}
\end{center}
\end{table}
}
{\small
\begin{table}[h]
\begin{center}
\begin{tabular}{|l|c|c|c|c|c|c|c|c|} \hline
Involutions&\multicolumn{8}{|c|}{Eigenvalues}\\ \hline
$P_1$&\multicolumn{8}{|c|}{$-1$}\\\hline
$P_2$&\multicolumn{4}{|c|}{$+1$} &\multicolumn{4}{|c|}{$-1$}\\\hline
$P_3$&\multicolumn{2}{|c|}{$+1$}
&\multicolumn{2}{|c|}{$-1$}&\multicolumn{2}{|c|}{$+1$}
&\multicolumn{2}{|c|}{$-1$}\\
\hline
$P_4$&\multicolumn{1}{|c|}{$+1$}&\multicolumn{1}{|c|}{$-1$}&\multicolumn{1}{|c|}{$+1$}&\multicolumn{1}{|c|}{$-1$}
&\multicolumn{1}{|c|}{$+1$}&\multicolumn{1}{|c|}{$-1$}&\multicolumn{1}{|c|}{$+1$}&\multicolumn{1}{|c|}{$-1$}
\\\hline\hline
{Eigenvectors}
&$J_4w$&$J_1J_6w$&$J_3w$&$J_1J_5w$&$J_1J_8w$&$J_1w$&$J_1J_7w$&$J_2w$\\\hline
\end{tabular}
\end{center}
\end{table}
}
The relations~\eqref{eq:Cl08_all} also shows that $J_j$, $j=1,\ldots,8$ permutes the basis up to sign. So the minimal admissible module of $\Cl_{4,4}$ is $16$ dimensional,
irreducible and integral.


\subsubsection{Integral structure on admissible $\Cl_{3,5}$-module}


The integral structure on the admissible $\Cl_{3,5}$-module exists according to Theorem~\ref{prop:7} and Corollar~\ref{cor:3}, since $\Cl_{3,5}\cong\Cl_{4,4}$.
 

\subsubsection{Integral structure on admissible $\Cl_{2,6}$-module}


The integral structure on the admissible $\Cl_{2,6}$-module exists according to Theorem~\ref{prop:7} and Corollar~\ref{cor:3}, since $\Cl_{2,6}\cong\Cl_{5,3}$.
 

\subsubsection{Integral structure on admissible $\Cl_{1,7}$-module}


The integral structure on the admissible $\Cl_{1,7}$-module exists according to Theorem~\ref{prop:7} and Corollar~\ref{cor:3}, since $\Cl_{1,7}\cong\Cl_{6,2}$.


\subsubsection{Integral structure on admissible $\Cl_{0,8}$-module}


Admissible integral $\Cl_{0,8}$-module was constructed in Section~\ref{sec:0n}.



\section{Admissible modules obtained by tensor product}\label{sec:tensor}



\subsection{Bott periodicity and admissible modules of dimensions $r+s>8$}


In this section we present some theorems that allow to use the Bott periodicity
$$
\Cl_{r+8,s}\cong \Cl_{r+4,s+4}\cong\Cl_{r,s}\otimes\mathbb R(16),\quad
\Cl_{r,s+8}\cong \Cl_{r+4,s+4}\cong\Cl_{r,s}\otimes\mathbb R(16)
$$
of Clifford algebras in order to prove that $\Cl_{r,s}$-modules are integer for $r+s>8$. 

\begin{theorem}\label{th:rs8}
Let us assume that $(V,\langle\cdot\,,\cdot\rangle_V)$ is an
 admissible integral $\Cl_{r,s}$-module 
and $(U,\langle\cdot\,,\cdot\rangle_U)$ is an
 admissible integral $\Cl_{0,8}$-module, where the representations $J_{y_j}\in\End(U)$ permute the integral basis of $U$ up to sign for all orthonormal generators $y_j$ of the Clifford algebra $\Cl_{0,8}$. Then the tensor 
product vector space 
$\big(V\otimes U,\langle\cdot\,,\cdot\rangle_V
\langle\cdot\,,\cdot\rangle_U\big)$ is an admissible integral $\Cl_{r,s+8}$-module.
\end{theorem}

\begin{proof}
Let $(z_1,\ldots,z_r,\zeta_1,\ldots,\zeta_s)$ be orthonormal 
generators of the Clifford algebra $\Cl_{r,s}$ with the quadratic 
form $\mathbb Q_{r,s}(a)=\sum_{i=1}^{r}a_i^2-\sum_{k=1}^{s}\alpha_k^2$ 
for $a=\sum_{i=1}^{r} a_iz_i+\sum_{k=1}^{s} \alpha_k\zeta_k$. 
Let $\{y_1,\ldots,y_8\}$ be orthonormal generators 
for $\Cl_{0,8}$ with quadratic form 
$\mathbb Q_{0,8}(b)=-\sum_{j=1}^{8}b_j^2$ for $b=\sum_{j=1}^{8}b_jy_j$ 
and, finally let $\tilde z_1,\ldots,\tilde
z_{r},\tilde\zeta_1,\ldots,\tilde\zeta_{s+8}$ be orthonormal
generators for the Clifford algebra $\Cl_{r,s+8}$ with quadratic 
form $\mathbb
Q_{r,s+8}(c)=\sum_{i=1}^{r}c_i^2-\sum_{k=1}^{s+8}\varsigma_k^2$ 
for $c=\sum_{i=1}^{r} c_i\tilde z_i+\sum_{k=1}^{s+8} \varsigma_s\tilde\zeta_s$. 

We know that the minimal admissible
$\Cl_{0,8}$-module $(U,\langle\cdot\,,\cdot\rangle_U)$
is isomorphic to $\mathbb R^{8,8}$ with quadratic form 
$\mathbb Q_{8,8}(u)=\sum_{i=1}^{8}u_i^2-\sum_{j=9}^{16}u_j^2$ 
for $u=\sum_{i=1}^{16} u_ie_i$, where $e_i$, $i=1,\ldots,16$ 
is the standard basis in $\mathbb R^{8,8}$. This module is also irreducible.
Then the endomorphisms $J_{y_j}\in\End(\mathbb R^{8,8})$ 
are such that $J_{y_j}^2=\Id_{\mathbb R^{8,8}}$, $j=1,\ldots,8$, $J_{y_i}J_{y_j}=-J_{y_j}J_{y_i}$ for $i\neq j$. 

Now one needs to find $E\in \End(\mathbb R^{8,8})$ satisfying conditions
\begin{equation}\label{eq:E8}
 EJ_{y_j}=-J_{y_j}E, \quad j=1,\ldots,8,\qquad E^2=\Id_{\mathbb R^{8,8}},
\end{equation} 
\begin{equation}\label{eq:eqsym8}
\langle Eu,u'\rangle_{\mathbb R^{8,8}}=\langle u,Eu'\rangle_{\mathbb R^{8,8}}\quad\text{for}\quad u,u'\in\mathbb R^{8,8},
\end{equation} 
where the scalar product in~\eqref{eq:eqsym8} is the scalar product
defined 
by the quadratic form~$\mathbb Q_{8,8}$. Define
$E=\prod_{j=1}^{8}J_{y_j}$ 
to be a volume form for the Clifford algebra $\Cl_{0,8}$. Then it is easy to check that $E$ 
satisfies conditions~\eqref{eq:E8} and~\eqref{eq:eqsym8}.

Denote $\widetilde V=V\otimes\mathbb R^{8,8}$,
$\langle\cdot\,,\cdot\rangle_{\widetilde
  V}=\langle\cdot\,,\cdot\rangle_V\langle\cdot\,,\cdot\rangle_{\mathbb
  R^{8,8}}$ and notice that the scalar product
$\langle\cdot\,,\cdot\rangle_{\widetilde V}$ is non-degenerate. 
Set  
\begin{align*}
\mathbb{R}^{r,s+8}\ni\tilde{z_i}&\mapsto \widetilde J_{\tilde
  z_i}=J_{z_i}\otimes E\in \End(V\otimes\mathbb{R}^{8,8}),\quad i=1,\ldots,r,
\\
\mathbb{R}^{r,s+8}\ni\tilde{\zeta_k}&\mapsto \widetilde J_{\tilde\zeta_k}=J_{\zeta_k}\otimes E\in \End(V\otimes\mathbb{R}^{8,8}),\quad k=1,\ldots,s,
\\
\mathbb{R}^{r,s+8}\ni\tilde{\zeta}_{s+j}&\mapsto\widetilde J_{\tilde\zeta_{s+j}}=\Id_V\otimes J_{y_j}\in \End(V\otimes\mathbb{R}^{8,8}),\quad j=1,\ldots,8,
\end{align*}
where $J_{z_i},J_{\zeta_k}\in\End(V)$, $i=1,\ldots,r$, $k=1,\ldots,s$, such that $J_{z_i}^2=-\Id_V$, $J_{\zeta_k}^2=\Id_V$, and $J_{y_j}\in\End(\mathbb R^{8,8})$, $j=1,\ldots,8$, such that $J_{y_j}^2=\Id_{\mathbb R^{8,8}}$. Then, it is easy to see that
$$
\widetilde J_{\tilde z_i}^2=-\Id_{\widetilde V}\ \text{ for }\ i=1,\ldots,r,\qquad \widetilde J_{\tilde\zeta_k}^2=\Id_{\widetilde V}\ \text{ for }\ k=1,\ldots,s+8.
$$
Moreover, we have for $v\otimes e\in V\otimes\mathbb{R}^{8,8}$,

$$
\widetilde J_{\tilde \zeta_{s+j_1}}\widetilde J_{\tilde \zeta_{s+j_2}}v\otimes e
=v\otimes J_{y_{j_1}}J_{y_{j_2}}e=-v\otimes J_{y_{j_2}}J_{y_{j_1}}e
=-\widetilde J_{\tilde \zeta_{s+j_2}}\widetilde J_{\tilde \zeta_{s+j_1}}v\otimes e
$$
for $j_1,j_2=1,\ldots,8$, $j_1\neq j_2$. 
$$
\widetilde J_{\tilde \zeta_k}\widetilde J_{\tilde \zeta_{s+j}}v\otimes e=J_{\zeta_k}v\otimes EJ_{y_j}e=-
J_{\zeta_k}v\otimes J_{y_j}Ee=-\widetilde J_{\tilde \zeta_{s+j}}\widetilde J_{\widetilde \zeta_k}v\otimes e
$$
for $k=1,\ldots,s$, $j=1,\ldots, 8$.
$$
\widetilde J_{\tilde z_i}\widetilde J_{\tilde \zeta_{s+j}}v\otimes e=J_{z_i}v\otimes EJ_{y_j}e=-
J_{z_i}v\otimes J_{y_j}Ee=-\widetilde J_{\tilde \zeta_{s+j}}\widetilde J_{\widetilde z_i}v\otimes e
$$
for $i=1,\ldots,r$, $j=1,\ldots, 8$.
$$
\widetilde J_{\tilde z_{i_1}}\widetilde J_{\tilde z_{i_2}}v\otimes e=J_{z_{i_1}}J_{z_{i_2}}v\otimes E^2e=-
J_{z_{i_2}}J_{z_{i_1}}v\otimes E^2e=-\widetilde J_{\tilde z_{i_2}}\widetilde J_{\tilde z_{i_1}}v\otimes e
$$
for $i_1,i_2=1,\ldots,r$, $i_1\neq i_2$.
$$
\widetilde J_{\tilde \zeta_{k_1}}\widetilde J_{\tilde \zeta_{k_2}}v\otimes e=J_{\zeta_{k_1}}J_{\zeta_{k_2}}v\otimes E^2e=-
J_{\zeta_{k_2}}J_{\zeta_{k_1}}v\otimes E^2e=-\widetilde J_{\tilde \zeta_{k_2}}\widetilde J_{\tilde \zeta_{k_1}}v\otimes e
$$
for $k_1,k_2=1,\ldots,s$, $k_1\neq k_2$.
$$
\widetilde J_{\tilde z_{i}}\widetilde J_{\tilde \zeta_{k}}v\otimes e=J_{z_{i}}J_{\zeta_{k}}v\otimes E^2e=-
J_{\zeta_{k}}J_{z_{i}}v\otimes E^2e=-\widetilde J_{\tilde \zeta_{k}}\widetilde J_{\tilde z_{i}}v\otimes e
$$
for $i=1,\ldots,r$, $k=1,\ldots,s$.

The next step is to verify that the scalar product $\langle\cdot\,,\cdot\rangle_{\widetilde V}=\langle\cdot\,,\cdot\rangle_V\langle\cdot\,,\cdot\rangle_{\mathbb R^{8,8}}$ satisfies
$$
\langle \widetilde J_{\tilde z}\tilde v,\tilde v'\rangle_{\widetilde V}+\langle \tilde v,\widetilde J_{\tilde z}\tilde v'\rangle_{\widetilde V}=0.
$$
We write $\tilde z=a+b$, where $a\in\mathbb R^{r,s}$, $b\in \mathbb R^{0,8}$, and $\tilde v=v\otimes u$, $\tilde v'=v'\otimes u'$ for $v,v'\in V$, $u,u'\in \mathbb R^{8,8}$. Then \begin{align}\label{al:skew}
&\langle \widetilde J_{\tilde z}\tilde v,
\tilde v'\rangle_{\widetilde{V}} +\langle \tilde v,
\widetilde J_{\tilde z}\tilde v'\rangle_{\widetilde V}
=
\langle (\widetilde J_{a}+\widetilde J_{b})
\tilde v,\tilde v'\rangle_{\widetilde V}
+\langle \tilde v,(\widetilde J_{a}+\widetilde J_{b})\tilde v'\rangle_{\widetilde V} 
\\
&=
\langle \widetilde J_{a}\tilde v,
\tilde v'\rangle_{\widetilde V}
+\langle \widetilde J_{b}\tilde v,\tilde v'\rangle_{\widetilde V} 
+
\langle \tilde v,\widetilde J_{a}
\tilde v'\rangle_{\widetilde V}+\langle \tilde v,\widetilde J_{b}\tilde v'\rangle_{\widetilde V} \nonumber
\\
&=
\langle J_{a} v\otimes Eu,v'\otimes u'\rangle_{\widetilde V}+\langle v\otimes u,J_{a}v'\otimes Eu'\rangle_{\widetilde V} \nonumber
\\
&+
\langle v\otimes J_{b} u,v'\otimes u'\rangle_{\widetilde V}+\langle v\otimes u,v'\otimes J_{b}u'\rangle_{\widetilde V} \nonumber
\\
&=
(J_{a} v,v')_{V}\langle  Eu,u'\rangle_{\mathbb R^{8,8}}+(v,J_{a}v')_{V}\langle  u,Eu'\rangle_{\mathbb R^{8,8}} \nonumber
\\
&+
(v,v')_{V}\Big(\underbrace{\langle J_{b}u,u'\rangle_{\mathbb R^{8,8}}+\langle u,J_{b}u'\rangle_{\mathbb R^{8,8}}}_{=0}\Big) \nonumber
\\
&=
\Big((J_{a} v,v')_{V}+(v,J_{a}v')_{V}\Big)\langle  Eu,u'\rangle_{\mathbb R^{8,8}}=0. \nonumber
\end{align}

To show that the resulting $\Cl_{r,s+8}$-module is integral 
we assume that both modules $(V,\langle\cdot\,,\cdot\rangle_V)$, $(\mathbb R^{8,8},\langle\cdot\,,\cdot\rangle_{\mathbb R^{8,8}})$ are integral. 
Then, if $\{v_{\alpha}\}_{\alpha=1}^{\dim V}$ and
$\{e_{p}\}_{p=1}^{16}$ 
are integral bases for $V$ and $\mathbb R^{8,8}$ respectively, we denote by $\{\tilde v_n=v_{\alpha}\otimes e_p\}_{n=1}^{16\dim V}$ the basis of $\widetilde V$. 
We assumed that the maps $J_{y_j}$, $j=1,\ldots, 8$, permute 
the basis $\{e_{p}\}_{p=1}^{16}$ up to sign. Then the map $E$ 
also permutes the basis $\{e_{p}\}_{p=1}^{16}$. We have 
\begin{equation*}
\langle\widetilde J_{\tilde z_i}\tilde v_{n},\tilde v_{m}\rangle_{\widetilde V}=
\langle J_{z_i}v_{\alpha}\otimes Ee_{p},v_{\beta}\otimes e_{q}\rangle_{\widetilde V}=
\langle J_{z_i}v_{\alpha}, v_{\beta}\rangle_V\cdot \langle Ee_{p},e_{q}\rangle_{\mathbb R^{8,8}}=\pm1\text{ or }0,
\end{equation*}
\begin{equation*}
\langle\widetilde J_{\tilde \zeta_k}\tilde v_{n},\tilde v_{m}\rangle_{\widetilde V}=
\langle J_{\zeta_k}v_{\alpha}\otimes Ee_{p},v_{\beta}\otimes e_{q}\rangle_{\widetilde V}=
\langle J_{\zeta_k}v_{\alpha}, v_{\beta}\rangle_V\cdot \langle Ee_{p},e_{q}\rangle_{\mathbb R^{8,8}}=\pm1\text{ or }0,
\end{equation*} for all $i=1,\ldots,r$, $k=1,\ldots,s$ and $n,m=1,\ldots,16\dim V$. Analogously
\begin{equation*}
\langle\widetilde J_{\tilde \zeta_{s+j}}\tilde v_{n},\tilde v_{m}\rangle_{\widetilde V}=
\langle v_{\alpha}\otimes J_{y_j}e_{p},v_{\beta}\otimes e_{q}\rangle_{\widetilde V}=
\langle v_{\alpha}, v_{\beta}\rangle_V\cdot \langle J_{y_j}e_{p},e_{q}\rangle_{\mathbb R^{8,8}}=\pm1\text{ or }0,
\end{equation*}
 for all $j=1,\ldots,8$ and $n,m=1,\ldots,16\dim V$.
 \end{proof}

\begin{theorem}\label{th:r8s}
Let us assume that $(V,\langle\cdot\,,\cdot\rangle_V)$ is an admissible integral $\Cl_{r,s}$-module and $(U,(\cdot\,,\cdot)_U)$ is an admissible integral $\Cl_{8,0}$-module, where the representations $J_{y_j}\in\End(U)$ permute the integral basis of $U$ up to sign for all orthonormal generators $y_j$ of the Clifford algebra $\Cl_{8,0}$. Then the scalar product vector space $\big(V\otimes U,\langle\cdot\,,\cdot\rangle_V(\cdot\,,\cdot)_U\big)$ is an admissible integral $\Cl_{r+8,s}$-module.
\end{theorem}

\begin{proof}
Let $(z_1,\ldots,z_r,\zeta_1,\ldots,\zeta_s)$ be 
orthonormal generators of the Clifford algebra $\Cl_{r,s}$ 
with the quadratic form $\mathbb
Q_{r,s}(a)=\sum_{i=1}^{r}a_i^2-\sum_{k=1}^{s}\alpha_k^2$ for
$a=\sum_{i=1}^{r} a_iz_i+\sum_{k=1}^{s} \alpha_k\zeta_k$. Let
$\{y_1,\ldots,y_8\}$ be orthonormal 
generators for $\Cl_{8,0}$ with quadratic form $\mathbb Q_{8,0}(b)=\sum_{j=1}^{8}b_j^2$ for $b=\sum_{j=1}^{8}b_jy_j$ and, finally let $\tilde z_1,\ldots,\tilde z_{r+8},\tilde\zeta_1,\ldots,\tilde\zeta_{s}$ be orthonormal generators for the Clifford algebra $\Cl_{r+8,s}$ with quadratic form $\mathbb Q_{r+8,s}(c)=\sum_{i=1}^{r+8}c_i^2-\sum_{k=1}^{s}\varsigma_k^2$ for $c=\sum_{i=1}^{r+8} c_i\tilde z_i+\sum_{k=1}^{s} \varsigma_k\tilde\zeta_k$. 

We know that the minimal admissible
$\Cl_{8,0}$-module $(U,(\cdot\,,\cdot)_U)$
is isomorphic to $\mathbb R^{16,0}=\mathbb R^{16}$ with quadratic form 
$\mathbb Q_{16}(u)=\sum_{i=1}^{16}u_i^2$ 
for $u=\sum_{i=1}^{16} u_ie_i$, where $e_i$, $i=1,\ldots,16$ 
is the standard basis in $\mathbb R^{16}$. This module is also irreducible.
Then for the endomorphisms $J_{x_j}\in \End(\mathbb R^{16})$ 
we have $J_{x_j}^2=-\Id_{\mathbb R^{16}}$, $j=1,\ldots,8$, $J_{x_i}J_{x_j}=-J_{x_j}J_{x_i}$ for $i\neq j$. 

Now we want to find $\mathscr E\in \End(\mathbb R^{16})$ satisfying conditions
\begin{equation}\label{eq:E8_2}
 \mathscr EJ_{x_j}=-J_{x_j}\mathscr E, \quad j=1,\ldots,8,\qquad \mathscr E^2=\Id,
\end{equation} 
\begin{equation}\label{eq:eqsym8_2}
(\mathscr Eu,u')_{\mathbb R^{16}}=(u,\mathscr Eu')_{\mathbb R^{16}}\quad\text{for}\quad u,u'\in\mathbb R^{10}
\end{equation} 
where the inner product in~\eqref{eq:eqsym8_2} is the standard Euclidean product. Define $\mathscr E=\prod_{j=1}^{8}J_{y_j}$ to be a volume form for $\Cl_{8,0}$. Then it is easy to check that $\mathscr E$ satisfies conditions~\eqref{eq:E8_2} and~\eqref{eq:eqsym8_2}.

Denote $\widetilde V=V\otimes\mathbb R^{16}$,
$\langle\cdot\,,\cdot\rangle_{\widetilde
  V}=\langle\cdot\,,\cdot\rangle_V(\cdot\,,\cdot)_{\mathbb
  R^{16}}$ and notice that the scalar product
$\langle\cdot\,,\cdot\rangle_{\widetilde V}$ is non-degenerate. 
Set  
\begin{align*}
\mathbb{R}^{r+8,s}\ni\tilde{z_i}&\mapsto \widetilde J_{\tilde
  z_i}=J_{z_i}\otimes \mathscr E\in \End(V\otimes\mathbb{R}^{16}),\quad i=1,\ldots,r,
\\
\mathbb{R}^{r+8,s}\ni\tilde{\zeta_k}&\mapsto \widetilde J_{\tilde\zeta_k}=J_{\zeta_k}\otimes \mathscr E\in \End(V\otimes\mathbb{R}^{16}),\quad k=1,\ldots,s,
\\
\mathbb{R}^{r+8,s}\ni\tilde{z}_{r+j}&\mapsto\widetilde J_{\tilde z_{r+j}}=\Id_V\otimes J_{x_j}\in \End(V\otimes\mathbb{R}^{16}),\quad j=1,\ldots,8,
\end{align*}
where $J_{z_i},J_{\zeta_k}\in\End(V)$, $i=1,\ldots,r$, $k=1,\ldots,s$, such that $J_{z_i}^2=-\Id_V$, $J_{\zeta_k}^2=\Id_V$, and $J_{x_j}\in\End(\mathbb R^{16})$, $j=1,\ldots,8$, such that $J_{x_j}^2=-\Id_{\mathbb R^{16}}$. Then, we finish the proof as in Theorem~\ref{th:rs8}.
\end{proof}

\begin{theorem}\label{th:r4s4}
Let us assume that $(V,\langle\cdot\,,\cdot\rangle_V)$ is an
 admissible integral $\Cl_{r,s}$-module 
and $(U,\langle\cdot\,,\cdot\rangle_U)$ is an
 admissible integral $\Cl_{4,4}$-module, where the representations $J_{y_j}\in\End(U)$ permute the integral basis of $U$ up to sign for all orthonormal generators $y_j$ of the Clifford algebra $\Cl_{4,4}$. Then the tensor 
product vector space 
$\big(V\otimes U,\langle\cdot\,,\cdot\rangle_V
\langle\cdot\,,\cdot\rangle_U\big)$ is an admissible integral $\Cl_{r+4,s+4}$-module.
\end{theorem}

\begin{proof}
Let $(z_1,\ldots,z_r,\zeta_1,\ldots,\zeta_s)$ be 
orthonormal generators of the Clifford algebra $\Cl_{r,s}$ 
with the quadratic form $\mathbb
Q_{r,s}(a)=\sum_{i=1}^{r}a_i^2-\sum_{k=1}^{s}\alpha_k^2$ for
$a=\sum_{i=1}^{r} a_iz_i+\sum_{k=1}^{s} \alpha_k\zeta_k$. 
Let the collection $\{x_1,x_2,x_3,x_4,y_1,y_2,y_3,y_4\}$ be orthonormal generators for 
$\Cl_{4,4}$ with quadratic form 
$\mathbb{Q}_{4,4}(b)=\sum_{i=1}^{4}b_i^2-\sum_{j=1}^{4}b^2_j$ 
for $b=\sum_{i=1}^{4}b_ix_i+\sum_{j=1}^{4}b_jy_j$. Denote
$\tilde z_1,\ldots,\tilde z_{r+4},\tilde\zeta_1,\ldots,\tilde\zeta_{s+4}$ orthonormal generators for the Clifford algebra $\Cl_{r+4,s+4}$ with quadratic form $\mathbb Q_{r+4,s+4}(c)=\sum_{i=1}^{r+4}c_i^2-\sum_{k=1}^{s+4}\varsigma_k^2$ for $c=\sum_{i=1}^{r+4} c_i\tilde z_i+\sum_{k=1}^{s+4} \varsigma_k\tilde\zeta_k$. 

We know that the minimal admissible
$\Cl_{4,4}$-module $(U,\langle \cdot\,,\cdot\rangle_U)$
is isomorphic to $\mathbb R^{8,8}$ with quadratic form 
$\mathbb Q_{8,8}$.  This module is also irreducible.
Then for the endomorphisms $J_{x_i},J_{y_j}\in \End(\mathbb R^{8,8})$ 
we have $J_{x_i}^2=-\Id_{\mathbb R^{8,8}}$, $J_{y_j}^2=\Id_{\mathbb R^{8,8}}$, $j=1,\ldots,4$, and moreover all of them mutually anti-commute. 

We define the endomorphism $\mathcal E\colon\mathbb R^{8,8}\to \mathbb R^{8,8}$ by $\mathcal E=\prod_{i=1}^{4}J_{x_i}\prod_{j=1}^{4}J_{y_j}$ or in other words $\mathcal E$ is the volume form for $\Cl_{4,4}$. Then 
\begin{equation}\label{eq:F}
\mathcal E^2=\Id_{\mathbb R^{8,8}},\qquad \mathcal EJ_{x_i}=-J_{x_i}\mathcal E,\qquad \mathcal EJ_{y_i}=-J_{y_i}\mathcal E,\quad i=1,2,3,4,
\end{equation}
and 
\begin{equation}\label{eq:Fproduct}
\langle \mathcal Eu, u'\rangle_{\mathbb R^{8,8}}=\langle u, \mathcal Eu'\rangle_{\mathbb R^{8,8}}\quad\text{for all}\quad u,u'\in\mathbb R^{8,8}.
\end{equation}

Denote $\widetilde V=V\otimes U\cong V\otimes\mathbb R^{8,8}$ and non-degenerate scalar product $\langle\cdot\,,\cdot\rangle_{\widetilde V}=\langle\cdot\,,\cdot\rangle_{V}\langle\cdot\,,\cdot\rangle_{\mathbb R^{8,8}}$. The space $\widetilde V$ has dimension $16\dim V$. Set also 
$$
\begin{array}{lll}
&\widetilde J_{\tilde z_k}=J_{z_k}\otimes \mathcal E,\quad k=1,\ldots, r,\qquad &\widetilde J_{\tilde \zeta_l}=J_{\zeta_l}\otimes \mathcal E,\quad l=1,\ldots,s,
\\
& \widetilde J_{\tilde z_{k+i}}=\Id_{V}\otimes J_{x_i},\quad i=1,2,3,4,
&\widetilde J_{\tilde \zeta_{s+j}}=\Id_{V}\otimes J_{y_j}\quad j=1,2,3,4.
\end{array}
$$
Then it is easy to see that
$
\widetilde J_{\tilde z_i}^2=-\Id_{\widetilde V}$, $i=1,\ldots,r+4$, $\widetilde J_{\tilde\zeta_j}^2=\Id_{\widetilde V}$ for $j=1,\ldots,s+4$.
It can be shown as in Theorem~\ref{th:rs8} that all $\widetilde J_{\tilde z_i}$ $\widetilde J_{\tilde \zeta_j}$ mutually anti-commute and the module $\widetilde V=V\otimes U\cong V\otimes\mathbb R^{8,8}$ is admissible integral $\Cl_{r+4,s+4}$-module.
\end{proof}

\begin{proposition}
If the admissible integral $\Cl_{r,s}$-module $(V,\langle\cdot\,,\cdot\rangle_V)$ in Theorems~\ref{th:rs8},~\ref{th:r8s}, and~\ref{th:r4s4} is of minimal dimension then the resulting $\Cl_{r,s+8}$-module $\big(V\otimes \mathbb R^{8,8},\langle\cdot\,,\cdot\rangle_V\langle\cdot\,,\cdot\rangle_{\mathbb R^{8,8}}\big)$, $\Cl_{r+8,s}$-module $\big(V\otimes \mathbb R^{16},\langle\cdot\,,\cdot\rangle_V(\cdot\,,\cdot)_{\mathbb R^{16}}\big)$, and  $\Cl_{r+4,s+4}$-module $\big(V\otimes \mathbb R^{8,8},\langle\cdot\,,\cdot\rangle_V\langle\cdot\,,\cdot\rangle_{\mathbb R^{8,8}}\big)$ are minimal admissible integral modules.
\end{proposition}

\begin{proof}
The admissible integral $\Cl_{0,8}$-module $(\mathbb R^{8,8},\langle\cdot\,,\cdot\rangle_{\mathbb R^{8,8}})$ is of minimal dimension equals $16$. Since admissible and irreducible modules has periodicity $8$ the resulting $\Cl_{r,s+8}$-module $\widetilde V$ will have minimal dimension equals $\dim \widetilde V=16\dim V$. 

Similar arguments used in the cases of $\Cl_{r+8,s}$ and $\Cl_{r+4,s+4}$-modules.
\end{proof}


\subsection{Twisted tensor product}\label{twisted tensor}


In this subsection, we give two  methods of construction of an admissible module
from a given admissible module of lower dimensions by making use of tensor product.
We show two cases, that is 
a construction of an admissible module 
for $\Cl_{0,n+2}$ from that of $\Cl_{n,0}$ and $\Cl_{0,2}$.
Another one is the construction of an admissible module of $\Cl_{r+1,s+1}$ 
from that of $\Cl_{r,s}$ and $\Cl_{1,1}$. Both methods also give us the
integral structure from those of the lower dimensions. Note that these constructions not always give the minimal dimensional resulting module, even if the initial admissible modules $\Cl_{n,0}$ (or $\Cl_{r,s}$) and $\Cl_{0,2}$ (and $\Cl_{1,1}$) are of minimal dimensions. The resulting module can exceed the minimal dimension two or four times.

First,
basing on the isomorphism $\Cl_{n,0}\otimes\Cl_{0,2}\cong\Cl_{0,n+2}$
we prove the following theorem.
\begin{theorem}\label{th:higher1}
Let us assume that $(V,(\cdot\,,\cdot)_V)$ is an admissible integral $\Cl_{n,0}$-module and $(U,\langle\cdot\,,\cdot\rangle_U)$ is an admissible integral $\Cl_{0,2}$-module. Then the scalar product vector space $\big(V\otimes U,(\cdot\,,\cdot)_V\langle\cdot\,,\cdot\rangle_U\big)$ is an admissible integral $\Cl_{0,n+2}$-module.
\end{theorem}
\begin{proof}
Let $(z_1,\ldots,z_n)$ be orthonormal generators of the Clifford algebra $\Cl_{n,0}$ with the quadratic form $\mathbb Q_{n,0}(a)=\sum_{i=1}^{n}a_i^2$ for $a=\sum_{i=1}^{n} a_iz_i$. Let $\{y_1,y_2\}$ be orthonormal generators for $\Cl_{0,2}$ with quadratic form $\mathbb Q_{0,2}(b)=-b_1^2-b^2_2$ for $b=b_1y_1+b_2y_2$ and, finally let $\tilde z_1,\ldots,\tilde z_{n+2}$ be orthonormal generators for the Clifford algebra $\Cl_{0,n+2}$ with quadratic form $\mathbb Q_{0,n+2}(c)=-\sum_{i=1}^{n+2}c_i^2$ for $c=\sum_{i=1}^{n+2} c_i\tilde z_i$. The map
$$
\begin{cases}
\tilde z_i\mapsto z_i\otimes y_1y_2,\quad\text{if}\quad i=1,\ldots,n,
\\
\tilde z_{n+1}\mapsto  1\otimes y_1,
\\
\tilde z_{n+2}\mapsto 1\otimes y_2.
\end{cases}
$$
defines the isomorphism between the Clifford algebras $\Cl_{0,n+2}$ and $\Cl_{n,0}\otimes \Cl_{0,2}$.

We saw in the previous section that the minimal admissible $\Cl_{0,2}$-module $(U,\langle\cdot\,,\cdot\rangle_U)$ is isomorphic to $\mathbb R^{2,2}$ with quadratic form $\mathbb Q_{2,2}(u)=\sum_{i=1,2}u_i^2-\sum_{j=3,4}u_j^2$ for $u=\sum_{i=1}^{4} u_ie_i$, where $e_i$, $i=1,2,3,4$ is the standard basis in $\mathbb R^{2,2}$. Then the endomorphisms $J_{y_1}$ and $J_{y_2}$ from $\End(\mathbb R^{2,2})$ are written in the basis $\{e_i\}_{i=1}^{4}$ as follows 
$$
J_{y_1}=\begin{pmatrix} 
0 & 0 & 1 & 0 
\\ 
0 & 0 & 0 & 1 
\\
1 & 0 & 0 & 0 
\\
0 & 1 & 0 & 0 
\end{pmatrix}\qquad
J_{y_2}=\begin{pmatrix} 
0 & 0 & 0 & 1 
\\ 
0 & 0 & -1 & 0 
\\
0 & -1 & 0 & 0 
\\
1 & 0 & 0 & 0 
\end{pmatrix}.
$$
We have $J_{y_1}^2=\Id$, $J_{y_2}^2=\Id$, $J_{y_1}J_{y_2}=-J_{y_2}J_{y_1}$, and 
$$
J_{y_1}J_{y_2}e_1=e_2,\quad J_{y_1}e_1=e_3,\quad J_{y_2}e_1=e_4.
$$

Now we want to find $ \mathcal F\in \End(\mathbb R^{2,2})$ satisfying conditions
\begin{equation}\label{eq:E2}
 \mathcal F J_{y_1}=-J_{y_1}\mathcal F,\qquad\ \ \ \mathcal FJ_{y_2}=-J_{y_2}\mathcal F,\qquad \mathcal F^2=-\Id,
\end{equation} 
\begin{equation}\label{eq:eq11}
\langle \mathcal Fu,u'\rangle_{\mathbb R^{2,2}}=\langle u,\mathcal Fu'\rangle_{\mathbb R^{2,2}}\quad\text{for}\quad u,u'\in\mathbb R^{2,2}
\end{equation} 
where the scalar product in~\eqref{eq:eq11} is the scalar product defined by the quadratic form $\mathbb Q_{2,2}$.
Conditions~\eqref{eq:E2} imply that the matrix for $\mathcal F$ has the form
\begin{equation}\label{eq:E}
\begin{pmatrix}
a & b & c & d
\\
-b & a & d & -c
\\
-c & -d & -a & -b
\\
-d & c & b & -a
\end{pmatrix}\qquad\text{with}\quad
\begin{array}{lllll}
& a^2-b^2-c^2-d^2=-1
\\
& ab=0,\ \ bc=0,\ \  bd=0.
\end{array}
\end{equation}
Checking the condition~\eqref{eq:eq11} we find that $b=0$.

Denote $\widetilde V=V\otimes\mathbb R^{2,2}$, $\langle\cdot\,,\cdot\rangle_{\widetilde V}=(\cdot\,,\cdot)_V\langle\cdot\,,\cdot\rangle_{\mathbb R^{2,2}}$ and notice that the scalar product $\langle\cdot\,,\cdot\rangle_{\widetilde V}$ is non-degenerate. Set  
\begin{align*}
&\widetilde J_{\widetilde z_i}=J_{z_i}\otimes \mathcal F,\quad i=1,\ldots,n,
\\
& \widetilde J_{\widetilde z_{n+1}}=\Id_V\otimes J_{y_1},\quad\text{and}\quad \widetilde J_{\widetilde z_{n+2}}=\Id_V\otimes J_{y_2},
\end{align*}
where $J_{z_i}\in\End(V)$, $i=1,\ldots,n$, such that $J_{z_i}^2=-\Id$. Then, it is easy to see that
$
\widetilde J_{\tilde z_i}^2=\Id_{\widetilde V}$ for $i=1,\ldots,n+2$.
Moreover, similar to discussions in Theorem~\ref{th:rs8} we have
$$
\widetilde J_{\tilde z_{n+1}}\widetilde J_{\tilde z_{n+2}}=-\widetilde J_{\tilde z_{n+2}}\widetilde J_{\tilde z_{n+1}},\quad
\widetilde J_{\tilde z_i}\widetilde J_{\tilde z_{n+1}}=-\widetilde J_{\tilde z_{n+1}}\widetilde J_{\widetilde z_i},\quad
\widetilde J_{\tilde z_i}\widetilde J_{\tilde z_{n+2}}=-\widetilde J_{\tilde z_{n+2}}\widetilde J_{\widetilde z_i},\quad
\widetilde J_{\tilde z_i}\widetilde J_{\tilde z_j}=-\widetilde J_{\tilde z_j}\widetilde J_{\tilde z_i},
$$
for $i,j=1,\ldots,n$, $i\neq j$.

The next step is to verify that the scalar 
product $\langle\cdot\,,\cdot\rangle_{\widetilde V}
=(\cdot\,,\cdot)_V\cdot\langle\cdot\,,\cdot\rangle_{\mathbb R^{2,2}}$ satisfies
$$
\langle \widetilde J_{\tilde z}\tilde v,\tilde v'\rangle_{\widetilde
  V}
+\langle \tilde v,\widetilde J_{\tilde z}\tilde v'\rangle_{\widetilde V}=0.
$$
We write $\tilde z=a+b$, 
where $a\in\mathbb R^{n,0}$, $b\in \mathbb R^{0,2}$, 
and $\tilde v=v\otimes u$, $\tilde v'=v'\otimes u'$ 
for $v,v'\in V$, $u,u'\in \mathbb R^{2,2}$. Then we argue as in~\eqref{al:skew}.

To show that the resulting $\Cl_{0,n+2}$-module is integral we assume
that both modules $(V,\langle\cdot\,,\cdot\rangle_V)$, $(\mathbb
R^{2,2},\langle\cdot\,,\cdot\rangle_{\mathbb R^{2,2}})$ are integral
and choose special form of the map $\mathcal F$, for instance we set
$a=d=0$ and $c=1$. Then, if $\{v_{\alpha}\}$ and $\{e_p\}$ are
integral 
bases for $V$ and $\mathbb R^{2,2}$ respectively, then for $\tilde v_n=v_{\alpha}\otimes e_p$
\begin{equation*}
\langle\widetilde J_{\tilde z_i}\tilde v_n,\tilde v_m\rangle_{\widetilde V}=
\langle J_{z_i}v_{\alpha}\otimes \mathcal Fe_p,v_{\beta}\otimes e_q\rangle_{\widetilde V}=
\langle J_{z_i}v_{\alpha}, v_{\beta}\rangle_V\cdot \langle \mathcal Fe_{p},e_{q}\rangle_{\mathbb R^{2,2}}=\pm1\text{ or }0.
\end{equation*}
\end{proof}

\begin{remark}\label{rem:J1J2}
Note that $\mathcal F\neq\pm J_{y_1}J_{y_2}$. 
\end{remark}

\begin{theorem}\label{th:higher2}
Let us assume that $(V,\langle\cdot\,,\cdot\rangle_V)$ is an
admissible integral $\Cl_{r,s}$-module 
and $(U,\langle\cdot\,,\cdot\rangle_U)$ is the minimal dimensional
admissible integral $\Cl_{1,1}$-module, 
then $\big(V\otimes U, \langle\cdot\,,\cdot\rangle_V\langle\cdot\,,\cdot\rangle_U\big)$ is an admissible integral $\Cl_{r+1,s+1}$-module.
\end{theorem}

\begin{proof}
Let $(z_1,\ldots,z_r, \zeta_{1},\ldots,\zeta_{s})$ be orthonormal generators of the Clifford algebra $\Cl_{r,s}$ with the quadratic form $\mathbb Q_{r,s}(a)=\sum_{i=1}^{r}a_i^2-\sum_{j=1}^{s}a_j^2$ for $a=\sum_{i=1}^{r} a_iz_i+\sum_{j=1}^{s} a_j\zeta_j$ and let $\{x,y\}$ be orthonormal generators for $\Cl_{1,1}$ with quadratic form $\mathbb Q_{2,2}(b)=b_1^2-b^2_2$ for $b=b_1x+b_2y$. Denote by $(\tilde z_1,\ldots,\tilde z_{r+1}, \tilde \zeta_{1},\ldots,\tilde \zeta_{s+1})$ orthonormal generators for the Clifford algebra $\Cl_{r+1,s+1}$ with the quadratic form $\mathbb Q_{r+1,s+1}(c)=\sum_{i=1}^{r+1}c_i^2-\sum_{j=1}^{s+1}c_j^2$ for $c=\sum_{i=1}^{r+1} c_i\tilde z_i+\sum_{j=1}^{s+1} c_j\tilde\zeta_j$. It is known that there is the isomorphism between $\Cl_{r+1,s+1}$ and $\Cl_{r,s}\otimes\Cl_{1,1}$ given by the following relation between the generators $\tilde z_1,\ldots, \tilde z_{r+1},\tilde\zeta_1,\ldots,\tilde\zeta_{s+1}$ of $\Cl_{r+1,s+1}$ and generators of $\Cl_{r,s}\otimes\Cl_{1,1}$:
$$  
\begin{cases}
\tilde z_i\cong z_i\otimes xy,\quad\text{if}\quad i=1,\ldots,r,
\\
\tilde\zeta_j\cong \zeta_i\otimes xy,\quad\text{if}\quad j=1,\ldots,s,
\\
\tilde z_{r+1}\cong 1\otimes x,
\\
\tilde \zeta_{s+1}\cong 1\otimes y.
\end{cases}
$$

We saw in the previous section that the admissible $\Cl_{1,1}$-module $(U,\langle\cdot\,,\cdot\rangle_U)$ is isomorphic to $\mathbb R^{2,2}$ with quadratic form $\mathbb Q_{2,2}(u)=\sum_{i=1}^{2}u_i^2-\sum_{j=3}^{4}u_j^2$ for $u=\sum_{i=1}^{4} u_ie_i$, where $e_i$, $i=1,2,3,4$ is the standard basis in $\mathbb R^{2,2}$. Then the endomorphisms $J_{y_1}$ and $J_{y_2}$ from $\End(\mathbb R^{2,2})$ are written in the basis $\{e_i\}_{i=1}^{4}$ as follows 
$$
J_{x}=\begin{pmatrix} 
0 & -1 & 0 & 0 
\\ 
1 & 0 & 0 & 0 
\\
0 & 0 & 0 & 1 
\\
0 & 0 & -1 & 0 
\end{pmatrix}\quad
J_{y}=\begin{pmatrix} 
0 & 0 & 1 & 0 
\\ 
0 & 0 & 0 & 1 
\\
1 & 0 & 0 & 0 
\\
0 & 1 & 0 & 0 
\end{pmatrix}\quad
J_{y}J_{x}=\begin{pmatrix} 
0 & 0 & 0 & 1 
\\ 
0 & 0 & -1 & 0 
\\
0 & -1 & 0 & 0 
\\
1 & 0 & 0 & 0 
\end{pmatrix}.
$$
Then $J_{x}^2=-\Id$, $J_{y}^2=\Id$, $J_{x}J_{y}=-J_{y}J_{x}$, and 
$
J_{x}e_1=e_2$, $J_ye_1=e_3$, $J_{y}J_{x}e_1=e_4$.

We need to find the endomorphism $F\colon\mathbb R^{2,2}\to \mathbb R^{2,2}$ such that 
\begin{equation}\label{eq:F11}
F^2=\Id,\qquad FJ_{x}=-J_{x}F,\qquad FJ_{y}=-J_{y}F,
\end{equation}
\begin{equation}\label{eq:Fproduct11}
\langle Fu, u'\rangle_{\mathbb R^{2,2}}=\langle u, Fu'\rangle_{\mathbb R^{2,2}}\quad\text{for all}\quad u,u'\in\mathbb R^{2,2}.
\end{equation}

Checking conditions~\eqref{eq:F11} and~\eqref{eq:Fproduct11} we find that matrix for $F$ has the form
$$
\begin{pmatrix}
a&b&c&d
\\
b&-a&d&-c
\\
-c&-d&-a,&-b
\\
-d&c&-b&a
\end{pmatrix}\qquad\text{with}\qquad 
\begin{array}{ll}
&a^2+b^2-c^2-d^2=1
\\
& bc=ad.
\end{array}
$$

Denote $\widetilde V=V\otimes\mathbb R^{2,2}$ and non-degenerate scalar product $\langle\cdot\,,\cdot\rangle_{\widetilde V}=\langle\cdot\,,\cdot\rangle_V\cdot\langle\cdot\,,\cdot\rangle_{\mathbb R^{2,2}}$. Set also 
$$
\begin{array}{lll}
&\widetilde J_{\tilde z_i}=J_{z_i}\otimes F,\quad i=1,\ldots,r,\qquad
&\widetilde J_{\tilde \zeta_j}=J_{\zeta_i}\otimes F,\quad j=1,\ldots,s,
\\
& \widetilde J_{\tilde z_{r+1}}=\Id_V\otimes J_{x},
&\widetilde J_{\tilde \zeta_{s+1}}=\Id_V\otimes J_{y}.
\end{array}
$$
Here $J_{z_i},J_{\zeta_j}\in \End(V)$, $i=1,\ldots,r$, $j=1,\ldots,s$ are such that $J_{z_i}^2=-\Id$, $J_{\zeta_j}^2=\Id$. Then it is easy to see that
$
\widetilde J_{\tilde z_i}^2=-\Id_{\widetilde V}$, $\widetilde J_{\tilde\zeta_j}^2=\Id_{\widetilde V}$
for $i=1,\ldots,r+1,\ j=1,\ldots,s+1$ due to~\eqref{eq:F11} and $J^2_x=-\Id_{\mathbb R^{2,2}}$, $J^2_y=\Id_{\mathbb R^{2,2}}$. We also obtain that $\widetilde J_{\tilde z_{i}}$ and $\widetilde J_{\tilde \zeta_{j}}$ mutually anti-commute for all
$i=1,\ldots,r+1$ and
for $j=1,\ldots,s+1$. Now we verify as in~\eqref{al:skew}
that the scalar product $\langle\cdot\,,\cdot\rangle_{\widetilde V}=\langle\cdot\,,\cdot\rangle_V\cdot\langle\cdot\,,\cdot\rangle_{\mathbb R^{2,2}}$ satisfies the skew symmetry property. As in the previous case we show that if both modules $(V,\langle\cdot\,,\cdot\rangle_V)$, $(\mathbb R^{2,2},\langle\cdot\,,\cdot\rangle_{\mathbb R^{2,2}})$ are integral, then the resulting module is integral. For this we can choose the map $F$ with the matrix having the entries $a=1$ and $b=c=d=0$.
\end{proof}

\begin{remark}\label{rem:F}
Observe that $J_xJ_y\neq \pm F$ since if we choose $d=\pm 1$, then $a$ or $b$ must be different from zero.
\end{remark}


\section{Integral $\Cl_{r,s}$-modules with $r+s\geq 9$}\label{sec:g9}



\subsection{Integral structure on $\Cl_{r,s}$-modules of dimension $r+s= 9$}

  
Since 
$$
\Cl_{9,0}\cong\Cl_{1,0}\otimes\Cl_{8,0},\qquad
\Cl_{8,1}\cong\Cl_{0,1}\otimes\Cl_{8,0}
$$   
we apply Theorem~\ref{th:r8s}. For the cases
$$
\Cl_{0,9}\cong\Cl_{0,1}\otimes\Cl_{0,8},\qquad
\Cl_{1,8}\cong\Cl_{1,0}\otimes\Cl_{0,8}
$$  we apply Theorem~\ref{th:rs8}. We use Theorem~\ref{th:r4s4} and get the integral structure due to the isomorpfisms
$$
\Cl_{5,4}\cong\Cl_{1,0}\otimes\Cl_{4,4},\qquad
\Cl_{4,5}\cong\Cl_{0,1}\otimes\Cl_{4,4}.
$$
For $\Cl_{6,3}$- and $\Cl_{2,7}$-modules we exploit the isomorphisms 
$$
\Cl_{6,3}\cong\Cl_{5,2}\otimes\Cl_{1,1},\qquad
\Cl_{2,7}\cong\Cl_{1,6}\otimes\Cl_{1,1}
$$
and Theorem~\ref{th:higher2}. In this case, counting dimensions, we conclude that $\Cl_{6,3}$ and $\Cl_{2,7}$ are of minimal dimensions. Note that we could only construct the admissible integral module $\Cl_{2,7}\cong\Cl_{1,6}\otimes\Cl_{1,1}$ and then use Theorem~\ref{prop:7} and Corollary\ref{cor:3} to justify the existence of an integral structure on $\Cl_{6,3}$ module, by the isomorphism $\Cl_{6,3}\cong\Cl_{2,7}$.
We also apply Theorem~\ref{prop:7} and Corollary~\ref{cor:3} to the isomorphic Clifford algebras
$$
\Cl_{s,r+1}=\Cl_{5,4}\cong\Cl_{3,6}=\Cl_{r,s+1}\quad\text{and}\quad
\Cl_{s,r+1}=\Cl_{1,8}\cong\Cl_{7,2}=\Cl_{r,s+1}
$$
and pullback the integral admissible structure of $\Cl_{s,r+1}$-module to the module of $\Cl_{r,s+1}$.

\subsection{Integral structure on $\Cl_{r,s}$-modules of dimension $r+s= 10$}

  
The integral admissible module for Clifford algebras $\Cl_{10,0}$, $\Cl_{9,1}$, and $\Cl_{8,2}$ we construct by applying Theorem~\ref{th:r8s}. For the cases of
$\Cl_{2,8}$, $\Cl_{1,9}$, and $\Cl_{0,10}$
we use Theorem~\ref{th:rs8}. We use Theorem~\ref{th:r4s4} and get the integral structure for modules of $\Cl_{4,6}$, $\Cl_{5,5}$, and $\Cl_{6,4}$. 
We also apply Theorem~\ref{prop:7} to the isomorphic Clifford algebras
$$
\Cl_{s,r+1}=\Cl_{6,4}\cong\Cl_{3,7}=\Cl_{r,s+1}\quad\text{and}\quad
\Cl_{s,r+1}=\Cl_{2,8}\cong\Cl_{7,3}=\Cl_{r,s+1}
$$
and pullback the integral admissible structure of $\Cl_{s,r+1}$-module to the module of $\Cl_{r,s+1}$.


\subsection{Integral structure on $\Cl_{r,s}$-modules of dimension $r+s>10$}


For the rest of cases we use Theorems~\ref{th:rs8}, ~\ref{th:r8s}, and~\ref{th:r4s4}. 



\section{Final remarks}\label{fremark}


(1) The constructed admissible integral modules show that the corresponding general $H$-type Lie algebras admit integer structural constants. The natural question arises: how many different Lie algebras are behind the general $H$-type Lie algebras if we discard the presence of the scalar product? Having in hand the integral basis, it is easier to answer this question.
As we noticed in Remark~\ref{three dim Heisenberg}
both of the Lie algebras based on spaces
$\mathbb{R}^{1,1}\oplus_{\perp}\mathbb{R}^{0,1}$
and
$\mathbb{R}^{2,0}\oplus_{\perp}\mathbb{R}^{1,0}$, corresponding to $\Cl_{1,0}$- and $\Cl_{0,1}$-modules,
are isomorphic to the three dimensional Heisenberg algebra,
although the metrics are different.
The 6-dimensional Lie algebras based on vector spaces
$\mathbb{R}^{4,0}\oplus_{\perp}\mathbb{R}^{2,0}$,
$\mathbb{R}^{2,2}\oplus_{\perp}\mathbb{R}^{0,2}$, related to $\Cl_{2,0}$- and $\Cl_{0,2}$-modules,
are also isomorphic, but not isomorphic to the Lie algebra based on the space
$\mathbb{R}^{2,2}\oplus_{\perp}\mathbb{R}^{1,1}$, arising from
$\Cl_{1,1}$-module. 
This can be proved by making use of the above 
constructed integral basis. The details including these and more
general relations between
$\mathbb{R}^{k,k}\oplus_{\perp}\mathbb{R}^{r,s}$ 
is treated in a forthcoming paper, see also~\cite{Mag} and~\cite{Se} for the
classification of low dimensional nilpotent Lie algebras.

(2) Note that for construction of integral structures on admissible
$\Cl_{r,1}$-modules, we could used the isomorphism
$\Cl_{0,r+1}\cong\Cl_{r,1}$, but not Theorem~\ref{prop:r0tor1} 
and get the integral basis from ones constructed in 
Section~\ref{sec:0n}. Or vice versa, we could get all the integral
modules $\Cl_{0,r+1}$ from Section~\ref{sec:0n} by making use
Theorem~\ref{prop:7} and Corollary~\ref{cor:3} 
from modules $\Cl_{r,1}$, where we first construct $\Cl_{r,1}$-module 
by making use of Theorem~\ref{prop:r0tor1}.

(3) In Section~\ref{sec:rs} we constructed admissible integral
$\Cl_{r,0}$-modules. It is known classical case that admissible
modules for the Clifford algebras $\Cl_{r,0}$ have a positive definite
product. We based on this knowledge when we found integral structures
for $\Cl_{r,0}$-modules with $r=1,2,3,4,8$. 
For the rest of the cases we started the construction from the
assumption that the module is admissible with a neutral scalar product
and then the construction 
of integral bases leads to the sub-module with positive definite
product. This shows that our method not only gives the integral basis
but also detects the possible 
signature of the scalar product.

(4) The integral admissible $\Cl_{0,s}$-modules, constructed in
Section~\ref{sec:0n} 
could be also found by using Theorem~\ref{th:higher1}. In this case
the resulting 
$\Cl_{0,3}$- and $\Cl_{0,5}$-modules would be of minimal dimensions, 
but $\Cl_{0,s}$-modules for $r=2,4,6,7,8$ would exceed the minimal 
dimension twice. That was also a reason why we constructed all $\Cl_{0,s}$-modules directly. 


\section{Appendix}\label{appendix}


\begin{landscape}
{\tiny
$$
\hskip-1.1cm
\begin{array}{|c|cccccccccccccccccc||}
\hline
&&&&&&&&\leftarrow&r-s&\rightarrow&&&&&&&&
\\
\hline
r+s &-8&-7& -6&-5&-4&-3&-2&-1&0&1&2&3&4&5&6&7&8&
\\
\hline
0&&&&&&&&&\mathbb R&&&&&&&&&
\\
1&&&&&&&&\circled{$\mathbb R^2$}&&\mathbb C&&&&&&&&
\\
2&&&&&&&\circled{$\mathbb R(2)$}&&\circled{$\mathbb R(2)$}&&\mathbb H&&&&&&&
\\
3&&&&&&\circled{$\mathbb C(2)$}&&\circled{$\mathbb R^2(2)$}&&\circled{$\mathbb C(2)$}&&\mathbb H^2&&&&&&
\\
4&&&&&\mathbb H(2)&&\circled{$\mathbb R(4)$}&&\circled{$\mathbb R(4)$}&&\mathbb H(2)&&\mathbb H(2)&&&&&
\\
5&&&&\circled{$\mathbb H^2(2)$}&&\mathbb C(4)&&\circled{$\mathbb R^2(4)$}&&\mathbb C(4)&&\circled{$\mathbb H^2(2)$}&&\mathbb C(4)&&&&
\\
6&&&\mathbb H(4)&&\mathbb H(4)&&\mathbb R(8)&&\mathbb R(8)&&\mathbb H(4)&&\mathbb H(4)&&\mathbb R(8)&&&
\\
&&&&&&&&&&&&&&&&&&
\\
7&&\mathbb C(8)&&\mathbb H^2(4)&&\mathbb C(8)&&\mathbb R^2(8)&&\mathbb C(8)&&\mathbb H^2(4)&&\mathbb C(8)&&\mathbb R^2(8)&&
\\
&&&&&&&&&&&&&&&&&&
\\
8&\mathbb R(16)&&\mathbb H(8)&&\mathbb H(8)&&\mathbb R(16)&&\mathbb R(16)&&\mathbb H(8)&&\mathbb H(8)&&\mathbb R(16)&&\mathbb R(16)&
\\
9&&\mathbb C(16)&&\circled{$\mathbb H^2(8)$}&&\mathbb C(16)&&\circled{$\mathbb R^2(16)$}&&\mathbb C(16)&&\circled{$\mathbb H^2(8)$}&&\mathbb C(16)&&\circled{$\mathbb R^2(16)$}&&
\\
10&\circled{$\mathbb R(32)$}&&\mathbb H(16)&&\mathbb H(16)&&\circled{$\mathbb R(32)$}&&\circled{$\mathbb R(32)$}&&\mathbb H(16)&&\mathbb H(16)&&\circled{$\mathbb R(32)$}&&\circled{$\mathbb R(32)$}&
\\
11&&\circled{$\mathbb C(32)$}&&\mathbb H^2(16)&&\circled{$\mathbb C(32)$}&&\circled{$\mathbb R^2(32)$}&&\circled{$\mathbb C(32)$}&&\mathbb H^2(16)&&\circled{$\mathbb C(32)$}&&\circled{$\mathbb R^2(32)$}&&
\\
\hline
\end{array}
$$
}
\end{landscape}


\end{document}